\documentclass[a4paper,11pt,reqno]{amsart}
\usepackage{mathabx}
\usepackage{lmodern}
\usepackage[utf8]{inputenc}
\usepackage[left=1.7cm,right=1.7cm,bottom=2.7cm,top=2.7cm]{geometry}
\usepackage[all,cmtip]{xy}
\usepackage{amsmath}
\usepackage{amsthm}
\usepackage{subcaption}
\usepackage{bbm} 
\usepackage{bm} 

\usepackage{amssymb}
\usepackage[utf8]{inputenc}
\usepackage[T1]{fontenc}

\AtBeginDocument{%
  \DeclareFontShape{T1}{lmr}{m}{scit}{<->ssub*lmr/m/scsl}{}%
}

\usepackage{enumerate}
\usepackage{graphicx}
\usepackage{mathrsfs}
\usepackage{tikz}
\usetikzlibrary{matrix,arrows}
\usepackage{xspace} 

\theoremstyle{plain}
\newtheorem{theorem}{Theorem}[section]
\newtheorem*{theorem*}{Theorem}
\newtheorem{proposition}[theorem]{Proposition}

\newtheorem{corollary}[theorem]{Corollary}

\newtheorem{example}[theorem]{Example}

\newtheorem{lemma}[theorem]{Lemma}
\newtheorem{problem}[theorem]{Problem}

\theoremstyle{definition}
\newtheorem{definition}[theorem]{Definition}

\theoremstyle{remark}

\newtheorem{claim}{Claim}[theorem]

\newtheorem{ejemplo}{{\sc Example}}
\newtheorem{notas}[theorem]{{\sc Remark}}
\newcommand{\nrm}[1]{\|#1\|}

\DeclareMathOperator{\osc}{\mathrm{Osc}}
\DeclareMathOperator{\eqp}{\mathrm{Equi}}

\newcommand{\prop}{\begin{proposition}}
\newcommand{\fprop}{\end{proposition}}
\newcommand{\cor}{\begin{corollary}}
\newcommand{\fcor}{\end{corollary}}

\newcommand{\defi}{\begin{definition}\rm}
\newcommand{\fdefi}{\end{definition}}
\newcommand{\eje}{\begin{ejemplo}}
\newcommand{\feje}{\end{ejemplo}}
\newcommand{\ejes}{\begin{ejemplos}}
\newcommand{\fejes}{\end{ejemplos}}
\newcommand{\lema}{\begin{lemma}}
\newcommand{\flema}{\end{lemma}}
\newcommand{\teor}{\begin{theorem}}
\newcommand{\fteor}{\end{theorem}}
\newcommand{\nota}{\begin{notas}\rm}
\newcommand{\fnota}{ \end{notas}}
\newcommand{\clam}{\begin{claim}}
\newcommand{\fclam}{\end{claim}}
\newcommand{\clams}{\begin{claim*}}
\newcommand{\fclams}{\end{claim*}}

\newcommand{\lclam}{\begin{lclaim}}
\newcommand{\flclam}{\end{lclaim}}
\newcommand{\prucl}{\prue[Proof of Claim:]}
\newcommand{\fprucl}{\fprue}
\newcommand{\ben}{\begin{enumerate}}
\newcommand{\een}{\end{enumerate}}
\newcommand{\bit}{\begin{itemize}}

\newcommand{\eit}{\end{itemize}}

\newcommand{\bmc}[1]{\overline{#1}^\mr{BM} }

\newcommand{\mc}[1]{\mathcal{#1}}
\newcommand{\mbb}[1]{\mathbb{#1}}
\newcommand{\mr}[1]{\mathrm{#1}}

\newcommand{\mk}[1]{\mathfrak{#1}}
\newcommand{\mbf}[1]{\mathbf{#1}}

\DeclareMathOperator{\epi}{\mr{Epi}}
\DeclareMathOperator{\aut}{\mr{Aut}}

\newcommand{\acts}{\curvearrowright}

\newcommand{\casos}{\begin{itemize}}
\newcommand{\fcasos}{\end{itemize}\setcounter{cs}{1}}

\newcommand{\ro}{\varrho}

\DeclareMathOperator{\quo}{/\hspace{-0.1cm}/}
\DeclareMathOperator{\iso}{\mathrm{Iso}}
\DeclareMathOperator{\Iso}{\mathrm{Iso}}

\DeclareMathOperator{\flim}{\mathrm{F}\!\lim}
\DeclareMathOperator{\age}{\mathrm{Age}}

\newcommand{\conj}[2]{ \{ {#1}\,:\,{#2} \} }

\newcommand{\om}{\omega}
\newcommand{\Om}{\Omega}

\newcommand{\hp}{$\mathrm{(HP)}$\xspace}
\newcommand{\jep}{$\mathrm{(JEP)}$\xspace}
\newcommand{\nap}{$\mathrm{(NAP)}$\xspace}
\newcommand{\ap}{$\mathrm{(AP)}$\xspace}
\newcommand{\napp}{weak amalgamation\xspace}

\newcommand{\arp}{$\mathrm{(ARP) }$\xspace}
\newcommand{\arpp}{$\mathrm{(ARP}^+\mathrm{)}$\xspace}
\newcommand{\sarpp}{$\mathrm{(SARP}^+\mathrm{)}$\xspace}
\newcommand{\auhp}{weak Fraïssé\xspace}

\newcommand{\sauhp}{Fraïssé\xspace}
\newcommand{\auh}{$\mathrm{(AuH)}$\xspace}
\newcommand{\latuh}{$\mathrm{(}\diamond\mr{-uH)}$\xspace}
\newcommand{\latauh}{$\mathrm{(}\diamond\mr{-AuH)}$\xspace}

\newcommand{\uh}{$\mathrm{(uH)}$\xspace}
\newcommand{\fh}[1]{$\mathrm{(}#1\mathrm{-H)}$\xspace}
\newcommand{\afh}[1]{$\mathrm{(A}#1\mathrm{H)}$\xspace}

\newcommand{\buit}{\emptyset}

\newcommand{\ga}{\gamma}
\DeclareMathOperator{\id}{\text{Id}}

\newcommand{\al}{\alpha}

\newcommand{\de}{\delta}
\newcommand{\De}{\Delta}
\newcommand{\la}{\lambda}

\newcommand{\La}{\Lambda}
\newcommand{\sig}{\sigma}
\newcommand{\Sig}{\Sigma}
\newcommand{\vphi}{\varphi}

\DeclareMathOperator{\emb}{\mathrm{emb}}

\DeclareMathOperator{\Emb}{\mathrm{Emb}}

\newcommand{\vep}{\varepsilon}

\newcommand{\eqs}{{\mathfrak X}}
\newcommand{\R}{{\mathbb R}}
\newcommand{\F}{{\mathbb F}}

\newcommand{\N}{{\mathbb N}}

\newcommand{\Q}{{\mathbb Q}}

\newcommand{\C}{{\mathbb C}}

\newcommand{\rest}{\upharpoonright}
\newcommand{\im}{\mathrm{Im}}
\newcommand{\diam}{\text{diam}}

\newcommand{\supp}{\mathrm{supp\, }}

\newcommand{\con}{\subseteq}

\newcommand{\prue}{\begin{proof}}
\newcommand{\fprue}{\end{proof}}

\makeindex
\title{Amalgamation and Ramsey properties of $L_p$ spaces}

\author[V. Ferenczi]{V. Ferenczi}
\address{Departamento de Matem\'atica, Instituto de Matem\'atica e Estat\'istica,
Universidade de S\~ao Paulo, 
rua do Mat\~ao, 1010, 
05508-090 S\~ao Paulo, SP, Brazil}
\email{ferenczi@ime.usp.br}

\author[J. L{{o}}pez-Abad]{J. L{{o}}pez-Abad}
\address{Departamento de Matem\'{a}ticas Fundamentales,
Facultad de Ciencias, UNED, 28040 Madrid, Spain}
\email{abad@mat.uned.es}

\author[B. Mbombo]{B. Mbombo}
\address{Departamento de Matem\'atica, Instituto de Matem\'atica e Estat\'istica,
Universidade de S\~ao Paulo, 
rua do Mat\~ao, 1010, 
05508-090 S\~ao Paulo, SP, Brazil}
\email{bricerodriguembombo@cegepjonquiere.ca}
\author{S. Todorcevic}
\address{Institut de Math\'{e}matiques de Jussieu, UMR 7586, 2 place Jussieu - Case 247, 75222 Paris Cedex 05, France, and Department of Mathematics, University of Toronto, Toronto, Canada, M5S 2E4}
\email{stevo@math.toronto.edu}

\thanks{V. Ferenczi, J. Lopez-Abad and B. Mbombo  were supported by Fapesp, projects 2012/20084-1, 2013/11390-4, 2013/24827-1 and 2016/25574-8. V. Ferenczi was supported by CNPq, projects 303034/2015-7 and 303721/2019-2. V. Ferenczi and S. Todorcevic were supported by USP Cofecub project number 2013-7/31466UC. J. Lopez-Abad   was partially supported 
    by the Ministerio de Econom\'\i a y Competitividad under grant MTM2016-76808-P.  S.Todorcevic was also supported by grants from NSERC(455916)   and CNRS(UMR7586).    This work was initiated during a visit of
J.L.-A.\ to the Universidade de Sao P\~{a}ulo in 2014, and continued during
visits of  J.L.-A. to the Fields Institute in the Fall 2014,  
a visit of all the authors at the Banff International Research Station in
occasion of the Workshop on Homogeneous Structures in the Fall 2015, a visit of J. L-A to the University of Toronto,  and a visit of V. F. to UNED in the winter of 2018. The
hospitality of all these institutions is gratefully acknowledged. }


\begin{document}
\begin{abstract}
We study  the dynamics of the group of isometries of $L_p$-spaces. In particular, we study the canonical actions of these groups on the space of $\de$-isometric embeddings of finite dimensional subspaces of $L_p(0,1)$ into itself, and we show that for every real number $1\le p<\infty$ with  $p \neq 4,6,8,\ldots$ they are $\vep$-transitive provided that $\de$ is small enough.  We achieve this by extending the classical equimeasurability principle  of  Plotkin and  Rudin. We define the central notion of a Fra\"iss\'e Banach space which underlies these results and of which the known separable examples are the spaces $L_p(0,1)$, $p \neq 4,6,8,\ldots$ and the Gurarij space.   We also give a proof of the  Ramsey property of the classes $\{\ell_p^n\}_n$, $p\neq 2,\infty$, viewing it as a multidimensional Borsuk-Ulam statement. We relate this to an arithmetic version of the Dual Ramsey Theorem of Graham and  Rothschild as well as to the notion of a spreading vector of   Matou\v{s}ek and  R\"{o}dl. Finally, we give a version of the Kechris-Pestov-Todorcevic correspondence that links the dynamics of the group of isometries of an approximately ultrahomogeneous space $X$ with a Ramsey property of the collection of finite dimensional subspaces of $X$.   
\end{abstract} 
\maketitle

\section{Introduction}
It is a classical result of  A.  Pełczyński and S. Rolewicz   \cite{PelRol}  that the spaces $L_p(0,1)$ are  almost transitive, in the sense that the   group of linear isometric surjections $\iso(L_p(0,1))$  acts almost transitively on the corresponding unit sphere of $L_p(0,1)$.  This was later extended by  W. Lusky  \cite{luss} who proved that in fact, the group $\iso(L_p(0,1))$ for a real  number $1\le p<\infty$ also acts almost transitively on each metric space $\Emb(X, L_p(0,1))$ of linear isometric embeddings from a  finite dimensional subspace $X$ of $L_p(0,1)$ into $L_p(0,1)$,  but only provided that $p=2$ or $p\notin 2\N$.    Other Banach spaces having this property are any Hilbert space or the {\em Gurarij space}, and  recently quasi-Banach spaces with the corresponding property have been found in \cite{CaGaKu}.  
This almost ``ultra'' transitive' property is the metric analogue of the so-called  ultrahomogeneity  property of algebraic structures, the core of Fraïssé theory in model theory, and the proper context for the   combinatorial characterization of the {\em extreme  amenability} of the corresponding automorphism group, known as the {\em Kechris-Pestov-Todorcevic (KPT) correspondence} \cite{KPT}. Recall that a topological group is extremely amenable when each of its continuous actions on a compact space has a fixed point, and that the KPT correspondence states that for ultrahomogeneous structures $\mc M$  the extreme amenability of its automorphism group $\aut(\mc M)$   is exactly the {\em Ramsey property}  of its class of finitely generated substructures, called the {\em age} and denoted by $\age(\mc M)$.  By this means, many new examples of extremely amenable groups have been given (see \cite{KPT}).  The theory of abstract ultrahomogeneous    metric structures has been studied in \cite{BYBeHeUs}, while 
the  KPT correspondence has been recently extended to this context by J. Melleray and T. Tsankov \cite{MeTsa}.  The KPT correspondence  for metric structures was for the first time used in \cite{BaLALuMbo1}, \cite{BaLALuMbo2},  in showing, for example, that the isometry group of the Gurarij space is extremely amenable by supplying  an  appropriate Ramsey  type result which relies on the Graham-Rothschild Theorem. We note that the Gurarij space shares this this property with the infinite dimensional Hilbert spaces  and the spaces $L_p(0,1)$, proved by M. Gromov and V. D. Milman \cite{GrMi}, and  by T. Giordano and V. Pestov \cite{GiPe} respectively,  relying on the method of concentration of measure.   

It follows from the Banach-Lamperti description of isometries of $L_p$ spaces that the isometry groups of    the spaces   $L_p(0,1)$ and $L_q(0,1)$, when $1\le p,q \neq 2 <\infty$,  are  topologically isomorphic.  However, there are canonical actions of the same nature that have very different properties, depending on $p$:  while for $p\notin 2\N$ all the canonical actions by composition $\iso(L_p(0,1))\acts \Emb(X,L_p(0,1))$ are almost transitive, it follows from a work of B. Randrianantoanina \cite{beata} based on an early result of H. P. Rosenthal \cite{Ro}, that   there are finite dimensional subspaces $X$ of $L_{2n}(0,1)$, $n\in \N$, $n>1$,  for which that action is far of being almost transitive, because   $X$ has well complemented and badly complemented copies on $L_{2n} $ (see Proposition \ref{li3rjiwjreer3434t5}).   One of the main goals of this paper is to study the canonical actions of isometry groups of the Lebesgue spaces, not only on the spaces of  isometric embeddings, but also on $\Emb_\de(X,L_p(0,1))$ the class of $\de$-isometric embeddings from $X$ into $L_p(0,1)$.  
While $\de$-isometric embeddings were already considered by M. Lupini \cite{Lup}, in a general theory of stability including, for example, operator spaces and systems, in this paper one of our objectives is to obtain finer results based on weaker (and/or more precise) properties of homogeneity for structures, in such a way that $L_p$ spaces are included in the classes we consider. With these examples in mind, we concentrate on the case of the Banach spaces, and develop a theory which may be specific to the Banach space setting. In particular, although our results should be extendable to the quasi-Banach setting and the case of $L_p$ spaces for $0<p<1$, we shall not consider that situation.

We say that a Banach space $E$ is  {\em Fraïssé} when for every dimension $k$ and $\vep>0$ there is $\de>0$ such that the canonical action by composition $\iso(E)\acts \Emb_\de(X,E)$ is $\vep$-transitive for every $X\in \age_k(E)$, the collection of $k$-dimensional subspaces of $E$. 
We will see that $L_p(0,1)$ is Fraïssé, provided that $p$ is not even.   Other Fraïssé spaces are any Hilbert space and the Gurarij space. 
The interest of these $\vep-\de$ continuity properties can be appreciated by the following:  First of all, they imply  that  for Fraïssé spaces $E$ the Banach-Mazur and a restricted version of the {\em Kadets} pseudometrics are uniformly equivalent  on $\age(E)$; secondly, there is a characterization of the Fraïssé property for Banach spaces by passing to the ultrapower and involving homogeneity on isometric (instead of $\de$-isometric) embeddings in the ultrapower; thirdly, Fraïssé spaces are isometrically determined by the collection of their finite dimensional subspaces and there is a Fraïssé correspondence; finally, the spaces who are finitely representable  on a Fraïssé space $E$ can be isometrically embedded into $E$, and, consequently, the Hilbert space $\ell_2$ is the minimal Fraïssé space and $\mbb G$ is the unique Fraïssé space with trivial cotype.  

The Gurarij space $\mbb G$ is, by definition, an abstract inductive limit of $\ell_\infty^n$'s.  The proof of the Fraïssé property of $\mbb G$ follows from a combination of the existence of  general {\em pushouts} of finite dimensional spaces, and the fact that $\de$-isometric embeddings are in some precise sense $2\de$-close to isometric embeddings.    For  the class $\age(L_p(0,1))$ there is not known full pushout; instead,  for $p\notin 2\N$ there is a restricted version stating that for every $k\in \N$ and $\vep>0$ there is $\de>0$ such that  if  $X,Y,Z\in \age(L_p(0,1))$  with $\dim X=k$, and $\ga\in \Emb_\de(X,Y)$ and $\eta\in \Emb_\de(X,Z)$, then there are $V\in \age(L_p(0,1))$ and isometric embeddings $i: Y\to V$ and $j:Z\to V$ such that 
$\nrm{i\circ \ga -j \circ \eta}\le \vep$.   This is exactly, by means of the  Fraïssé correspondence (see Corollary \ref{ijoiriower3278766755}), the Fraïssé property of $L_p(0,1)$.  We prove this by establishing the {\em approximate equimeasurability principle}, the continuous statement extending the classical equimeasurability principle of  Plotkin and  Rudin: 
Suppose that $\mu,\nu$ are Borel measures on $\R^n$ for which the coordinate  functions  $x_j$ are $p$-integrable.  Then $\widehat{\mu}^{(p)}(a)=\int |1+ \sum_{j<n} a_j x_j|^p d\mu(x)=\int |1+ \sum_{j<n} a_j x_j|^p d\nu(x)=\widehat{\nu}^{(p)}(a)$ for all $a=(a_j)_{j<n}\in \R^n$  then the measures $\mu$ and $\nu$ are equal. We prove that $\mu$
  and $\nu$ are close, for example with respect to the Lévy-Prohorov metric,  when the corresponding characteristics are close, and moreover  we obtain a  full characterization.

We  study  the approximate Ramsey property \arp of the classes $\{\ell_p^n\}_n$, that also can be seen as a version of a multidimensional  Borsuk-Ulam Theorem (see  \S\S\S\ref{Borsuk-Ulam}). In general,   a class of finite dimensional    $\mc G$  has the \arp   when for every $X,Y\in \mc G$ and every $\vep>0$ there is $Z\in \mc G$ such that for every 1-Lipschitz mapping $c:\Emb(X,Z)\to [0,1]$ there is some isometric embedding $\ga:Y\to Z$ such that the oscillation of $c$ in the set of compositions $\ga\circ \Emb(X,Y)$ is at most $\vep$.  It is interesting to mention that, while the proof of the \arp of $\{\ell_\infty^n\}_n$ uses  the dual Ramsey Theorem (DR) of   Graham and  Rothschild  (see \cite{BaLALuMbo1} ), our proof of the \arp of  $\{\ell_p^n\}_n$ $p\neq 2,\infty$    utilize an  arithmetical version of (DR), namely, that for partitions of equal sized pieces.    

We also analyze restricted versions of the previous notions: Given a class $\mc G$ of finite dimensional Banach spaces, we introduce what we call $\mc G$-Fraïssé spaces, those for which the natural actions on $\de$-embeddings are $\vep$-transitive, provided that the embeddings have as domain an element of $\mc G$.  In this way, every $L_p(0,1)$, being $p$ even or not, is the {\em Fraïssé limit} of $\{\ell_p^n\}_n$.   We also restrict the type of embeddings we are interested in, for example by analyzing  Fraïssé lattices, where now isometries and embeddings must respect the lattice structure.   We find the first Fraïssé Banach lattice, an $M$-space, denoted by $\mbb G^\diamond$, that is the lattice version of the Gurarij space, and that has an extremely amenable group of lattice isometries, proved using a KPT correspondence for Banach lattices. 

The paper is organized as follows. In \S \ref{fraisse_banach_spaces} we introduce and study Fraïssé Banach spaces, as well as the local versions of them, meaning that the canonical actions of the rotations are restricted to embeddings defined on spaces a fixed family. For those spaces, we see the uniform equivalence of the Banach-Mazur and the Kadets pseudometrics.  In \S\S\ref{fraisse_correspondence_sda}  we prove the Fraïssé correspondence for Banach spaces, including its local version, and in \S\S\ref{characterization_of_homogeneities} we characterize the Fraïssé property of a Banach space in terms of a uniform equivalence of metrics.  In \S\ref{oi4iuuuuttyy} we introduce the lattice versions, including the proof of the fact that every $L_p(0,1)$  is lattice Fraïssé, a fact that follows from an   approximation result by G. Schechtman \cite{Sch} on $\de$-isometric embeddings defined on $\ell_p^n$.    
Section \ref{oi43894589754789} is devoted to the proof of the approximate equimeasurability principle for $p\notin 2\N$,  in \S\S\ref{auh_for_L_p} we see  how this is used to show that those $L_p(0,1)$ are Fraïssé.  The proof the principle is given in \S\S\ref{approx_equim}.   Section \ref{kemrkwmekorwew} is devoted to the approximate Ramsey property, in particular of the class $\{\ell_p^n\}_n$, and its reformulation à la Borsuk-Ulam is given in \S\S\S \ref{Borsuk-Ulam}, while its proof and the relation with an approximate Ramsey statement for equisurjections is the content of \S\S\ref{arp_of_ell_p_equisurjections}.  The last Section \ref{gurarij_m_space} is dedicated to the existence of a Fraïssé $M$-space whose group of lattice isometries is extremely amenable.

\section{ Fraïssé Banach spaces}\label{fraisse_banach_spaces}
We consider spaces over $\mbb F=\R,\C$;  given $n\in \N$, the {\em unit basis of} $\mbb F^n$, denoted by $(u_j^{(n)})_{j<n}$, or simply $(u_j)_{j<n}$, is the sequence that for each $j$, the $k^\mr{th}$ coordinate of $u_j^{(n)}$ is  delta of Dirac $\de_{j,k}$.   When needed,  will use the set theoretical  convention of identifying   an integer $n$ with  $\{0,1,\dots, n-1\}$. 
Given   Banach spaces $X$ and $Y$ and $\de \ge 0$, 
a {\em $\de$-isometric embedding (or $\de$-isometry)} $T: X \rightarrow Y$ is a linear map such that for all $x \in X$ one has that
    $$\frac1{1+\de}\nrm{x} \leq \nrm{Tx} \leq  (1+\de) \|x\|.$$
  When $\de=0$ we will simply use isometric embeddings   to refer to $0$-isometric embeddings. 
Let $\Emb_\de(X,Y)$ be the collection of $\de$-isometric embeddings between $X$ and $Y$, and let $\mr{Iso}(X)$ be the group of isometries on $X$.  
Given two families $\mc H$ and $\mc G$        of finite dimensional spaces we write $\mc H\preceq \mc G$ when for every $X\in\mc H$ there is $Y\in \mc G$ that is isometric to $X$, denoted by $X\equiv Y$; then $\mc H \equiv \mc G$   denotes that $\mc H\preceq\mc G\preceq \mc H$, and $\mc H_\equiv$ is the class of all finite dimensional spaces with an isometric copy in $\mc H$.

\defi
Given a Banach space $E$, let $\mr{Age}(E)$ be the class of all finite dimensional subspaces  of  $E$.  Following standard convention (see for example \cite[pp 324]{Ho}) we will say that  $X$ and $Y$ have the {\em same age} when $\age(X)\equiv \age(Y)$. Given a family $\mc H$ of finite dimensional spaces, let $\mc H_n$ be the subfamily of $\mc H$ consisting of those spaces of dimension $n$. In particular, we write $\mr{Age}_n(X)$ to denote $(\mr{Age}(X))_n$.  Given a class of finite dimensional Banach spaces $\mc H$ and a Banach space $E$, let $\mc H_E$  be the collection of subspaces of $E$ isometric to some element in $\mc H$.

Recall the {\em gap or opening} metric on $\mr{Age}_n(E)$ is defined by
 $$\La_E(X,Y):=\max\left\{\max_{x\in B_X} \min_{y\in B_Y}\nrm{x-y}_E,\max_{y\in B_Y} \min_{x\in B_X}\nrm{x-y}_E \right\};$$
 in other words, $\La_E(X,Y)$ is the $\nrm{\cdot}_E$-Hausdorff distance between the unit balls of $X$ and $Y$.  This induces the   following {\em Gromov-Hausdorff} function, 
{\em $E$-Kadets}    on  $\mr{Age}_n(E)^2$, defined  as 
$$\ga_E(X,Y):=\inf \conj{\La_E(X_0, Y_0)}{ X_0,Y_0\in \mr{Age}_n(E), \, X_0\equiv X, \, Y_0\equiv Y}.$$
When $E$ is universal $\ga_E$ is the original  {\em Kadets pseudometric} (see \cite{Kad}, \cite{KaOs}), although in general   $\ga_E$ may not be a pseudometric. We will see that in other natural cases, $\ga_E$ satisfies the triangle inequality. It is easy to see that $\ga_E(X,Y)=0$ if and only if $X$ and $Y$ are isometric.   
There is another well-known pseudometric with this property. This is the   {\em Banach-Mazur} pseudometric on $\mr{Age}_n(E)$: 
$$d_\mr{BM}(X,Y):=\log(\inf_{T:X\to Y} \nrm{T}\cdot \nrm{T^{-1}})$$
where the infimum runs over all isomorphisms $T:X\to Y$.  It is well-known that   $(\mr{Age}_n(E),d_\mr{BM})$ is a pre-compact pseudometric space; that is, every sequence in $\mr{Age}_n(E)$ has a $d_\mr{BM}$-convergent subsequence, not necessarily to an element of $\mr{Age}_n(E)$.

\fdefi

 \begin{definition}\label{nnvnjkff}
 Let $E$ be an infinite dimensional  Banach space, and let $\mc G\preceq \age(E)$. 
 \begin{enumerate}[(a)]
\item $E$ is  {\em $\mc G$-homogeneous} \fh{\mc G} when for every $X\in \mc G$  and every   and every $\ga,\eta\in \Emb(X,E)$ there is some $g\in \iso(E)$ such that $g\circ\ga= \eta$; in other words, when for each $X\in \mc G$, the natural action  $\iso(E)\acts \Emb(X,E)$ by composition is transitive.
\item $E$ is   is called \emph{approximately  $\mc G$-homogeneous} \afh{\mc G}  when for every  
 $X\in \mc G$ and every $\vep>0$ the natural action by composition $\iso(E)\acts \Emb(X,E)$ is $\vep$-transitive, that is,   whenever   $\ga,\eta\in \Emb(X,E)$  there is   $g\in \iso(E)$ such that $\nrm{g\circ\ga-\eta}<\vep$.   
\item $E$ is   is called \emph{weak $\mc G$-Fraïssé}  when for every $X\in \mc G$ and every $\vep>0$ there is a $\de>0$ such that the action $\iso(E)\acts \Emb_\de(X,E)$ is $\vep$-transitive.

 \item $E$ is {\em $\mc G$-Fraïssé}   when for every dimension $k\in \N$ and every $\vep>0$ there is a $\de>0$ such that the action $\iso(E)\acts \Emb_\de(X,E)$ is $\vep$-transitive for every  $X\in \mc G_k$.

 \end{enumerate}
 When $\mc G=\age(E)$, then we will use {\em ultrahomogeneous} \uh, {\em approximately ultrahomogeneous} \auh, {\em weak-Fraïssé} and {\em Fraïssé} for the corresponding $\mc G$-homogeneities.    \end{definition}
  It is worth to point out that the notion of approximate ultrahomogeneity defined here coincides with the general notion of approximate ultrahomogeneity when the Banach spaces are regarded as   metric structures (see \cite{BY}, \cite{BYBeHeUs}, \cite{MeTsa}). Also, the  particular case of Fraïssé property with modulus independent of the dimension was studied in \cite{Lup}  and it was named {\em stable} Fraïssé property.

 We will say that a mapping $\varpi:S\times [0,\infty[\to [0,\infty[$, where $S$ is an arbitrary set, is a {\em  modulus} when $\varpi(s,\cdot)$ is increasing and continuous at zero with value zero for every $s\in S$.
 The following is easy to prove.
 \begin{proposition}
\begin{enumerate}[1)]
\item $E$ is weak-Fraïssé if and only if there exists a modulus $\varpi_{\mc G,E}: \mc G\times [0,\infty[\to [0,\infty[$, called modulus of stability of $\mc G$ in $E$, such that 
for every  
 $X\in \mc G$ every $\de\ge 0$ and every $\vep>0$  one has that  $\iso(E)\acts \Emb_\de(X,E)$ is {\em $\varpi(X,\de)+ \vep$-transitive}.

 In this case   $\varpi_{\mc G,E}: \mc G\times [0,\infty[\to [0,\infty[$ defined    for $X\in \mc G$ and  $\de\ge 0$ as $\varpi_{X,E}(\de)=\inf\conj{\vep>0}{\iso(E)\acts \Emb_\de(X,E) \text{ is $\vep$-transitive}}$ is the optimal modulus of stability of $\mc G$ in $E$. 

  \item $E$ is $\mc G$-Fraïssé if and only if  $E$ is weak $\mc G$-Fraïssé  with a modulus that only depends on the dimension of $X\in \mc G$.  In this case, $\varpi_{\mc G,E}:\N\times [0,\infty[\to [0,\infty[$,   
  $$\varpi_{\mc G,E}(k,\de)=\inf\conj{\vep>0}{\iso(E)\acts \Emb_\de(X,E) \text{ is $\vep$-transitive for all $X\in \mc G_k$}}$$ is the optimal modulus of stability of $\mc G$ in $E$.
\item If $\varpi:S\times[0,\infty[\to [0,\infty[$ is a modulus, then 
$$\varpi^*(s,\de):=\inf_{\de'>\de}\varpi(s,\de')$$ is a continuous modulus such that $\varpi(s,\de)\le \varpi^*(s,\de)$.    \qed
\end{enumerate}  
   
 \end{proposition}

\begin{example}
A Hilbert space is  \uh and  also is a Fraïssé Banach space with modulus $\varpi(k,\de)=2\de$.
\end{example}
\begin{proof}
Suppose that $\mc H$ is a Hilbert space.  Clearly, a subspace of $\mc H$ is again Hilbert, so, the Banach-Mazur limit of subspaces of $\mc H$ is again (mod isometry) a subspace of $\mc H$.   Suppose now that  $\ga: F\to \mc H$ is a $\de$-isometry, with $F\con \mc H$ finite dimensional. Choose an orthonormal basis $(x_j)_{j<n}$ of $F$, and  such that $(\ga x_j)_{j<n}$ is an orthogonal sequence (the singular value decomposition). Let $\iota: F\to \mc H$ be the isometric linear embedding  defined by $\iota(x_j):= \ga(x_j)/\nrm{\ga(x_j)}$ for all $j<n$. Then given scalars $(a_j)_{j<n}$ we have that 
 \begin{align*}
 \nrm{\sum_j a_j \ga x_j - \sum_j a_j \iota x_j}^2=&\langle \sum_{j} a_j \left( 1-\frac{1}{\nrm{\ga x_j}}\right) \ga x_j, \sum_{j} a_j \left( 1-\frac{1}{\nrm{\ga x_j}}\right) \ga x_j\rangle =\\
 =&  \sum_{j} |a_j |^2 |\nrm{\ga x_j} -1|^2 \le \de^2 \nrm{\sum_j a_j x_j}^2.
  \end{align*} 
 So, if we extend $\iota$ to an isometry $I\in \mr{Iso}(\mc H)$, then $\nrm{\ga - I\rest F}\le \de$, as desired.   
\end{proof}  
\begin{example}
The Gurarij space $\mbb G$ \cite{Gu} is Fraïssé with modulus $\varpi(k,\de)=2\de$ but not \uh.\footnote{In \cite{BaLALuMbo2}, $\de$-embeddings are always considered to be contractions, and because of this the modulus $\varpi$ of Gurarij space can be proved to be $\de$}
\end{example}
\begin{proof}
 It was already known by Gurarij that $\mbb G$ is not even transitive: It is well known that an isometry moves a point of differentiability of the unit sphere to a point of differentiability. Since $\mbb G$ is universal, its unit sphere has points of differentiability  and points of non-differentiability.  On the other direction,  $\mbb G$ is Fraïssé with modulus  $\varpi(k,\de)=\de$  (see \cite[Corollary 2.16]{BaLALuMbo2}, \cite[\S\S 6.1]{Lup} ).
 \end{proof}

\begin{problem}
Are $\mbb G$ and the Hilbert space $\mc H$ the only separable stable Fraïssé Banach spaces?
\end{problem}
 
\begin{example}
For every $1\le p<\infty$ the space $L_p[0,1]$ is $\{\ell_p^n\}_n$-Fraïssé (see Proposition \ref{3l4kirjio34riji43}). In fact, $L_p[0,1]$ is the {\em Fraïssé limit} of $\{\ell_p^n\}_n$ (see  Theorem \ref{oi43hjiio434371}). 
\end{example}

\begin{example}\label{oiioio43447856}
 W. Lusky \cite{luss}, using the equimeasurability theorem of   A. I. Plotkin \cite{Plo1}  and W. Rudin \cite{Ru}, 
  proved that for $p\notin 2\N$ the space $L_p(0,1)$ is (AuH)\footnote{In that paper Lusky states that for every $X\in \age(E)$, $\vep>0$ and $\ga\in \Emb(X,L_p)$ there is a surjective $1+\vep$-isomorphism of $L_p(0,1)$ extending $\ga$, but its proof directly gives the \auh of $L_p(0,1)$ for those $p$'s}.  On the other hand, the isometry group never acts transitively on the unit sphere of $L_p(0,1)$ if $p\neq 2$. One of our main results  will be to show that in fact for $p\notin 2\N$ the space $L_p(0,1)$ is  Fraïssé (Theorem \ref{auh_for_L_p}).

\end{example}

\begin{problem}
Are $\mbb G$, $L_p(0,1)$, $p\neq 4,6,8,\dots$ the only separable Fraïssé Banach spaces?
\end{problem}	
It is well-known there are other almost transitive Banach spaces; for example $E:=L_p(X)$ for any almost transitive Banach space $X$.  In particular,  $L_p(\mbb G)$  could be \auh or even Fraïssé. However this is not so because there are well-complemented and not well complemented isometric copies of $\ell_p^n$,   hence the corresponding Bochner spaces cannot be \auh. A similar reasoning holds for $L_p(L_q)$, at least when $1 \leq p,q <2$ and $p \neq q$. Similarly, although much more complicated, we will see  in the  next Proposition \ref{li3rjiwjreer3434t5}  that the spaces $L_p(0,1)$ for $p=4,6,8,\dots$ cannot be \auh.  This fact was already proved by  Lusky  by using a counterexample of Rudin in \cite{Ru} exposing the non-equimeasurability theorem for  those $p$'s.    

\prop \label{li3rjiwjreer3434t5} Assume $p \in 2\N$, $p \geq 4$. For any $C \geq 1$ and $\de\ge 0$, there  are isometric $X,Y\in \age(L_p(0,1))$ such that for any  bounded linear mapping $T:L_p(0,1)\to L_p(0,1)$, if $T\rest X \in \Emb_\de(X,Y)$, then  $\nrm{T} \geq C$.
 \fprop

\prue   It is proved by B. Randrianantoanina \cite{beata}  that for $p\in 2\N$, $p>2$,  the uncomplemented subspace $Y_p$ of $L_p$ built by  H. Rosenthal in \cite{Ro} is isometric to a certain complemented subspace $Z_p$ of $L_p$ spanned by $3$-valued independent symmetric random variables. Since the space $Z_p$ is the span of a sequence of independent mean zero random variables,  it has an unconditional basis, \cite[Remark 2, page 278]{Ro}. Let $(y_n)$ (respectively $(z_n)$) be the associated basis of $Y_p$ (resp. $Z_p$), and
$Y_n=\langle y_j\rangle_{j<n}$,
 $Z_n=\langle z_j\rangle_{j<n}$.
On the one hand the $Z_n$'s are uniformly complemented in $Z_p$ and therefore in $L_p(0,1)$; on the other hand
the $Y_n$'s are not, otherwise by reflexivity and a weak limit argument,
$Y_p$ would be. In other words,  there are projections $Q_n$    onto $Z_n$ with uniform bound $c$, but $C_n:=\inf\conj{C}{ Y_n {\rm\ is\ } C {\rm -complemented\ in\ }
 L_p(0,1)}$ tends to $+\infty$.
For any extension $T$ on $L_p(0,1)$ of  a  $\de$-isometric map $t$ between $Y_n$ and $Z_n$,   we have that  $t^{-1} Q_n T$ is a projection onto $Y_n$. Since $\|t^{-1} Q_n T\| \geq C_n$, it follows that $\|T\| \geq c^{-1} C_n/(1+\de)$, which tends to $+\infty$.
\fprue
\nota\label{j34irjw4ierw566e}
 The terminology ``homogeneous'' is commonly used in classical  and metric model theory. In Banach space theory,  it  has to be related to the concept of ``disposition'', for example used by V. I. Gurarij in \cite{Gu} to define his space (of universal ``placement'').   A Banach space $E$ is of {\em approximate disposition} when for every $X,Y\in \age(E)$,  $\de>0$,   $\iota\in \Emb(X,Y)$ and every $\ga\in \Emb(X,E)$ there is $\eta\in \Emb_\de(Y,E)$ such that $\eta\circ \iota= \ga$. When $\age(E)$ consists of all finite dimensional normed spaces,  $E$ is called of {\em universal} approximate disposition.  It is easy to see that if $E$ is \auh then it is of approximate disposition.
   
 \fnota

We have the following characterization of $\mc G$-Fraïssé Banach spaces. Note that when $\mc G\preceq \age(E)$ the $E$-Kadets function $\ga_E$ is well defined on $\mc G$. 
\begin{theorem}\label{main_thm_fraisse}  The following are equivalent for a Banach space $E$ and $\mc G\preceq \age(E)$. 

\begin{enumerate}[1)]
\item $E$ is  $\mc G$-Fraïssé and $\ga_E$ is a complete pseudometric on $\mc G$. 

\item  $E$ is  weak $\mc G$-Fraïssé and $\ga_E$ is a complete pseudometric that is  uniformly  equivalent to $d_\mr{BM}$ on $\mc G_k$ for every $k$.
\item   $E$ is  weak $\mc G$-Fraïssé and $\mc G$ is $d_\mr{BM}$-compact. 

\end{enumerate}
Consequently,  $E$ is Fraïssé if and only if it is weak-Fraïssé and $\age(E)$ is $d_\mr{BM}$-compact. 
\end{theorem}
It follows from this that the Hilbert   and the Gurarij spaces are very special Fraïssé spaces: Recall that a Banach space $Y$ is {\em finitely representable} in $X$ if  $\mr{Age}_k(Y)$ is included in the $d_\mr{BM}$-closure $\overline{\mr{Age}_k(X)}^\mr{BM}$ of $\age_k(X)$ for every $k$.  
 \prop\label{iiiobgnghfgdsd}
 Let $E$ be a Fraïssé Banach space. The following are equivalent for  a separable Banach space $X$.
 \begin{enumerate}[1)]
 \item $X$ is finitely representable on $E$.
 \item $X$ can be isometrically embedded into $E$.
  \end{enumerate}  
Consequently, 
\begin{enumerate}[1)]\addtocounter{enumi}{2}
\item  $\ell_2$ is the minimal separable Fraïssé Banach space,
\item  $\mbb G$ is the only separable Fraïssé  Banach space with trivial cotype\footnote{Recall that a  Banach space $X$  has  cotype $q$  for some $2 \le q\le \infty$ when there is some constant $C$ such that  $(\sum_{j=1}^n \nrm{x_j}^q )^{1/q}\le C ((1/2^n)\sum_{(\theta_j)_{j=1}^n \in \{-1,1\}^n}\nrm{\sum_{j=1}^n \theta_j x_j}^q)^{1/q}$ for every finite sequence $x_1,\dots, x_n\in X$. $X$ has no-trivial cotype when $X$ has cotype $q$ for some $2\le q<\infty$.}
\item   $\mbb G$ is the maximal separable Fraïssé Banach space. 
 
\end{enumerate}

 \fprop
\begin{proof}  
{\it 1)} implies {\it 2)}:  Suppose that $X$ is finitely representable in $E$.   Let $(x_j)_{j\in \N}$ be a dense set of $X$, and for each $n$, let $X_n:=\langle x_j\rangle_{j<n}$. Since $X$ is finitely representable and $\age(E)$ is Banach-Mazur compact,  we can find for each $n$  some $\ga_n\in \Emb(X_n,E)$. By the \auh of $E$, for each $n\ge 1$ we can find $g_n\in \iso(E)$ such that $\nrm{g_n\circ \ga_n\rest X_{n-1}-\ga_{n-1}}\le 1/2^n$. Let $\eta_0:=\ga_0$, $\eta_n:=g_1\circ \cdots \circ g_n \circ \ga_n$ for $n\ge 1$. It follows that $\nrm{\eta_{n}\rest X_{n-1}-\eta_{n-1}}\le 1/2^{n}$, $(\eta_l \rest X_n)_{l\ge n}$ is a Cauchy sequence with respect to the norm metric. Since this is a complete metric, we can define $\eta: \langle x_j\rangle_{j\in \N}\to E$, $\eta(x):=\lim_{l\ge n} \eta_l(x)$ if $x\in X_n$. Clearly this is an isometric embedding that extends to $X$.  

{\it 3):} By Dvoretzky's Theorem $\ell_2$ is finitely representable in $E$, so $\ell_2$ isometrically embeds into $E$ if $E$ is Fraïssé.  {\it 4):}   A classical result by B. Maurey and G. Pisier \cite{MaPi} states that $E$ has trivial cotype if and only $E$ contains  all $\ell_\infty^n$'s uniformly; that is, there is some $C\ge 1$ such that  $d_\mr{BM}(\ell_\infty^n, \age_n(E))\le C$ for all $n$.  By the finite version of James theorem, $\ell_\infty^n\in \overline{\age_n(E)}^\mr{BM}$, so it follows that $\ell_\infty^n\in \age(E)$ for every $n$. Then $E$ is universal for separable spaces. Since $\mbb G$ is the unique, up to isometry, universal separable (AuH) Banach space (see \cite{kubissolecki} or Proposition \ref{oi43hjiio43437}), $E$ and $\mbb G$ are isometric.
\end{proof}

For the proof of Theorem \ref{main_thm_fraisse} we will use the next.
\begin{proposition}\label{lkklkir8} Let $E$ be a Banach space, and suppose that $\mc G\preceq \age(E)$. 
\begin{enumerate}[1)]
\item  If $E$ is \afh{\mc G}, then $\ga_E$ is defined on $\mc G_E$ by the formula
$\ga_E(X,Y)=\inf_{g\in \iso (E)} \La_E(g X, Y)$. 
Consequently,  $\ga_E$ is a pseudometric on each $\mc G_k$. 
\item In general,  $d_\mr{BM}(X,Y) \le 4 k \La_E(X,Y)$ for every $X,Y\in \age_k(E)$ such that $\La_E(X,Y)\le 1/(2 k) $, and    consequently,  if $E$ is   \afh{\mc G},    the identity  map $(\mc G_k,\ga_E)\to    (\mc G_k,d_{\mr {BM}} )   $ is uniformly continuous for each $k$.
 
\item If $E$ is weak $\mc G$-Fraïssé, then $ \ga_E$ is a pseudometric on each $\mc G_k$  topologically equivalent to $d_\mr{BM}$.
\end{enumerate}
\end{proposition}  
\begin{proof}
{\it 1):} We  use the following. 
\clam\label{loi43jiori34443}
For $\ga,\eta\in \Emb_\de(X, E)$  one has that $\La_E(\ga X,\eta X)\le 2(1+\de)\nrm{\ga-\eta}$.
\fclam
\prucl	
Given $\ga x\in X$ of norm one,
\begin{align*}
\nrm{\ga x - \frac{\eta x}{\nrm{\eta x}}} \le &  \nrm{\ga x - \eta x }+\nrm{\eta x - \frac{\eta x}{\nrm{\eta x}}}=\nrm{\ga x - \eta x }+|\nrm{\eta x}-1|
= \nrm{\ga x - \eta x }+|\nrm{\eta x}-\nrm{\ga x}|\le \\
\le & 2\nrm{\ga x - \eta x }\le 2(1+\de)\nrm{\ga-\eta}. \qedhere
\end{align*}
\fprucl
Fix $X,Y\in \mc G_E$ with $\dim X=\dim Y$.  Fix also isometric embeddings $\ga:X\to E$, $\eta: Y\to E$  and $\vep>0$,   let $g,h\in \iso(E)$ be such that $\nrm{g\rest  X-\ga},\nrm{h\rest   Y -\eta}\le \vep$.    Since $\La_E$ is invariant for the natural action of $\iso(E)$ on $\age(E)$,
 \begin{align*}
 \La_E(h^{-1}g   X,   Y)=& \La_E(g   X,h  Y)\le  \La_E(g   X,\ga   X)+\La_E(h   Y, \eta   Y) +\La_E(\ga   X, \eta   Y) \le \\
 \le & 2\nrm{g\rest   X-\ga} + 2\nrm{h\rest   Y -\eta} + \La_E(\ga   X,\eta   Y) \le  \La_E(\ga   X,\eta   Y)  +4 \vep.
 \end{align*}

{\it 2)}: 
Let $(x_j)_{j<k}$ be an {\em Auerbach} basis of $X$, that is, $(x_j)_j$ is normalized and  $\max_j |a_j|\le \nrm{\sum_j a_j x_j}$ for every $(a_j)_j$.  Suppose that $\La_E(X,Y)\le 1/(2k)$.  For each $j<k$, let $y_j\in B_Y$ be such that $\nrm{x_j-y_j}\le \La_E(X,Y)$. Given $(a_j)_j$, we have that  $\nrm{\sum_j a_j y_j}\le \nrm{\sum_j a_j x_j} +\max_j |a_j| \sum_j \nrm{x_j-y_j}\le (1+ k \La_E(X,Y))  \nrm{\sum_j a_j x_j} $, and similarly, $\nrm{\sum_j a_j y_j}\ge (1-k \La_E(X,Y))\nrm{\sum_j a_j x_j}$.  Since $\La_E(X,Y)<1/k$, it follows that $\theta:X\to Y$ defined linearly by $x_j\mapsto y_j$  is an isomorphism between $X$ and $Y$, and since $\La_E(X,Y)\le 1/(2k)$ we obtain that $(1-k\La_E(X,Y))\ge1/ (1+2 k \La_E(X,Y))$.  This implies that $\nrm{\theta},\nrm{\theta^{-1}}\le 1+2k\La_E(X,Y)$, and so
$$d_\mr{BM}(X,Y) \le\log( \nrm{\theta}\cdot \nrm{\theta^{-1}})\le \log((1+2k \La_E(X,Y))^2)\le 4 k \La_E(X,Y).$$
It follows from this inequality and the fact that $d_\mr{BM}$ is $\iso(E)$-invariant that    $\id: (\mc G_k,\ga_E)\to    (\mc G_k,d_{\mr {BM}} )   $ is uniformly continuous for each $k$.

{\it 3)}: For suppose that $E$ is weak $\mc G$-Fraïssé. To simplify the notation, we may assume that $\mc G\con \age(E)$. Suppose that $(X_n)_n$ in $\mc G_k$  $d_\mr{BM}$-converges to  $X\in \mc G_k$. Fix $\vep>0$, and let $0<\de\le 1$ be such that $\iso(E)\acts \Emb_\de(X,E)$ is $\vep/4$-transitive. Let $n_0$ be such that $d_\mr{BM}(X,X_n)\le \log({1+\de})$, and for each $n\ge n_0$ choose $\theta_n:X \to X_n$ so that $\nrm{\theta_n}\cdot \nrm{\theta_n}^{-1}\le 1+\de$. Observe that   $\theta_n\in \Emb_{\de}(X,X_n)$. It follows that, for such $n$'s, there is $g_n\in \iso(E)$ such that $\nrm{\theta_n-g_n\rest X}\le \vep/4$, and by Claim \ref{loi43jiori34443}, $\ga_E(X,X_n)\le \La_E(X_n,g_n (X))\le 2(1+\de)\vep/4\le \vep$.  
\end{proof}  

\begin{proof}[{\sc Proof of Theorem \ref{main_thm_fraisse}}] Obviously, the statement we have to prove is equivalent to the corresponding one for any other family $\mc H$ such that $\mc H\equiv \mc G$, so for convenience we assume that $\mc G=\mc G_E$ and that it is $\La_E$-closed in $\age(E)$.

{\it 1)} $\implies$ {\it 2)}:  We already know by Proposition      \ref{lkklkir8}  {\it 2)} that  $\id: (\mc G_k,\ga_E)\to    (\mc G_k,d_{\mr {BM}} )   $ is uniformly continuous for each $k$. The other part readily follows from the next. 

 \clam
Suppose that $\vep>0$ and $0<\de\le 1$. If $\iso(E)\acts \Emb_\de(X,E)$ is $\vep$-transitive, then $d_\mr{BM}(X,Y)\le \de/2$ implies that $\ga_E(X,Y)\le 4 \vep$ for every $Y\in \age(E)$.
 \fclam
 \prucl
  We assume the hypotheses, and  suppose that $d_\mr{BM}(X,Y)\le \de/2$. Let $\theta: X\to Y$ be an isomorphism such that $\nrm{\theta }\nrm{\theta^{-1}}\le e^{\de/2}\le 1+\de$. Hence, there is some $g\in \iso(E)$ such that $\nrm{g\rest X-\theta}\le \vep$. By Claim  \ref{loi43jiori34443}, $\La_E(g X, \theta X)\le 4\vep$, so $\ga_E(X,Y)\le 4\vep$.
 \fprucl
   {\it 2)} $\implies$ {\it 3)}: Suppose that $\ga_E$ is a complete pseudometric that is   uniformly equivalent to $d_\mr{BM}$  on $\mc G_k$ for every $k$. We have to prove that $\mc G_k$ is compact, so let   $(X_n)_n$ be a sequence in $\mc G_k$. Since $d_\mr{BM}$ is a compact pseudometric on the class of all $k$-dimensional Banach spaces, we can extract a $d_\mr{BM}$-Cauchy subsequence $(X_n)_{n\in N}$. Since   $\id: (\mc G_k,d_\mr{BM})\to    (\mc G_k,\ga_E)$    is uniformly continuous,  $(X_n)_{n\in N}$ is $\ga_E$-Cauchy.  Since  $\ga_E$ is   complete in $\mc G$, it follows that $(X_n)_{n\in N}$ $\ga_E$-converges to some $X\in \mc G_k$.   Finally, the identity  $\id: (\mc G_k,\ga_E)\to    (\mc G_k,d_{\mr {BM}} )$ is continuous,  hence, $(X_n)_{n\in N}$ $d_\mr{BM}$-converges to $X$.   

  {\it 3)} $\implies$ {\it 1)} : It follows from Proposition \ref{lkklkir8} {\it 3)} and the fact that we are supposing that $\mc G_k$ is $d_\mr{BM}$-compact that $\ga_E$ is also a compact pseudometric, hence complete.  Fix now $k$ and $\vep>0$, and for each $X\in \mc G_k$,  choose $0< \de_{X}\le 1$ such that $\iso(E)\acts \Emb_{\de_X}(X,E)$ is $\vep/2$-transitive. By compactness of $(\mc G_k,d_\mr{BM})$, we can find $(X_j)_{j<n}$ such that $\mc G_k\con \bigcup_{j<n}B_{\mr{BM}}(X_j,\de_{X_j}/3)$.  We claim that $\de:=(1/3)\min_j \de_{X_j}$ works: For suppose that  $\ga,\eta\in \Emb_\de(X,E)$ with $X\in \mc G_k$.  Let $j<n$  and let $\theta: X_j\to X$ be an isomorphism such that $\nrm{\theta}\nrm{\theta^{-1}}\le  e^{\de_{X_j}/3}$. Since $\de_{X_j}\le 1$,  it follows that 
$$\nrm{\theta},\nrm{\theta^{-1}}\le \nrm{\theta}\nrm{\theta^{-1}}\le 1+\frac{\de_{X_j}}2.$$
 Since $(1+\de)(1+\de_{X_j}/2)\le  (1+\de_{X_j}/3)(1+\de_{X_j}/2)\le 1+\de_{X_j}$,    it follows that $\ga \circ \theta, \eta\circ \theta \in \Emb_{\de_{X_j}}(X_j, E)$, so there is $g\in \iso(E)$ such that $\nrm{g\circ \ga \circ \theta- \eta \circ \theta}\le \vep/2$ . Hence,  $\nrm{g\circ \ga - \eta  }\le (\vep/2) \nrm{\theta^{-1}}\le \vep$.  
 
 \end{proof}  
It is interesting to  note that being Fraïssé   is preserved by taking ultrapowers.  Recall that given a Banach space $E$, and given a non-principal ultrafilter $\mc U$ on $\N$, we write $E_\mc U$ to denote the {\em ultrapower} $E^\N/\mc U$. For each $n\in \N$, let    $\pi_n: E^\N \to E$ be the $n^\mr{th}$ projection, $\pi_n((x_m)_m)=x_n$.  We denote by $\iso(E)_\mc U$ the subgroup of $\iso(E_\mc U)$  consisting of all isometries of the ultrapower $E_\mc U$ of the form $[(x_n)_n]_\mc U\mapsto [(g_n(x_n))_n]_\mc U$ for some sequence $(g_n)_n\in \iso(E)^\mbb N$.    It is well known that $\age(E_\mc U)\equiv \overline{\age(E)}^\mr{BM}$. 

\prop  
Let $E$ be a Banach space, and let $\mc U$ be a non-principal ultrafilter on $\N$. The following are equivalent.
\begin{enumerate}[1)]
\item $E$ is Fraïssé.

\item  $E_\mc U$ is Fraïssé    and $(\iso(E))_\mc U$ is dense in $\iso(E_\mc U)$ with respect to the strong operator topology SOT. 
\item  For  every $X\in \age(E_\mc U)$ one has that  $(\iso(E))_\mc U\acts \Emb(X,E_\mc U)$ is almost transitive. 

\item  For  every $X\in \age(E_\mc U)$ one has that  $(\iso(E))_\mc U\acts \Emb(X,E_\mc U)$ is transitive. 
\item  For  every separable $X\subset E_\mc U$ one has that  $(\iso(E))_\mc U\acts \Emb(X,E_\mc U)$ is transitive.

\item   $E_\mc U$ is \uh and  $(\iso(E))_\mc U$ is dense in $\iso(E_\mc U)$ with respect to the SOT. 
\end{enumerate} 
Moreover, if any of the previous conditions hold we have that 
\begin{equation}\label{wrmlkqw}
\varpi_{E_\mc U}(k,\de)=  \varpi_E^*(k,\de).
\end{equation}
\fprop
In particular, it follows that if $E$ is Fraïssé, then its ultrapowers are  Fraïssé and \uh.

\prue  

{\it 1)}$\implies$ {\it 2)}: 
\clam\label{43iojrio34jrio34rhh}
For every  $k\in \N$, $X\in \age_k(E_\mc U)$ and $0< \xi<\de$  we have    that $(\iso(E))_\mc U$ acts $(\varpi_E(k,\de)+\xi)$-transitively on $\Emb_{\de-\xi}(X,E_\mc U)$.
 In particular, the action $(\iso(E))_\mc U\acts \Emb(X,E_\mc U)$ is approximately transitive for every $X\in \age(E_\mc U)$,  and    
$\varpi_{E_\mc U}(k,\de)\le   \inf_{\de'>\de}\varpi_E(k,\de')$.  
\fclam
\prucl
 For suppose that  $X\in \age_k (E_\mc U)$. Since we know that $\age_k(E_\mc U)\equiv \overline{\age_k(E)}^\mr{BM}\equiv\age_k(E)$, we may assume that $X\in \age_k(E)$.  Let $\ga\in \Emb_{\de-\xi}(X,E_\mc U)$. For every $x\in X$, write $\ga(x)=[(\ga_n(x))]_\mc U$. It follows that the set $A=\conj{n\in \N}{\ga_n\in \Emb_{\de}(X,E)}\in \mc U$, so for each $n\in A$, choose $g_n\in \iso(E)$ such that $\nrm{g_n\rest X- \ga_n}\le \varpi_k(\de)+\xi/2$.  For each $n\notin A$, let $g_n:=\id_E$. Define $I:=[(g_n)_n]_\mc U$; then  $A\con\conj{n\in \N}{ \nrm{\pi_n \circ I \rest X - \ga_n}\le \varpi_k(\de)+\xi}$, so, this set is in $\mc U$. This means that $\nrm{I\rest X-\ga}\le \varpi_k(\de)+\xi$.
 \fprucl
It  follows from  this, and the characterization in Theorem \ref{main_thm_fraisse} that $E_\mc U$ is Fraïssé. Moreover, the Claim  also easily implies that   $(\iso(E))_\mc U$ is dense in $\iso(E_\mc U)$. 

{\it 2)}$\implies$ {\it 3)}:   Since $E_\mc U$ is Fraïssé, it is (AuH), and this together with the fact we are assuming that $(\iso(E))_\mc U$ is dense in $\iso(E_\mc U)$, readily implies {\it 3)}.  

{\it 3)}$\implies$ {\it 4)}:     Suppose that $X\in \age(E_\mc U)$ and  $\ga\in \Emb(X,E_\mc U))$. For every $x\in X$, write $\ga(x)=[(\ga_n(x))]_\mc U$,  and for each $k\in \N$, choose $I_k:=[(g_n^{(k)})]_\mc U\in (\iso(E))_\mc U$ such that 
 $\nrm{I_k\rest X-\ga}< 1/k$.  For each $k$, let $A_k:=\conj{n\in \N}{ \ga_n \in \Emb(X,E)\text{ and } \nrm{g_n^{(k)} \rest X- \ga_n}\le  1/k}\in \mc U$.   By making a small perturbation if needed, we may assume that for every $n\in \N$ one has that $\ga_n $ is not an isometry.  So, in particular, $\bigcap_k A_k=\buit$. For each $n \in \bigcup_k A_k$, let $k(n)$ be the maximal $k$ such that $n\in A_k$, and define $g_n:= g_n^{(k(n))}$, while for $n\notin \bigcup_k A_k$, let $g_n:=\id_E$. Let $I:=[(g_n)]_\mc U$. Then for every $k\in \N$, 
 $A_k\con \conj{n}{\nrm{g_n \rest X-\ga_n}\le 1/k}$, so, this set is in $\mc U$, and  as a consequence we have that $I\rest X= \ga$. 

{\it 4)}$\implies$ {\it 5)}:  the proof is very similar to the previous one of that {\it 3)}$\implies$ {\it 4)}. We leave the details to the reader.

{\it 5)}$\implies$ {\it 6)}:  the fact that $E_\mc U$ is \uh is trivial, and the second fact is a direct consequence of {\it 4)}.

In order to obtain \eqref{wrmlkqw} we will not prove {\it 6)}$\implies$ {\it 1)} directly, but {\it 6)}$\implies$ {\it 2)} and then {\it 2)}$\implies$ {\it 1)}.

{\it 6)}$\implies$ {\it 1)}:  Suppose otherwise that $E$ is not Fraïssé; by Theorem  \ref{main_thm_fraisse}, there is some $k$, $\vep>0$ and for every $n$ some $X_n\in \age_k(E)$ and $\ga_n\in \Emb_{1/n}(X_n,E)$ such that  
\begin{equation}\label{kj4rui34r4}
\inf_{g\in \iso(E)} \nrm{g\rest X_n- \ga_n}>\vep.
\end{equation}
For each $n$, let $(x_j^{(n)})_{j<k}$ be an Auerbach basis of $X_n$, and for each $j<k$ let $x_j:=[(x_j^{(n)})]_\mc U$, $X:=\langle x_j\rangle_{j<k}$, and $\ga:X\to E_\mc U$, $\ga(\sum_j a_j x_j):=[(\ga_n(\sum_j a_j x_j^{(n)}))]_\mc U$.  Then,  $\ga\in \Emb(X,E_\mc U)$; By {\it 6)}, there is $I=[(g_n)]_\mc U\in (\iso(E))_\mc U$ such that $\nrm{I\rest X- \ga}<\vep/2$.  In particular, $A:=\conj{n}{\nrm{g_n\rest X -    \ga_n }\le \vep/2}\in \mc U$.  Let now $B:=\conj{n}{\text{$ (x_j^{(n)})_{j<k} $ and $(x_j)_{j<k}$ are 2-equivalent} }\in \mc U$.   It follows that given $n\in A\cap B$, and given $x= \sum_j a_j x_j^{(n)}\in X_n$, 
$$ \nrm{g_n (x)- \ga_n(x)} =\nrm{\pi_n \circ I (\sum_j a_j x_j)-  \pi_n\ga(\sum_j a_j x_j)}\le \frac{\vep}{2}\nrm{\sum_j a_j x_j}\le \vep \nrm{x}.$$
This contradicts \eqref{kj4rui34r4}

Finally, suppose that $E_\mc U$ is Fraïssé with modulus $\varpi_{E_\mc U}$, and that $(\iso(E))_\mc U$ is dense in $\iso(E_\mc U)$.
\clam\label{43iojrio34jrio34rhha}
$\varpi_{E_\mc U}(k,\de)=\varpi_{E_\mc U}^*(k,\de)=\varpi^*_{E_\mc U}(p,\de)$.
\fclam
\prucl
Fix $\ga\in \Emb_\de(X,E_\mc U)$, for $X\in \age_k(E_\mc U)$, and $\vep>0$. Without loss of generality, we assume that $\pi\circ \ga\notin \Emb(X,E_\mc U)$. For each $k$, let $\de_k:=\de+1/k$, and let $A_k:=\conj{n\in \N}{\pi_n \ga \in \Emb_{\de_k}(X,E_\mc U)}$, that belongs to $\mc U$. Let $\ga_k:X\to E_\mc U$  be  such that $\pi_n\circ \ga_k = \pi_n\circ \ga$ for every $n\in A_k$. Then $\ga_k\in \Emb_{\de_k}(X,E_\mc U)$, so there is some  $I_k:= [(g^{(k)}_n)_n]_\mc U\in (\iso(E))_\mc U$ such that the set  $B_k=\conj{n\in A_k }{\nrm{g_n\rest X-\pi_n \circ \ga_k}\le \varpi_{E_\mc U}(k,\de_k)+\vep}\in \mc U$. For each $n\in \bigcup_{k}B_k$, let $k(n)$ be the maximum of those $k$'s such that $n\in B_k$. For such $n$, let $g_n:=g_n^{(k(n))}$, and for $n\notin \bigcup_k B_k$, let $g_n:=\id_E$. It follows that $I:=[(g_n)_n]_\mc U\in (\iso(E))_\mc U$ satisfies that $\nrm{I\rest X- \ga}\le \inf_{\de'}\varpi(k,\de')+\vep$.  
 \fprucl
Now let us see that $\varpi_E(k,\de)\le \varpi_{E_\mc U}(k,\de)$: given $\ga\in \Emb_\de(X,E_\mc U)$, we define $\bar \ga \in \Emb_\de(X,E_\mc U)$, $\bar\ga(x):= [(\ga(x))]_\mc U$. Let $I=[(g_n)_n]_\mc U\in (\iso(E))_\mc U$ be such that 
$\nrm{I\rest X-\bar\ga}\le \varpi_{E_\mc U}(k,\de)+\vep/2$, so $A=\conj{n\in \N}{ \nrm{g_n\rest X-\ga}\le \varpi_{E_\mc U}(k,\de)+\vep}\in \mc U$; if we choose $n\in A$, then $\nrm{g_n\rest X-\ga}\le  \varpi_{E_\mc U}(k,\de)+\vep$. 
The readers familiar with the logic structures, will notice that  some of the previous implications  are particular instances of standard facts in continuous or classical model theory. For example, 3) implies 5) is the   countable saturation of non-trivial ultrapowers, while the reverse implication is an instance of \L os Theorem (see, for example \cite{Ho} or \cite{BYBeHeUs}).

From this, and the  Claim \ref{43iojrio34jrio34rhh} and Claim \ref{43iojrio34jrio34rhha} we have that $\varpi_{E_\mc U}(k,\de)= \inf_{\de'>\de}\varpi_{E}(k,\de')$. 
\fprue

\cor	
For $p\neq 4,6,8,\dots$, any non-trivial ultrapower of $L_p(0,1)$ is an \uh $L_p$-space.  \qed
\fcor
It had been already observed in \cite[Proposition 4.13]{ACCGM} that any non-trivial ultrapower of $\mbb G$ is ultrahomogeneous. 
We shall now see that Fraïssé Banach spaces are locally determined.  The following  is a slight modification of a similar concept introduced in \cite{Lup}.
\defi\label{wlkrjwjrew87755}
Given a family $\mc G$ of finite dimensional spaces, let $[\mc G]$ be the class of all   Banach spaces $E$ such that   the collection of subspaces of elements of $\mc G_E$ is $\La_E$-dense in $\age(E)$. 
\fdefi
\nota 
It is easily see that if $E\in [\{\ell_p^n\}_n]$, then $E$ is a $\mc L_{p,1+}$-space, that is for every $\de>0$ and every $X\in \age(E)$ there is some $X\con Y\in \age(E)$ such that $d_{\mr BM}(Y,\ell_{p}^{\dim Y})\le \de$.  Conversely, if $E$ is an $\mc L_{p,1+}$-space that is in addition a {\em stable }$\{\ell_p^n\}_n$-Fraïssé space,  then $E\in [\{\ell_p^n\}_n]$. Stable $\mc G$-Fraïssé spaces are   those for which their moduli does not depend on the dimensions.  Given $X\in \age(E)$ and $\vep>0$, let $0<\de\le 1$ be such that $\varpi_{\mc G, E}(\de)\le \vep/4$, and let $X\con Y\in \age(E)$, $\ga\in \Emb_{\de}(\ell_p^n,Y)$ and $\eta\in \Emb(\ell_p^n,E)$  for some $n$.  Find $g\in \iso(E)$ such that $\nrm{\ga- g\circ \eta}\le \vep/4$. Let $Z:= \im (g\circ \eta)$. It follows from Claim \ref{loi43jiori34443} that  $\La_E(Y, Z)\le 2 (1+\de)\nrm{\ga- g\circ \eta}\le \vep$. 
\fnota	

\begin{theorem}\label{k3whiueuiuhiewew}
 Suppose that $E$ and $F$ are $\mc G$-Fraïssé Banach spaces, with $\mc G\preceq \age(E),\age(F)$ and $E\in [\mc G]$. The following are equivalent.
 \begin{enumerate}[1)]
 \item $F\in [\mc G]$.
 \item $E$ is isometric to $F$.   
 \end{enumerate}   
\end{theorem}  
This result motivates the next notion.
\defi
Let $\mc G$ be a  class of finite dimensional Banach spaces. The {\em Fraïssé limit} of $\mc G$, denoted by $\flim \mc G$ is, if exists,  the unique separable $\mc G$-Fraïssé Banach space $E$ such that $E\in [\mc G]$.
\fdefi
We have the following interesting fact.
\prop\label{kjriojireee8877444}
If $\flim\mc G$ and $\flim \overline{\mc G}^\mr{BM}$ exists, then they are isometric.
\fprop
\prue
Set $E=\flim\mc G$,  $F=\flim \overline{\mc G}^\mr{BM}$ and $\mc H:=\bmc{\mc G}$. Since  $F$ is $\mc G$-Fraïssé,  by Theorem  \ref{k3whiueuiuhiewew},  it suffices to prove that $F\in [\mc G]$: fix $X\in \age(F)$ and $\vep>0$, and let $Y\in (\bmc{\mc G})_F$ be such that  $X\con_{\vep/2} Y$. Let $0<\de\le 1$ be such that $\varpi_{\mc H, F}(\dim Y, \de)\le \vep/8$, and let $Y_0\in \mc G$, $\ga\in \Emb_{\de}(Y_0, Y)$ and $\eta\in \Emb(Y_0, F)$. Find $g\in \iso(F)$ such that $\nrm{\ga-g \circ \eta}\le \vep/8$. It follows that $Y_1:=\im (g\circ \eta)\in \mc G_F$ and by    Claim  \ref{loi43jiori34443}, $\La_F(Y, Y_1)\le \vep/2$. Hence, 
$X\con_{\vep/2} Y\con_{\vep/2} Y_1$ and so $X\con_{\vep} Y_1$.  
\fprue
\begin{proof}[\sc Proof of Theorem \ref{k3whiueuiuhiewew}]
  ${\it 2)}\implies {\it 1)}$ is trivial. Let us prove  that ${\it 1)}\implies {\it 2)}$:  Given two subspaces $F,G$ of a given space $E$, and $\vep>0$, we write $X\con_\vep Y$ when there is some $Z\con \age(Y)$ of dimension $\dim Z=\dim X$ such that $\La_E(X,Z)\le \vep$.
   Fix two rapidly decreasing strictly positive sequences $(\de_k)_k$ and $(\vep_k)_k$  with 
 \begin{enumerate}[i)]
 \item   $\de_k\le \vep_k\le 1$ for every $k$.
\item $\prod_k (1+\de_k)\le \sqrt{2}$, $\sum_{k>l} \vep_k \le \vep_l/4$. 
 \item $\varpi_X(\dim X_k, \de_k),\varpi_Y(\dim Y_k, \de_k)\le \vep_k/4$
  
 \end{enumerate}  
  Fix now sequences $(X_k)_k$ in  $\mc G_X$ and  $(Y_k)_k$   in $\mc G_Y$ whose respective unions are dense, and such that $\sum_{k<l} X_k \con_{\vep_l}  X_l$ and $\sum_{k<l} Y_k \con_{\vep_l}  Y_l$.  
Using that both $X$ and $Y$ are $\mc G$-Fraïssé Banach spaces and that $\bigcup_k X_k$ and $\bigcup_k Y_k$ are dense in $X$ and $Y$, respectively,   we are going to find two sequences of integers $(m_k)_k$ and $(n_k)_k$ with $m_k\le n_k<m_{k+1}$,  and  $\ga_k\in \Emb_{\de_{m_k}}(X_{m_k},Y_{n_{k}})$, $\eta_k\in \Emb_{\de_{n_k}}(Y_{n_{k}},X_{m_{k+1}})$ such that   
\begin{equation}\label{i5iou4776783rtyy}
	 \nrm{\eta_k\circ \ga_k -i_{X_{m_k}, X}  }\le \vep_{m_k} \text{ and }\nrm{\ga_{k+1}\circ \eta_k -i_{Y_{n_{k}, Y  }  }  }\le {\vep_{n_k}}.
\end{equation}
  Suppose that $\ga_k :X_{m_k}\to Y_{n_{k}}$ is defined.    Since $Y_{n_{k}}\in \mc G_Y$,   we can  find $f\in \Emb(Y_{n_{k}}, X)$ such that $\nrm{f \circ \ga_k -i_{X_{m_k}, X } }\le \varpi_X(\dim X_{m_k}, \de_{m_k})+\de_{m_k}$.   Let $m_{k+1}>n_{k}$ be large enough so that there is some   linear mapping $\iota:f(Y_{n_{k}})\to X_{m_{k+1}}  $ such that   $\nrm{\iota - i_{f(Y_{n_{k}}),X} }\le \de_{n_k}/2$. Then it follows that $\eta_k:=\iota\circ f\in \Emb_{\de_{n_{k}}}(Y_{n_{k}}, X_{m_{k+1}})$  and $\nrm{\eta_k \circ \ga_k -\id_{X_{m_k}}}\le \vep_{m_k}$.  Similarly one finds $n_{k+1}$ and $\ga_{k+1}$. 
We re-enumerate as follows: Let $\widehat{X}_k:=X_{m_k}$, $\widehat{Y}_k:= Y_{n_k}$, $\widehat{\vep}_k:= \vep_{m_k}$, $\widehat{\de}_k:=\de_{m_k}$ for every $k$. 
\clam
Fix $x\in X$ and $\vep>0$. Let $m$ be such that $\{x\}\con_{\vep/2} \widehat{X}_m$ and   $\widehat\vep_{m-1}\le \vep/2$. Then for every $k,l\ge m$ and $v\in X_k$, $w\in X_l$ such that $\nrm{v-x},\nrm{w-x}\le \vep\nrm{x}$ we have that 
\[ \nrm{\ga_k (v) -\ga_l (w)}\le 3 \vep \nrm{x}.\]
\fclam
Once this is established, we define $\ga: X\to Y$ as follows: Give $x\in X$, let $x[n]\in X_n$ be such that 
$\nrm{x-x[n]}=d(x, X_n)$. By the choice of the sequences $(X_n)_n$, $x[n]\to_n x$. By the previous claim, $(\ga_n(x[n]))_n$ is a Cauchy sequence. Let $\ga(x):=\lim_n \ga_n(x[n])$. It is easily seen that $\ga$ is a bounded linear mapping $\ga:X\to Y$ such that $\nrm{\ga}\le 1$. Similarly one defines $\eta:Y\to X$ and proves that $\ga\circ \eta= \id_Y$, $\eta\circ \ga=\id_X$. 
\prucl
 Fix all data, and suppose that $k\ge l$. If $k=l$, then $\nrm{\ga_k (v) -\ga_k(w)}\le \nrm{\ga_k}\nrm{v-w}\le 4 \vep\nrm{x}$. Suppose that $k=l +n$ with $n>0$. Set $w_0:= w$, and  for each $1\le j < n$, set $w_{j+1}:=  \eta_{l+j} (\ga_{l+j}(w_j))\in \widehat{X}_{l+j+1}$.  Notice that  $\nrm{\ga_{l+j}(w_j)}\le \prod_k (1+\de_k)^2\nrm{w_0}\le 4\nrm{x}$.  
  Now, it follows from \eqref{i5iou4776783rtyy}, 
 \begin{align}\label{lki5u5u1}
\nrm{\ga_{k} (w_n )-\ga_l (w)}=& \nrm{\ga_{k} (w_n )-\ga_l (w_0)}\le \sum_{j=0}^{n-1} \nrm{\ga_{l+{j+1}} (w_{j+1} )-\ga_{l+j} (w_{j})}   =  \nonumber \\
=&   \sum_{j=0}^{n-1} \nrm{\ga_{l+{j+1}} (\eta_{l+j} (\ga_{l+j}(w_j))      )-\ga_{l+j} (w_{j})}  \le \sum_{j=0}^{n-1} \widehat{\vep}_{l+j} \nrm{\ga_{l+j}}\nrm{w_j} \le  \nonumber \\
\le &4 \nrm{x}\sum_{j=0}^{n-1} \widehat{\vep}_{l+j} \le \frac{\vep}2 \nrm{x}. 
 \end{align}
On the other hand,  since $\nrm{w_{j+1}-w_j}= \nrm{\eta_{l+j}(\ga_{l+j}(w_j))-w_j}\le \widehat{\vep}_{l+j}\nrm{w_j}$,  again invoking \eqref{i5iou4776783rtyy}, 
\begin{align}\label{lki5u5u2}
\nrm{w_n- v}\le &  \sum_{j=0}^{n-1}\nrm{w_{j+1}-w_j}+\nrm{w_0- v} =\sum_{j=0}^{n-1}  \widehat{\vep}_{l+j}\nrm{w_j} + 2\vep\nrm{x}\le \widehat{\vep}_{l-1}\nrm{x} + 2\vep\nrm{x}\le \frac 52\vep \nrm{x}
\end{align} 
Combining \eqref{lki5u5u1} and \eqref{lki5u5u2} we obtain that $\nrm{\ga_{k}(v)-\ga_l(w)}\le 3\vep\nrm{x}$. 
\fprucl

\vspace{1cm}

 Note now that given $x\in X_{m_k}$ of norm 1, and $l\ge k$, we have that 
$\nrm{\ga_l (x)- \ga_{l+1}\circ \eta_l \circ \ga_l (x)}\le \vep_l \nrm{\ga_l(x)}\le \vep_l (1+\de_l) \le 2\vep_l$. Also,   
$\nrm{\ga_{l+1} (x)- \ga_{l+1}\circ \eta_l \circ \ga_l (x)}\le \nrm{\ga_{l+1}}\vep_l\le 2\vep_l$.  
So, given $l\ge k$ and $m\in \N$, we have that $\nrm{\ga_{l+m}(x)- \ga_l(x)}\le 4\sum_{j=l}^{l+m-1}\vep_j \le \vep_{l-1}$. This proves that $(\ga_l(x))_{l\ge k}$ is a Cauchy sequence, and similarly, $(\eta_l(y))_{l\ge k}$ is also Cauchy for $y\in Y_{n_k}$. we define $\ga: \bigcup_k X_{m_k}\to \bigcup_{k}Y_{n_k}$, and $\eta:  \bigcup_{k}Y_{n_k}\to \bigcup_k X_{m_k}$, by $\ga(x):=\lim_{l\ge k} \ga_k(x)$ and $\eta(y):=\lim_{l\le k} \eta_k(y)$ for $x\in X_{m_k}$ and $y\in Y_{n_k}$. The extensions of $\ga$ and $\eta$ to $X$ and $Y$, respectively are isometric embedding and one if the inverse of the other, so $X$ and $Y$ are isometric.   ${\it 2)}\implies {\it 1)}$ is trivial.
\end{proof}  

In particular,  if $X$ and $Y$ are separable Fraïssé spaces such that $\age(X)\equiv \age(Y)$, then $X$ is isometric to $Y$, denoted by $X\equiv Y$.  However we have the following stronger characterization. 	

  \begin{proposition}[Uniqueness]\label{oi43hjiio43437}  
For separable \auh spaces $X$ and $Y$ the following are equivalent.
  \begin{enumerate}[(a)]
  \item $X\equiv Y$.
  \item $\age(X)\equiv \age(Y)$.   
  \end{enumerate}

 \end{proposition}

\begin{proof}
  We will find  $\con$-increasing sequences $(X_n)_n$ and $(Y_n)_n$ and isometric embeddings $\ga_n:X_n\to Y_n$ and $\eta_n:Y_n\to X_{n+1}$ such that
  \begin{enumerate}[1)]
  \item $X_n\in \mc F$ and $Y_n\in \mc G$ for every $n$, and $\bigcup_n X_n$ and $\bigcup_n Y_n$ are dense in $X$ and $Y$, respectively.
  \item $\nrm{\eta_n\circ \ga_n-\id_{X_n}}\le 2^{-n}:=\vep_n$.
  \item $\nrm{\ga_{n+1}\circ \eta_n-\id_{Y_n}}\le 2^{-n}$.
  \end{enumerate}
Once this is done, given $x\in X_n$ we have that   $\nrm{\ga_{n+1}(x)-\ga_n(x)}\le 2^{-n+1}$. So, $(\ga_n(x))_{m\ge n}$ is a Cauchy sequence. We define then $g_n: X_n\to Y$ by $g_n(x):=\lim_{m\ge n} \ga_m(x)$. Obviously, $g_{n}\rest X_n=\ga_n$, so we can define $g:\bigcup_n X_n\to Y$ piece-wise, and we extend it to $g:X\to Y$. Similarly one defines $h:Y\to X$. It is easy to see that $h \circ g= \id_X$ and $g\circ h= \id_Y$.  Let us argue that the sequences  above exist:  We fix $\{x_n\}_n$ and $\{y_n\}_n$  dense subsets of $X$ and $Y$, respectively with $x_0=y_0=0$. 
Let $X_0=Y_0:=\{0\}$, $\ga_0(0)=0$. Suppose defined $X_0\con \cdots \con X_n$ with $\{x_k\}_{k\le j}\con X_j$, $Y_0\con \cdots \con Y_n$ with $\{y_k\}_{k\le j}\con Y_j$,  $\ga_j:X_j\to Y_j$  for $j\le n$ and $\eta_j: Y_j\to X_{j+1}$ for $j<n$ such that   
$\nrm{\ga_{j+1}\circ \eta_{j}- \id\rest Y_{j}}, \nrm{\eta_{j}\circ \ga_{j}-\id\rest X_{j}}   \le \vep_{j}$ for every $j<n$.   We choose $X_{n+1}$ and $\eta_n: Y_n\to X_{n+1}$ as follows.   First fix  $\theta\in \Emb(Y_n,X) $. This is possible since $Y_n\in \age(Y)\equiv \age(X)$. Now let $g\in \iso(X)$ be such that $\nrm{g\circ \theta\circ \ga_n- \id_{X_n}}\le \vep_n$, and set $\eta_n:=g\circ \theta\in \Emb(Y_n,X)$.  Let $X_{n+1}:=X_n+ \im \eta_n+  \langle x_j\rangle_{j\le n+1}$. Then $\eta_n:Y_n\to X_{n+1}$ satisfies what we want. Similarly one can find the desired $Y_{n+1}$ and $\ga_{n+1}:X_{n+1}\to Y_{n+1}$.   
\end{proof}

\subsection{Classes of finite dimensional spaces}\label{fraisse_correspondence_sda}
The question addressed now is for which families $\mc G$ of finite dimensional spaces there is some separable Fraïssé Banach space $E$  such that $\mc G\equiv\age(E)$, or more generally, such that $\mc G\preceq\age(E)$ and $E\in [\mc G]$.   In the discrete algebraic case, this is the content of the classical  result by R. Fraïssé \cite{Fra} characterizing  ultrahomogeneous countable  first order structures in terms of properties of their classes of finitely generated substructures.  We present now  a similar characterization for Banach spaces (see also \cite{Lup} for similar results for {\em stable} Fraïssé operator spaces and systems, or \cite{BY}  in general for metric structures).

\begin{definition}[Fra\"\i ss\'e classes]
Let $ \mc G$ be a class of   finite dimensional Banach spaces. 
\begin{enumerate}[(a)]
\item $ \mc G$ has the  \emph{hereditary property \hp} when for every $ X\in  \mc G$ and every $Y$, if $\Emb(Y,X)\neq \buit$, then   $Y\in  \mc G_\equiv$,
\item $ \mc G$ has the \emph{amalgamation property \ap} when $\{0\}\in \mc G$ and  for every  $X,Y,Z\in  \mc G$  and every isometric embeddings   $\ga:X\to Y$ and $\eta:X\to Z$ there is $H\in  \mc G$  and isometric embeddings $i:Y\to H$ and $j:Z\to H$ such that $i\circ \ga =j\circ \eta$.
\item $ \mc G$ has the \emph{near amalgamation property \nap} when $\{0\}\in \mc G$ and  for every $\vep>0$, $X,Y,Z\in  \mc G$  and every isometric embeddings   $\ga:X\to Y$ and $\eta:X\to Z$ there is $H\in  \mc G$  and isometric embeddings $i:Y\to H$ and $j:Z\to H$ such that $\nrm{i\circ \ga -j\circ \eta}\le \vep$.

\item $ \mc G$ is a  {\em weak amalgamation class}    when  $\{0\}\in \mc G$ and   for every $\vep>0$ and $X\in  \mc G$ there is $\de>0$ such that for every $Y,Z\in  \mc G$ and $\de$-isometric embeddings   $\ga:X\to Y$ and $\eta:X\to Z$ there is $H\in  \mc G$  and isometric embeddings $i:Y\to H$ and $j:Z\to H$ such that $\nrm{i\circ \ga-j\circ \eta}\le \vep$;
\item $ \mc G$  is an {\em amalgamation class}   when $\{0\}\in \mc G$ and  for every dimension $k\in \N$  and $\vep>0$ there is $\de>0$ such that if  $X\in \mc G_k$,  $Y,Z\in \mc G$ and $\ga\in \Emb_{\de}(X,Y)$,  $\eta\in \Emb_{\de}(X,Z)$,  then there is   $H\in  \mc G$  and isometric embeddings  $i:Y\to H$ and $j:Z\to H$ such that $\nrm{i\circ \ga-j\circ \eta}\le \vep$.

\item $ \mc G$ is a {\em Fraïssé class} when it is hereditary $d_\mr{BM}$-compact amalgamation class.  

\end{enumerate}	

\end{definition}	
The modulus of a (weak) amalgamation class is defined similarly to the modulus of a (weak) $\mc G$-Fraïssé space: 
the class $\mc G$ has modulus $\varpi:\N\times\R^+\to \R^+$  when $\{0\}\in \mc G$ and   for every $\de>0$  every $k$ and every $\vep>0$,  if $X\in  \mc G_k$, $Y,Z\in  \mc G$  and $\ga\in \Emb_{\de}(X,Y)$,  $\eta\in \Emb_{\de}(X,Z)$,  then there is   $H\in  \mc G$  and isometric embeddings $i:Y\to H$ and $j:Z\to H$ such that $\nrm{i\circ \ga-j\circ \eta}\le \varpi(k,\de)+\vep$.  We have the following interesting implication.

\begin{proposition}\label{loi34iooi547889659866}
\begin{enumerate}[1)]
\item  The amalgamation and the near amalgamation properties are equivalent for compact and hereditary families. 
 \item  The Banach-Mazur closure of an   hereditary amalgamation class is  a Fraïssé class. 
\end{enumerate}  

\end{proposition}  
\begin{proof}
{\it 1)}: Fix a compact and hereditary family $\mc G$ with the \nap, and fix also  $X,Y,Z\in \mc G$ and isometric embeddings $\ga:X\to Y$ and $\eta:X\to Z$. Choose a sequence $(\vep_n)_n$ of strictly positive real numbers and  decreasing to zero, and for each $n$, let $V_n\in \mc G$ and $i_n\in \Emb(Y,V_n)$ and $  j_n\in \Emb(Z,V_n)$ such that $\nrm{i_n\circ \ga- j_n\circ \eta}\le \vep_n$. For each $n$, let  $W_n:=\im i_n +\im j_n\con V_n$. Since $\mc G$ is hereditary, $W_n\in \mc G$. By passing to a subsequence if needed, we assume that the all $W_n$ have the same dimension, and that $W_n$ converge with respect to the Banach-Mazur pseudometric to $W$, that belongs to $\mc G$. For each $n$, let $\theta_n: W_n\to W$ be such that $ \lim_n\max\{\nrm{\theta_n}, \nrm{\theta_n^{-1}}\}=1$.  Let $i: Y\to W$, $j:Z\to W$ be  accumulation points of $(\theta_n\circ i_n)_n$ and of $(\theta_n\circ j_n)_n$, respectively. It follows that  $i$ and $j$ are isometric embeddings and $i \circ \ga= j\circ \eta$.  

{\it 2)}: Fix an hereditary amalgamation class $\mc G$, and $\mc H$ be the Banach-Mazur closure of $\mc G$. Notice that $\mc H$ is also hereditary. Now fix a dimension $k$ and $\vep>0$. Let $\de>0$ be witnessing that $\mc G$ is an amalgamation class for the parameters $k$ and  $\vep/2$. We claim that $\de/2$ works for $\mc H$: Fix $X\in \mc H_k$, $Y,Z\in \mc H$ and $\ga\in \Emb_{\de/2}(X,Y)$, $\eta\in \Emb_{\de/2}(X,Z)$.  For each $n\ge 1$   let $\de_n:= \de/(12n)$, choose $X_0, Y_n, Z_n\in \mc G$ and $\theta^X:X_0\to X$, $\theta_n^Y: Y\to Y_n$,   $\theta_n^Z: Z\to Z_n$, surjective isomorphisms such that $\nrm{(\theta^X)^{-1}}=1$, $\nrm{\theta^X}\le 1+\de_1$ and $\theta_n^Y\in \Emb_{\de_n}(Y, Y_n)$ and $\theta_n^Z\in \Emb_{\de_n}(Z,Z_n)$. Notice that given $n\ge1$, $\ga_n:=\theta_n^Y\circ \ga\circ \theta^X\in \Emb_{\de}(X_n,Y_n) $ and $\eta_n:=\theta_n^Z\circ \eta\circ \theta^X\in \Emb_{\de}(X_n,Z_n) $ so we can choose $V_n\in \mc G$ and $i_n\in \Emb(Y_n,V_n)$, $j_n\in \Emb(Z_n,V_n)$ such that $\nrm{i_n\circ \ga_n -j_n\circ \eta_n}\le \vep/2$. Since $\mc G$ is hereditary, as before in {\it 1)} we may assume that $\dim V_n$ is constant, and the sequence $(V_n)_n$ converges in the Banach-Mazur norm to $V\in \overline{\mc G}^\mr{BM}$.  Choose $\theta_n^V: V_n \to V$ be such that $\theta_n^V\in \Emb_{\de_n}(V,V_n)$. Choose a convergent subsequence of $(\theta_n^V\circ i_n \circ \theta_n^Y)_n$ and of $(\theta_n^V\circ j_n \circ \theta_n^Z)_n$ with limits $i\in \Emb(Y,V)$ and $j\in \Emb(Z,V)$, respectively. Then $\nrm{i\circ \ga -j\circ \eta}\le \vep$. 
 \end{proof}  
In the previous proposition, the condition of being hereditary seems necessary: For example, Let $\mc F$ consists of all 2-dimensional  polyhedral spaces together with the spaces $\ell_\infty^n$ of any dimension $n$. Then $\mc F$ is an amalgamation class and its closure is the class of all 2-dimensional normed spaces together  with $\ell_\infty^n$ of any dimension, which does not have the near amalgamation property.

In classical Fraïssé theory for discrete algebraic structures, Fraïssé classes $\mc G$  have in addition the joint embedding property \jep:  For every $X,Y\in \mc G$ there is  some $Z\in \mc G$ such that $X$ and $Y$ can be isomorphically embedded into $Z$. However it is easy to see that this property in the context of Banach spaces is a property of amalgamation classes (because  $\{0\}\in \mc G$). 

 Amalgamation  families  with modulus  not depending on the dimension $k$ where introduced by M. Lupini in \cite{Lup}; they are said to have the {\em stable} near amalgamation property and in this case Fraïssé classes are called stable Fraïssé classes. Examples of them are the class of all finite dimensional Hilbert, and Banach spaces, corresponding to the two known stable Fraïssé spaces: the Hilbert and the Gurarij, respectively.  We do not know if the class $\age(L_p(0,1))$ is a Fraïssé class for $p\neq 4,6,8,\dots$.   
It is easy to see that $\age(E)$ has always (HP), (JEP) and when $E$ is \auh, \auhp, Fraïssé  then   $\age(E)$ has \nap, is a weak amalgamation class or is an amalgamation class, respectively, and in the (weak) Fraïssé case the respective moduli are the same (when $E$ is Fraïssé,  the compactness of $\age(E)$ follows from Theorem \ref{main_thm_fraisse}).  We call Fraïssé correspondence  the reverse implication, presented in Corollary \ref{ijoiriower3278766755}. Other limits and correspondences of this type have been studied in \cite{BY} and in \cite{Ku}. 

As for the  Fraïssé and the \auhp properties, there is also a metric explanation of  the relation between the weak amalgamation and the amalgamation properties. Given a family $ \mc G$ of finite dimensional spaces with the \jep and the \nap, we can define a ``Kadets-like'' pseudometric on each $ \mc G_k$,  	
$$\ga_ \mc G(X,Y):=\inf \conj{\La_Z(X_0, Y_0)}{ X_0,Y_0\in \mr{Age}_k(Z), \, Z\in  \mc G,\,  X_0\equiv X, \, Y_0\equiv Y}.$$
The following characterization is proved similarly as the equivalences in  Theorem \ref{main_thm_fraisse}. We leave the details to the reader. 
 \begin{proposition}\label{kooiio44432}   The following are equivalent.  
 \begin{enumerate}[(a)]
 \item  $ \mc G$ is an amalgamation class and $\ga_\mc G$ is a complete pseudometric. 
 \item $ \mc G$ is a weak amalgamation class and $\ga_\mc G$ is a complete pseudometric that is uniformly equivalent to $d_\mr{BM}$  on $ \mc G_k$ for every $k$.
 \item  $ \mc G$ is a $d_\mr{BM}$-compact weak amalgamation class. \qed

 \end{enumerate}  

 \end{proposition}

\begin{theorem}\label{oi43hjiio434371}
If $\mc G$ is an amalgamation class with modulus $\varpi$, then $\flim \mc G$ exists and  it has $\mc G$-Fraïssé modulus $\varpi^*$.

\end{theorem}  
The existence of these spaces have been proved  for stable Fraïssé Banach spaces, and for other structures in functional analysis 	in \cite{Lup}.

\cor[Fra\"\i ss\'e correspondence]\label{ijoiriower3278766755}
 The following are equivalent for a class   $\mc G$ of finite dimensional Banach spaces.
 \begin{enumerate}[(1)]
\item $\mc G$ is a Fraïssé class.
 \item $\mc G\equiv\mr{Age}(E) $ of a unique separable Fraïssé  Banach space $E=\flim \mc G$.  
\end{enumerate}
\fcor	
\prue
Suppose that $\mc G$ is a Fraïssé class, i.e. $\mc G$ is a compact, hereditary amalgamation class.   It follows from Theorem \ref{oi43hjiio434371} that $\flim \mc G$ exists. Let us see that 
 $\age(\flim\mc G)\equiv\overline{\mc G}^\mr{BM}$:  We already know that $\mc G\preceq \age(\flim\mc G)$. Let us see that $\age(\flim\mc G)\preceq \mc G$: Since $\flim\mc G\in [\mc G]$ and $\mc G$ is hereditary, we obtain that $\age(\flim \mc G)\preceq \overline{\mc G}^\mr{BM}=\mc G$. 
\fprue

\cor \label{ijoiriowersfsfd3278766755}
 Let $\mc G$  be a class of finite dimensional Banach spaces.
\begin{enumerate}[(1)]
\item If $\mc G$ is an hereditary amalgamation class, then $\flim\mc G$ is the unique separable Fraïssé Banach space $E$ such that $\age(E)\equiv \bmc{\mc G}$. 
\item If $E$ is a separable Banach space such that $\age(E)$ is an amalgamation class, then $\flim\age(E)$ is the unique   separable Fraïssé Banach space $X$  with an isometric copy of $E$ such that $\age(X)= \bmc{\age(E)}$. 
\end{enumerate}
 \fcor	
\prue
 {\it(1)}:    If $\mc G$ is an hereditary amalgamation class, then by Proposition \ref{loi34iooi547889659866} {\it2)}, $\bmc{\mc G}$ is a Fraïssé class,  so $\flim \bmc{\mc G}$ exists and it is a Fraïssé Banach space. Since $\flim \mc G$ also exists, it follows by  Proposition \ref{kjriojireee8877444} that $\flim \mc G=\flim \bmc{\mc G}$. {\it(2)}: Set $X:=\flim\age(E)$. We know from {\it(1)} that $X$ is the unique Fraïssé Banach space with $\age(X)=\bmc{\age(E)}$, and it follows from  Proposition \ref{iiiobgnghfgdsd} that $E$ can be isometrically embedded into $X$. 
\fprue

Before we present a proof of Theorem \ref{oi43hjiio434371}, we give some examples of classes of spaces and limits. 
\begin{example}
For $1\le p\le \infty$, the family $\mc F=\{\ell_p^k\}_{k\in \N}$  is clearly compact, and  it is an amalgamation class.  For $1\le p\neq 2<\infty$, this is done in Proposition \ref{3l4kirjio34riji43}  using  a work of G. Schechtman in \cite{Sch}.   In this case, the Fraïssé limit $\flim\{\ell_p^n\}_n$ is $L_p[0,1]$ (see Proposition \ref{3l4kirjio34riji43}). 

For $p=\infty$, it is rather easy to see that $\{\ell_\infty^n\}_n$ is an amalgamation class, whose Fraïssé limit is the Gurarij space $\mbb G$.

When $p=2$,     $\{\ell_2^k\}_k$ is also hereditary, because a subspace of a Hilbert space is a Hilbert space; so $\{\ell_2^k\}_k$ is a Fraïssé class, whose Fraïssé limit is $\flim\{\ell_2^n\}_n= \ell_2$ and it is automatically a Fraïssé  Banach space.

\end{example}

\begin{example}\label{oi34ioroi34r4}
 $\age(C[0,1])$ is  a Fraïssé class, and $\flim \age(C[0,1])= \mbb G$
 \end{example}
 
 \begin{proof} Being both $C[0,1]$ and the Gurarij space $\mbb G$  universal spaces, we obtain that $\age(C[0,1])=\age(\mbb G)$ is a Fraïssé class because $\mbb G$ is Fraïssé.\end{proof}

\begin{example}\label{12uihiu23ui324}
For $p\neq 4,6,8,\dots$, the class $\age(L_p(0,1))$ is  Fraïssé, whose Fraïssé limit is $L_p(0,1)$.   We will see in Theorem \ref{auh_for_L_p}  that for these $p$'s, $L_p(0,1)$ is weak-Fraïssé; moreover,   for all $p$'s $\age(L_p(0,1))$ is compact:  
Observe that in general,  given a Banach space, the completion $\overline{\age(E)}^\mr{BM}$ of $\age(E)$ under the Banach-Mazur pseudometric is exactly, modulo $\equiv$, equal to $\age(E_\mc U)$ for every non-trivial ultrafilter $\mc U$ over $\N$. It follows from the S. Kakutani  characterization of abstract $L_p$ spaces (see  \cite{Kak}  or \cite[Theorem 1.b.2. of Vol. 2]{LiTza})  that each $E_\mc U$ is isometric to some $L_p(\mu)$ (in fact, W. Henson proved in \cite[Theorem 2.4]{He} that this particular ultrapower is   isometric to the $\ell_p$-sum of $\mk c$ copies of $L_p([0,1]^\mk c)$, $\mk c$ being the cardinality of the continuum). Hence $\overline{\age(L_p(0,1))}^\mr{BM}\equiv\age(L_p(\mu))$. Finally, it is easy to see that $\age(L_p(\mu))\preceq \age(L_p(0,1))$.   
\end{example}

We continue with the proof of the Fraïssé correspondence. 
Recall that given a sequence $(X_n,\ga_n)_n$ where each $X_n$ is a normed space and $\ga_n\in \mr{Emb}(X_n,X_{n+1})$, the inductive limit  $\lim_n (X_n,\ga_n)$ of $(X_n,\ga_n)_n$ is defined as  the following normed space: We define isometric embeddings $\ga_{m,n}\in \Emb(X_m,X_n)$ for $m\le n$ as follows;    $\ga_{m,m}:= \id_{X_m}$ and $\ga_{m,n+1}:=\ga_n\circ \ga_{m,n}$;    let $V$ be the subspace of the product space $\prod_n X_n$ defined as
$$V:= \conj{(x_n)_n}{ \text{there is some $m$ such that $x_n=\ga_{m,n}(x_m)$ for all $n\ge m$}}.$$
In $V$ one defines the seminorm 
$$\nrm{(x_n)_n}:= \nrm{x_m}_{X_m}$$
where $m$ is such that $\ga_{m,n}(x_m)=x_n$ for all $n\ge m$. It is easy to see that $\nrm{\cdot}$ is well defined, and that 
$$N:=\conj{(x_n)_n\in V}{ \text{there is some $m$ such that $x_n=0$ for all $n\ge m$}}= \conj{(x_n)_n\in V}{ \nrm{(x_n)_n}=0}.$$
Let $V_0$ be the quotient space $V/N$ endowed with the norm $\nrm{\cdot}$.  Finally, let $\lim_n (X_n,\ga_n)$ be the completion of $V/N$. For each $m$, let $\ga_m^{(\infty)}: X_m\to \lim_n (X_n,\ga_n)$ be defined for $x\in X_m$ by 
$$ \ga_m^{(\infty)}(x):= (\overset{(m-1)}{\overbrace{0,\dots,0}}, x, \ga_{m,m+1}x, \ga_{m,m+2}x,\dots)+N.$$
Obviously, $\ga_m^{(\infty)}$ is an isometric embedding, and the sequence $(X^\infty_n)_n$, $X^\infty_n:=\ga_n^{(\infty)}(X_n)$, is increasing with dense union in $\lim_n(X_n,\ga_n)$.  
\begin{lemma}\label{oi788798434} 
Suppose that $ \mc G$ has the \jep. Then $\mc G$ is an amalgamation class  with modulus of stability $\varpi$ if and only if for every  $\De\con \R^+$ finite, $\vep>0$ and $\mc H\cup\{Y\}\con \mc G$ finite there is $Z\in  \mc G$  and some $I\in \Emb(Y,Z)$ such that for every $X\in \mc G$ and every $\de\in \De$  if $\ga,\eta\in \Emb_{\de}(X,Y)$  then there is $J\in \Emb(Y,Z)$ such that $\nrm{I\circ \ga - J \circ \eta}\le \varpi(\dim X,\de)+\vep$.  
\end{lemma}  
\begin{proof}
Fix all data.    Order $\mc H:=\{X_j\}_{j=1}^m$, $\De:=\{\de_{j}\}_{j=1}^{n}$.  For each $ 1\le j\le m$ and $1\le l\le n$, let   $\{\ga_k^{(j,l)}\}_{k=1}^s$ be a   $\vep/3$-dense subsets of $\Emb_{\de_l}(X_j,Y)$.    Inductively we find a sequence $(V_k)_{k=1}^{s^2 m n +1}$ in $ \mc G$, $V_1:=Y$ and isometric embeddings $I_k\in \Emb(V_k, V_{k+1})$, $1\le k\le s^2m n$, such that  for every $  1\le k_0,k_1 \le n$, $ 1\le j\le  m$ and $1\le l\le n$, setting $k:=s^2 (j-1) (l-1)+k_0 k_1 +1$    there is some $J\in \Emb(Y, V_{k+1})$ such that $\nrm{J\circ \ga_{k_0}^{(j,l)}- I_{k}\circ I_{k-1}\circ \cdots \circ I_1 \circ \ga_{k_1}^{(j,l)}   }\le \varpi(\dim X_j, \de_l)+\vep$.  Then $V:=V_{s^2mn+1}$ and $I:=I_{s^2m n}\circ I_{s^2m n-1}\circ \cdots \circ I_1$ work. 
\end{proof}  

\prue	[{\sc Proof of Theorem \ref{oi43hjiio434371}}]  
 The proof is an standard back-and-forth argument. Suppose that $ \mc G$ is an amalgamation class,  with modulus of stability $\varpi$.  
Let $\{\de_n\}_n:=\Q\cap [0,1]$, $\de_0=0$,  $(\vep_n)_n$ be a positive sequence such that $\sum_{n>m}\vep_n\le \vep_m$,   and let   $\mc A=\{Z_n\}_n\con  \mc G$ be a countable $d_\mr{BM}$-dense  subset of $ \mc G$.  We use Lemma \ref{oi788798434} to  find    a sequence $(X_n,I_n)_n$, where 
\begin{enumerate}[(a)]
\item  $X_n\in  \mc G$ and  $I_n\in \mr{Emb}(X_n,X_{n+1})$;
 \item for every $n\in \N$ and every $X\in \{Z_j\}_{j\le n}\cup\{X_j\}_{j\le n}$ and every $k\le n$, if $\ga,\eta\in \Emb_{\de_k}(X,X_n)$ then there is $J\in \Emb(X_n,X_{n+1})$ such that $\nrm{I_{n} \circ \ga - J\circ \eta}\le \varpi(\dim X,\de_k)+ \vep_n$;
 \item $\mr{Emb}(Z_m,X_{n+1})\neq \buit$ for every $m\le n$. 
\end{enumerate}	
For (c) we use the (JEP) of $ \mc G$, that we know is true for amalgamation classes.  We claim that the inductive limit $E:=\lim_n (X_n,I_n)$ is $\mc G$-Fraïssé with modulus $\varpi$, $E\in [\mc G]$    and that $ \mc G\preceq \age(E)$. 
\clam
$E$ is $\mc G$-Fraïssé.
\fclam	
\prucl
Fix $\de'>\de\ge 0$, $\vep>0$, and $X\in \mc G$ and fix $\ga,\eta\in \Emb_\de(X,E)$. 
Choose $n$ large enough such that
\begin{enumerate}[i)]
\item  $\vep_{n-1}<\vep$;
\item there are $j,k\le n$, $\de''<\vep/2$  such that
\begin{enumerate}[{ii}.1)]
\item $\de<\de_k<\de'$ and $(\varpi(\dim X, \de_k) +\vep)\de''<\vep/3$, 
  \item there is an onto map $\theta\in \Emb_{\de''}(X,Z_j)$  and $\ga_0,\eta_0\in \Emb_{\de_k}(Z_j, X_n)$ such that $\nrm{I_n^\infty\circ\ga_0\circ \theta- \ga},\nrm{I_n^\infty\circ\eta_0 \circ \theta -\eta}\le \vep/2$. 

\end{enumerate}  	
   
\end{enumerate}    
For each $r$ and $s$ set $I_r,r:=\id_{X_n}$; $I_{r, r+s+1}:= I_{r+s}\circ I_{r,r+s}$.   We find now $J_s\in \Emb( X_{n+2s},X_{n+2s+1})$  and $ L_s\in \Emb( X_{n+2s+1},X_{n+2s+2})$ such that 
\begin{enumerate}[(a)]
\addtocounter{enumi}{3}
\item $\nrm{J_0\circ \eta_0 -I_n\ga_0}\le \varpi(d, \de_k)+ \vep_n$, being $d:=\dim X$ 
\item  $\nrm{L_{s+1}\circ J_s-I_{n+2s,n+2s+2}}\le \vep_{n+2s+1}$ for every $s\ge 0$.
\item  $\nrm{J_s\circ L_s-I_{n+s-1,n+2s+1}}\le \vep_{n+2s}$ for every $s\ge 1$. 
  \end{enumerate}  
  Setting $\bar{\vep}_n:=\varpi(d,\de_k)+\vep_n$, we have the following  approximately commuting  infinite diagram. 	
\begin{center}
\begin{figure}[h]
\begin{tikzpicture}[descr/.style={fill=white,inner sep=0pt}]

\node at (-1.0,  1.2) (f1c1) {$Z_k$};

\node at (2.0, 1.2) (f1c2) {$X_{n}$};

\node at (5.0, 1.2) (f1c3) {$X_{n+1}$};
\node at (8.0, 1.2) (f1c4) {$X_{n+2}$};
\node at (11.0, 1.2) (f1c5) {$X_{n+3}$};
\node at (14.0, 1.2) (f1c6) {$X_{n+4}$};

\node at (2.0, -1.2) (f2c2) {$X_{n}$};

\node at (5.0, -1.2) (f2c3) {$X_{n+1}$};
\node at (8.0, -1.2) (f2c4) {$X_{n+2}$};
\node at (11.0, -1.2) (f2c5) {$X_{n+3}$};
\node at (14.0, -1.2) (f2c6) {$X_{n+4}$};

\path[->,font=\normalsize]

(f1c1) edge node[above] {$\ga_2$} (f1c2)
(f1c2) edge node[above] {$I_n$} (f1c3)
(f1c3) edge node[above] {$I_{n+1}$} (f1c4)
(f1c4) edge node[above] {$I_{n+2}$} (f1c5)
(f1c5) edge node[above] {$I_{n+3}$} (f1c6)

(f1c1) edge node[above] {$\eta_2$} (f2c2)
(f1c3) edge node[above] {$L_1$} (f2c4)
(f1c5) edge node[above] {$L_2$} (f2c6)

(f2c2) edge node[above] {$I_n$} (f2c3)
(f2c3) edge node[above] {$I_{n+1}$} (f2c4)
(f2c4) edge node[above] {$I_{n+2}$} (f2c5)
(f2c5) edge node[above] {$I_{n+3}$} (f2c6)

(f2c2) edge node[above] {$J_0$} (f1c3)
(f2c4) edge node[above] {$J_1$} (f1c5)

;

\node at (2.0,  .2) (node1) {\scalebox{2.8}{$\circlearrowleft$} };
\node at (2.0,  0.2) (vep) {\scalebox{.8}{$\bar{\vep}_n$} };
\node at (5,  0.2) (node2) {\scalebox{2.8}{$\circlearrowleft$} };
\node at (5,  0.2) (vep2) {\scalebox{.7}{$\vep_{n+1}$} };
\node at (8,  0.2) (node2) {\scalebox{2.8}{$\circlearrowleft$} };
\node at (8,  0.2) (vep2) {\scalebox{.7}{$\vep_{n+2}$} };
\node at (11,  0.2) (node2) {\scalebox{2.8}{$\circlearrowleft$} };
\node at (11,  0.2) (vep2) {\scalebox{.7}{$\vep_{n+3}$} };
\node at (14,  0.2) (vep2) {\scalebox{1.8}{$\cdots$} };

 \end{tikzpicture}
\caption{} \label{oi34ioroi4rwe11}
\end{figure}  
\end{center}

Given  a sequence  $x=(x_k)_k\in \bigcup_l X^\infty_l$, we define for each $k$, $ \overline{y}_k(x):= I_{n+2k+1}^\infty J_k(x_{n+2k})\in X^\infty_{n+2k+1}$. It follows from the (e) and (f) that $(\overline{y}_k(x))_k$ is a Cauchy sequence, so we define $J: \bigcup_l X^\infty_l\to E$, $J(x)=\lim_k \overline{y}_k(x)$, and then we extend it to $J:E\to E$. Similarly, we define, given $y=(y_k)_k\in \bigcup_l X^\infty_l$, one defines the Cauchy sequence  $\overline{x}_k(y):= I_{n+2k+2}^\infty L_k(x_{n+2k+1})\in X^\infty_{n+2k+2}$, and the corresponding $L: E\to E$. It is easy to see that $L\circ J=J\circ L= \id_E$. 

Set now  $\ga_1:=I_{n}^\infty\circ \ga_0$ and  $\eta_1:=I_{n}^\infty\circ \eta_0$. It follows that    $J(\eta_1(z))= \lim_{s} I_{n+2s+1}^\infty J_k(I_{n,n+2s} (\eta_0(z)))$ for every $z\in Z_k$. Now for each $s$ one has that 
$\nrm{J_s(I_{n,n+2s} (\eta_0(z))) - I_{n, n+2s+1}(\ga_0(z)) }\le \varpi(\dim X,\de_k)+\sum_{l\le 2k}\vep_{n+l}\le \varpi(\dim X,\de_k)+ \vep_{n-1}$. This means that  $\nrm{J\circ \eta_1 - \ga_1}\le \varpi(\dim X,\de_k)+ \vep_{n-1}$. Consequently,
\begin{align*}
\nrm{J\circ \eta- \ga}\le & \nrm{J\circ \eta_1\theta- \ga_1\theta}+ \nrm{J\circ \eta_1\theta- J\circ \eta}+\nrm{\ga_1\theta- \ga}\le   (\varpi(\dim X,\de_k)+ \vep_{n-1})\nrm{\theta}+ 2 \frac{\vep}{3} \le \\
\le & \varpi(\dim X, \de_k)+\vep\le \varpi(\dim X, \de')+\vep.   \qedhere
\end{align*}
\fprucl
This proves that $E$ is $\mc G$-Fraïssé with modulus $\varpi^*$. 
 By  definition of $E$, the sequence $( X_n^\infty)_n$ witnesses that $E\in [\mc G]$. 
 \clam
 $\mc G\preceq \age(E)$.
 \fclam
 \prucl
 For suppose that $Z\in \mc G_k$. 
 Recall that $(\vep_n)_n$ was chosen so that $\sum_{n\ge m+1}\vep_n\le \vep_m$. Find a decreasing positive sequence $(\de_n)_n$  such that $\varpi_{\mc G,E}(k,\de_n)\le \vep_n$ for every $n$. Since $Z\in \overline{\mc A}^\mr{BM}$ and $\mc A\con \age(E)_\equiv$,   we can find for each $n$ some $\ga_n\in \Emb_{\de_n}(Z,E)$. Now for each $n$, let $g_n\in \iso(E)$ be such that $\nrm{g_n \circ \ga_{n+1}-\ga_n}\le \vep_n$, and set $\eta_0:=\ga_0$; $\eta_{n+1}:=g_0\circ \cdots \circ g_{n}\circ \ga_{n+1}$.   Then $\nrm{\eta_{n+l}-\eta_{n}}\le \sum_{j=n}^{n+l-1} \vep_j \le \vep_{n-1}$, so $(\eta_n)_n$ is a Cauchy sequence of $\de_n$-embeddings from $Z$ into $E$, so its limit $\eta:Z\to E$ is an isometric embedding, and consequently $Z\in \age(E)_\equiv$.
 \fprucl
\fprue

\subsection{Characterization of homogeneities}\label{characterization_of_homogeneities}
We rephrase the homogeneity properties introduced in terms of algebraic, topological and  metric properties of some list of canonical functions. 
Let  $\mc N_k$ be the collection of all norms on the vector field $\mbb F^k$,
 let $\emb(\mbb F^k,E)$ be the family of all 1-1 linear mappings from $\mbb F^k$ into $E$.
 Given  a collection  $ \mc G$ of normed spaces over $\mbb F$, let  $\mc N_k( \mc G)$ be the collection of those norms $M\in \mc N_k$ such that $(\mbb F^k,M)\in  \mc G_\equiv$;   we define $\emb( \mc G_k):=\bigcup_{M\in \mc N_k( \mc G)}\Emb((\mbb F^k,M),E)$; we consider the natural action $\iso(E)\acts\Emb(X,E)$, $g\cdot \ga:=g\circ \ga$, and the consequent $\iso(E)\acts \Emb( \mc G_k,E)$. We have now      $\widehat{\nu}_{k}: \emb( \mc G_k,E)\to \mc N_k( \mc G)$, defined as the norm
 $\widehat{\nu}_{k}(\ga)(x):=\nrm{\ga x}_E$; obviously, $\widehat{\nu}_{k}$ is $\iso(E)$-invariant, so we have naturally defined the quotient mapping $$\nu_{k}: \emb(\mbb F^k,E)\quo \iso(E) \to \mc N_k(E),$$
  $\nu_{k}([\ga]):=\widehat{\nu}_{k}(\ga)$. In general,  $\nu_k(\emb( \mc G_k,E))\con \mc N_k( \mc G)$, and when $ \mc G\preceq \age(E)$ one has  $\nu_k(\emb( \mc G_k,E))= \mc N_k( \mc G)$

 The following  proposition rephrases the approximate $ \mc G$-homogeneity of $E$ in terms of an algebraic property of each $\nu_k$.

 \prop Suppose that $ \mc G\preceq \age(E)$. Then $E$ is approximately $ \mc G$-homogeneous if and only if each $\nu_{k}$ is 1-1 on $\emb( \mc G_k,E)\quo\iso(E)$.    \qed
 \fprop

The next proposition is a characterization of $E$ being weak $\mc G$-Fraïssé in terms of topological properties of each $\nu_k$. So, 
we naturally endow $\mc N_k$ and $\mr{emb}(\mbb F^k,E)$ with the topology of point-wise convergence, that coincides with the compact-open topology.  The following is easy to prove.

\begin{proposition}\label{oiij34i455}
  Both $\widehat{\nu}_{k}$ and $\nu_{k}$ are continuous.\qed
\end{proposition}  
We introduce two metrics; let  
$$\om(M,N):=\log\left(\max\{     \nrm{\id}_{(\mbb F^k,M),(\mbb F^k,N)},  \nrm{\id}_{(\mbb F^k,N),(\mbb F^k,M)}   \} \right),$$ that defines the point-wise topology on $\mc N_k$; the second one $d(\eta,\ga):=\nrm{\eta-\ga}_{\ell_1,E}:=\max_{j=1}^n \nrm{\eta(u_j)- \ga(u_j)}_E$ defines the pointwise topology on $\emb(\mbb F^k,E)$; notice that $d$  is $\iso(E)$-invariant for the natural action $g\cdot \ga:=g\circ \ga$ introduced above, and  that the quotient metric $\widehat{d}([\ga],[\eta]):=\inf_{g\in \iso(E)} \nrm{g\circ \ga-\eta}_{\ell_1^k,E}$  defines the quotient topology on $\emb(\mbb F^k,E)\quo \iso(E)$.

\prop\label{uyy66773434} Suppose that $ \mc G\preceq \age(E)$. 
\begin{enumerate}[1)]
\item  $E$ is  weak $\mc G$-Fraïssé  if and only if for every $k$ one has that 
$\nu_k$ is 1-1 on $\emb( \mc G,E)\quo\iso(E)$, and for every $\ga\in \emb( \mc G_k,E)\quo\iso(E)$ the mapping $\nu_k$ is continuous and open at $[\ga]$.\footnote{Recall that $f:X\to Y$ is open at $x$ when for every open neighborhood $U$ of $x$ there is some open neighborhood $V$ of $f(x)$ such that $V\con f(U)$.}  

\item Consequently, if  $E$ is weak  $\mc G$-Fraïssé, then    each restriction $\nu_k: \emb( \mc G,E)\quo\iso(E)\to \mc N_k( \mc G)$ is a homeomorphism.
\item $E$ is  \auhp  if and only   $\nu_{k}$ is a homeomorphism for every $k\in \N$.
\end{enumerate}

\fprop
\prue
{\it 1)}: Suppose that $E$ is weak $ \mc G$-Fraïssé, and fix $k\in \N$. We know that $\nu_{k}$ is 1-1, so let  $\ga\in \emb( \mc G,E)$, and we see that $\nu_k$ is open at $[\ga]$. Fix $\vep>0$;  we have to prove that there is some $\de>0$ such that if $\eta\in \emb(\mbb F^k,E)$ satisfies that $\om(\nu_k([\ga]), \nu_k([\eta]))<\de$,  then there is some $g\in \iso (E)$ such that $d(g \eta, \ga)<\vep$.  We set $M:=\nu_k(\ga)$ and  $X:=(\mbb F^k, M)$.   Let   $\xi>0$ be such that $\iso(E)\acts \Emb_\xi(X,E)$ is  $\vep/ \nrm{\id}_{X,\ell_1^k}$-transitive. We claim that $\de:=\log(1+\xi)$ works. For suppose that $\eta\in \emb(\mbb F^k, E)$  is such that   $\om(M,N)<\de$, where $N:=\nu_k(\eta)$. Then $\nrm{\eta}_{X,E}\le \nrm{\eta}_{(\mbb F^k, N),E}  \cdot e^{\om(M,N)}\le (1+\xi)$ and similarly one has that  $\nrm{x}_E\ge (1+\xi)^{-1} N(x)$, so $\eta\in \Emb_\xi(X,E)$. Hence there is some $g\in \iso(E)$ such that 
$\nrm{g\ga-\eta}_{X,E}\le  \vep/ \nrm{\id}_{X,\ell_1^k}$, and consequently $\nrm{g\ga-\eta}_{\ell_1^k,E}\le  \vep$.

Suppose now that for every $k$ one has that $\nu_k$ is $1-1$ on $\emb( \mc G,E)\quo\iso(E)$ and open at each $[\ga]$ with $\ga\in \emb( \mc G,E)$. Fix $X\in  \mc G$ and $\vep>0$;  let $M\in \mc N_k( \mc G)$ be such that $X\equiv (\mbb F^k,M)$ and let $\theta:    (\mbb F^k, M)\to X$ be an isometry, and $\ga\in \emb( \mc G,E)$ be such that $\nu_{k}(\ga)=M$;  let $\de$ be such that
 if $\eta\in \emb(\mbb F^k,E)$ is such that  $\om(M,\nu_{k}(\eta))\le \de$, then $\widehat{d}([\ga],[\eta])<\widehat{\vep}:= \vep/(2 \nrm{\theta^{-1}}_{X,\ell_1^k})$. We claim that $\iso(E)\acts\Emb_\de(X,E)$ is $\vep$-transitive; for suppose that $\eta_0,\eta_1\in \Emb_\de(X,E)$; we set $M_j:=\tau_k(\eta_j \circ \theta  )\in \mc N_k(E)$ for $j=0,1$. Since $\om(M,M_0),\om(M,M_1)\le \de$,  we get that $\widehat{d}([\eta_0\circ\theta], [\ga]   ) , \widehat{d}([\eta_1\circ\theta], [\ga] <\widehat\vep$, hence  $\widehat{d}([\eta_0\circ\theta],[\eta_1\circ\theta])<2\widehat\vep$;   this means that there is some $g\in \iso(E)$ such that $\nrm{g\circ \eta_0 \circ \theta-\eta_1 \circ \theta}_{\ell_1^k,E}<2\widehat\vep$, and hence
 $$\nrm{g\circ \eta_0 -\eta_1 }_{X,E}\le \nrm{g\circ \eta_0\circ \theta ,\eta_1 \circ \theta}_{\ell_1^k,E}\cdot   \nrm{\theta^{-1}}_{X,\ell_1^k}\le \vep.$$
 {\it 2)} and  {\it 3)}  is a direct consequence of {\it 1)} using  that continuous bijection  is a homeomorphism exactly when it is open at each point. 
 \fprue

We pass now to reformulate    $E$ being  $\mc G$-Fraïssé as a metric property of each $\nu_k$.  So, we introduce a metric on $\emb(\mbb F^k,E)$ in a way that bounded sets are sent to $\om$-bounded sets  by the mappings $\nu_k$. Given $\ga,\eta\in \emb(\mbb F^k,E)$, let  
$$\mathfrak{d}(\ga,\eta):= \max \left\{\nrm{\ga-\eta}_{\ell_1^k, E} , \left| \om(\nu_k(\ga), \nrm{\cdot}_{\ell_1^k}) - \om(\nu_k(\eta), \nrm{\cdot}_{\ell_1^k}) \right|\right\};$$
this is an $\iso(E)$-invariant metric  that defines the pointwise convergence topology on $\emb(\mbb F^k,E)$,  and the quotient topology is defined by the quotient metric 
$$\widehat{\mathfrak{d}}([\ga],[\eta]):=\max\left\{\inf_{g\in \mr{Iso}(E)} \nrm{\ga- g\eta}_{\ell_1^k,E},     \left| \om(\nu_k(\ga), \nrm{\cdot}_{\ell_1^k}) - \om(\nu_k(\eta), \nrm{\cdot}_{\ell_1^k}) \right|     \right\}   ;$$
We have the following
\begin{proposition}
 $\nu_k :(\emb(\mbb F^k, E)\quo \iso(E),   \widehat{\mathfrak{d}}) \to (\mc N_k(E),\om)$ is uniformly continuous on bounded sets.
\end{proposition}  
\begin{proof}
 Fix $\ga_0\in \emb(\mbb F^k, E)$ and $K>0$, and let us see that $\nu_k$ is uniformly continuous on the $\widehat{\mathfrak{d}}$-ball with center $[\ga_0]$ and radius $K$. Given $x\in \mbb F^k$   we have that
 \begin{align*}
 \nrm{\eta(x)} \le &\nrm{\ga(x)}+ \nrm{\ga- \eta}_{\nu_k(\ga),E}\cdot \nrm{\ga(x)} \le \nrm{\ga(x)}+ \nrm{\ga- \eta}_{\ell_1^k,E} \cdot \nrm{\id}_{\nu_k(\ga),\ell_1^k} \nrm{\ga(x)}   \\
 \le &\left(1+ \widehat{\mk d}([\ga],[\eta])  \cdot \nrm{\id}_{\nu_k(\ga),\ell_1^k}  \right) \nrm{\ga(x)}
 \end{align*}
and similarly $\nrm{\ga(x)}\le (1+ \widehat{\mk d}([\ga],[\eta])   \cdot \nrm{\id}_{\nu_k(\ga),\ell_1^k} ) \nrm{\eta(x)}$. This means that 
\[\om(\nu_k([\ga]),\nu_k([\eta]))\le \log\left(1+ \widehat{\mk d}([\ga],[\eta])   \cdot \nrm{\id}_{\nu_k(\ga),\ell_1^k} \right)  \le     \nrm{\id}_{\nu_k(\ga),\ell_1^k} \widehat{\mk d}([\ga],[\eta]) \le e^K \widehat{\mk d}([\ga],[\eta]).\qedhere\]
 \end{proof}

This is the metric characterization of  being $\mc G$-Fraïssé. 
\prop \label{wekrwekrewkew}Suppose that $ \mc G\preceq \age(E)$.   
\begin{enumerate}[1)]

\item    $E$ is  $ \mc G$-Fraïssé if and only if  for every $k\in \N$  the mapping  
$\nu_{k}:(\mr{emb}(\mbb F^k,E)\quo \iso(E),\widehat{\mathfrak{d}})\to (\mc N_k(E),\om)$ is   uniformly open  on bounded subsets of $\mc N_k( \mc G)$, i.e., given $A\con \mc N_k( \mc G)$ that is $\om$-bounded and given $\vep>0$, there is some $\de>0$ such that if $\nu_k([\ga])\in A$ and $\eta\in \emb(\mbb F^k,E)$ is such that $\om(\nu_k([\ga]),\nu_k([\eta]))<\de$, then $\widehat d ([\ga],[\eta])<\vep$.

 Consequently, if $E$ is $ \mc G$-Fraïssé   then  every  restriction $\nu_{k}:(\mr{emb}( \mc G,E)\quo \iso(E),\widehat{\mathfrak{d}})\to (\mc N_k( \mc G),\om)$ is a homeomorphism  that is uniform when restricted to bounded sets. 
 \item  $E$ is  Fraïssé  if and only   for every $k\in \N$ one has that $\nu_{k}:(\mr{emb}(\mbb F^k,E)\quo \iso(E),\widehat{\mathfrak{d}})\to (\mc N_k(E),\om)$   is a homeomorphism  that is uniform when restricted to bounded sets.

\end{enumerate}
\fprop
\prue
{\it 1)} is proved following the same ideas as in the proof of Proposition \ref{uyy66773434} {\it 1)}.  {\it 2)} readily follows from {\it 1)}.   
\fprue

\section{Lattice and $\ell_p^n$-homogeneity}\label{oi4iuuuuttyy}
In the next Section we will prove that $L_p(0,1)$ for $p$ not even are Fraïssé Banach spaces. As we mentioned before this not the case for $p=4,6,8,\dots$. However ``partial'' homogeneity properties are valid for all $L_p(0,1)$. This is the case when dealing with partial isometries defined on $\ell_p^n$'s. The aim of this section is to study those. In trying to understand these embeddings and isometries of $L_p(0,1)$ the setting of {\em lattices} is very natural.  In this context we consider lattice isometries and lattice embeddings on a Banach lattice $X$, i.e., linear isometries or isometric embeddings which preserve the lattice structure. In the important case of a finite dimensional sublattice $F$, i.e. $F=\langle f_j\rangle_{j<n}$ where the $(f_j)_j$'s is a pairwise disjoint sequence of positive elements of $X$, a lattice isometric embedding is an isometric embedding $T$ such that  $(T(f_i))_j$ is a pairwise disjoint sequence of positive elements. For a complete information about them we refer the reader to \cite[Vol. 2]{LiTza}.   

 With this in mind, we may consider lattice versions of homogeneity, by replacing isometries (respectively isometric embeddings) by lattice isometries (resp. lattice isometric embeddings) and (finite dimensional) subspaces by (finite dimensional) sublattices. So, given a sublattice $Y$ of $X$, let $\Emb_\diamond(Y,X)$ be the space of lattice isometric embeddings from $Y$ into $X$,  let $\iso_\diamond(X)$ be the topological group of surjective lattice isometries on $X$ with the SOT, and let $\iso_\diamond(X)\acts \Emb_\diamond(Y,X)$ be the canonical action $g\cdot \ga:= \ga \circ g^{-1}$. In particular we shall be interested in the following definitions:

\defi \label{oi23io3444}Let $X$ be a Banach lattice.
\begin{enumerate}[(a)] 
\item  $X$ is {\em lattice ultrahomogeneous} \latuh if  $\iso_\diamond(X)\acts \Emb_\diamond(F,X)$ is transitive for every finite dimensional sublattice $F$ of $X$;
\item  $X$ is {\em approximately lattice ultrahomogeneous }\latauh  if $\iso_\diamond(X)\acts \Emb_\diamond(F,X)$ is $\vep$-transitive for every finite dimensional sublattice $F$ of $X$.
\end{enumerate} 
\fdefi

There is a third natural notion in the case of lattices which is as follows.
We say that an isometry (respectively isometric embedding) is disjoint preserving (or d.p.) if it sends disjoint vectors to disjoint vectors, and that a subspace is disjointly generated if it is generated by a sequence of disjoint vectors. So, the difference between these versions and the lattice ones is that now we do not impose positivity. Then a Banach lattice $X$ is {\em disjointly homogeneous} if any d.p. isometric embedding defined on a  disjointly generated finite dimensional subspace of $X$ extends to a global d.p. isometry. 
It is {\em approximately disjointly homogeneous} if for any d.p. isometry $t$ defined on a disjointly generated finite dimensional subspace $F$  of $X$ and for any $\epsilon>0$, there is a global d.p. isometry $T$
 such that $\|T\rest {F}-t\| \leq \epsilon$. These properties will not be as relevant as the previous ones because of the following observation.

\begin{proposition}\label{latt-disj} Assume $X$ is a K\"othe function   space (see \cite[Vol. 2]{LiTza}). If $X$ is
 (approximately) lattice homogeneous then it is 
(respectively approximately) disjointly homogeneous. 
\end{proposition}

\begin{proof} Assume $X$ is lattice homogeneous. Let $F=\langle f_j\rangle_{j<n}$ with $(f_j)_j$ pairwise disjoint, and $t$ an isometric d.p. map from $F$ onto $G=\langle tf_j\rangle_{j<n}$. There exist isometric maps $u$ and $v$ on $X$ acting by changes of signs, and sending each $f_j$ to $|f_j|$, and $tf_j$ to $|tf_j|$, respectively.
The map $vtu^{-1}$ is a lattice isometric map from $\langle |f_j|\rangle_{j<n}$ onto
$\langle |tf_j|\rangle_{j<n}$. If $T$ is a lattice isometry  on $X$ extending
$vtu^{-1}$ then $v^{-1}Tu$ is a d.p. isometry on $X$ extending $t$. The same proof holds
for approximate lattice homogeneity.
\end{proof}

\subsection{Stable approximate $\{\ell_p^n\}_n$-homogeneity of $L_p(0,1)$}

 We start with a few classical definitions. Let $1\le p<+\infty$.
A {\em simple} space is a finite dimensional subspace of $L_p(0,1)$ generated by simple functions. By {\em copy of $\ell_p^n$} we mean some linearly isometric copy of a finite dimensional $\ell_p^n$ inside $L_p(0,1)$. It is a classical result by S. Banach that if $1\le p<+\infty, p \neq 2$, any isometric embedding between $L_p$-spaces is automatically d.p. (\cite[Theorem 3.2.5]{FJ1}). In particular, any isometric embedding of  a copy of $\ell_p^m$ into a copy of $\ell_p^n$ is disjoint preserving, and any copy of $\ell_p^n$ is a disjointly generated subspace.
It is also an easy observation that, for $1\le p<+\infty$, any finite dimensional subspace generated by simple functions
in $L_p(0,1)$ is a  subspace of 
some simple copy of some $\ell_p^n$ (with same support).
S. Banach, \cite[pp. 178]{B}, stated the general formula of surjective isometries
on $L_p(0,1)$, $1 \leq p<+\infty, p \neq 2$; any such isometry $T$ is defined by
$$T(f)(t)=h(t) f(\phi(t)) \text{ for every $t\in [0,1]$},$$ where $\phi$ is a measurable non-singular transformation of $[0,1]$ onto itself, and $h$ is a function satisfying $|h|^p=d(\lambda \circ \phi)/d\lambda$, where $\lambda$ is the Lebesgue measure (see   \cite[Chapter 3]{FJ1}). The    description of isometric embeddings between $L_p$-spaces, called   the Banach-Lamperti Theorem,  and  a complete proof of it  was done by J. Lamperti in \cite{Lam}.  When $p=2$ the above formula defines a linear isometry but there are, obviously, other isometries on $L_2(0,1)$. 

Let $1\le p<+\infty$. Recall that a simple function of $L_p(0,1)$ is a measurable mapping having finitely many values. We shall say that
an isometry on $L_p(0,1)$ is {\em simple preserving} if it maps simple functions to simple functions.
If $A$ is a Lebesgue measurable subset of $[0,1]$ of positive measure, we denote by $L_p(A)$ the $p$-Lebesgue space associated to the measure space $(A, \mc L(A), \la\rest A)$ consisting of restricting the Lebesgue measure $\la$ to the Lebesgue measurable subsets of $A$.  In addition, if  $F$ is a subspace of  $L_p(A)$, then the support of $F$ is the union of the supports of all $f$, $ f\in F$; in particular, for a vector in, or a subspace of $X$,  {\em full support} in    $L_p(A)$ means that the corresponding support   is equal to  the support of $X$.

\lema\label{singleisometry} Let $1\le p<+\infty$. Let $A,B$ be subsets of $[0,1]$ of positive measure, and let $u \in L_p(A)$, $v \in L_p(B)$ be functions with full support and  with equal norms. Then there exists a (disjoint preserving) isometry $T$ from $L_p(A)$ onto $L_p(B)$ such that $T(u)=v$. 
Furthermore if $u$ and $v$ were simple functions then $T$ is simple preserving, and if $u$ and $v$ were non-negative then $T$ is a lattice isometry.\flema

\prue It is enough to show this for $A=[0,1]$ and  $u=\mathbbm 1_{[0,1]}$. 
Furthermore using a natural isometry between $L_p(B)$ and $L_p(0,1)$ we may  
assume $B=[0,1]$.
Since $|v|>0$, It is then clear that if $V(x)=\int_0^x |v(t)|^p dt$, then  $$T_v(f)(x)=f(V(x)) v(x) \text{ for every $x\in [0,1]$} $$
defines a linear isometry of $L_p(0,1)$ sending $u$ to $v$.   The fact that this map is disjoint preserving (and respectively simple preserving, a lattice isometry) is obvious from the definition of $T_v$. \fprue

It is well-known that when $1\le p <+\infty, p \neq 2$, any isometry of $L_p(0,1)$ sends full support vectors to full support vectors see \cite[Theorem 3.2.2]{FJ1}. 
This proves that
$L_p(0,1)$ cannot be ultrahomogeneous in this case, nor even lattice ultrahomogeneous  (so, as we commented before, the isometry group does not even act transitively on the unit sphere of $L_p(0,1)$ when $p \neq 2$). The next lemma will be useful to deal with perturbations in this context.

\lema\label{approx}  Let $1 \leq p<+\infty$. Let $F$ be a finite dimensional subspace of $L_p(0,1)$ and let $\vep>0$. 
Then there exists a linear, disjoint preserving, isometric embedding $\ga$ of $L_p(0,1)$ into itself, whose image does not have full support, and such that $\|(\ga-\id)\rest {F}\| \leq \vep$. 
If $F$ is a sublattice then $L$ can be chosen to be a lattice isometric embedding.\flema

\prue Write $F=\langle f_j\rangle_{j<n}$, where the $f_i$'s have norm $1$. It is enough to find some isometric embedding $\ga$ whose image does not have full support and
such that $\|(\ga-\id)(f_j)\| \leq \vep$ for all $j$ and for some $\vep$ small enough. By a perturbation argument we may also assume that each $f_j$ is continuous. Let $M={\rm max}_{1 \leq j \leq n}\nrm{f_j}_\infty$, and let $\delta(s)$ be a common modulus of uniform continuity of all $f_j$'s, i.e.  such that $|x-y| \leq \delta(s)$ implies $|f_j(x)-f_j(y)| \leq s$ for every $ j=1,\dots,n$. Let $\lambda<1$ be close enough to $1$ so  that
$\lambda \big( 1/\lambda-1)M+\delta({1}/{\lambda}-1)\big)^p+(1-\lambda)M^p <\vep$.  
We define an isometric embedding $\ga$ whose image is $L_p(0,\lambda) \subset L_p(0,1)$ by 
$$\ga(f)(x):= \mathbbm 1_{[0,\la]}(x) f(\frac{x}{\lambda}).$$
Then for $j=1,\ldots,n$ one has that 
\[
\|\ga(f_j)-f_j\|^p=\int_0^{\lambda} |\frac{1}{\lambda} f_j(\frac{x}{\lambda})-f_j(x)|^p dx
+\int_{\lambda}^1 |f_j(x)|^p dx \leq  \lambda \big( (\frac{1}{\lambda}-1)M+\delta(\frac{1}{\lambda}-1)\big)^p+(1-\lambda)M^p<\vep.\qedhere
\]
	
\fprue

\prop\label{ultracopylp} For every $1\le p<\infty$ the Banach lattice $L_p(0,1)$ is  approximately lattice homogeneous, or in other words, the action $\iso_\diamond(L_p(0,1))\acts \Emb_\diamond(\ell_p^n,L_p(0,1))$ is almost transitive. 
 \fprop

\prue  We start with the following. 
\clam \label{liwjriwejriwe}
Fix $n\in \N$, and let  $\mc E_{F}$ and $\mc E_{N}$ be the families of lattice embeddings of $\ell_p^n$ into $L_p(0,1)$ whose image has full support and does not have full support, respectively.   Then the canonical actions $\iso_\diamond(L_p(0,1))\acts \mc E_F$ and $\iso_\diamond(L_p(0,1))\acts \mc E_N$ are transitive. 

\fclam
\prucl
We prove only that $\iso_\diamond(L_p(0,1))\acts \mc E_F$ is transitive; the other proof is similar, and we leave the details to the reader.  Let $F$ and $G$ be two copies of some $\ell_p^n$, with full support, and
let $t$ be an isometry from $F$ onto $G$. 
Write $F=\langle f_j\rangle_{j<n}$, where the $f_i$'s are normalized disjointly supported, and let
$A_i:=\supp f_i$. Likewise we define $G=\langle tf_j\rangle_{j<n}$ and $B_i=\supp tf_i$.
By Lemma \ref{singleisometry}, we may for each $i$ define an isometry $T_i$ from $L_p(A_i)$ onto
$L_p(B_i)$ sending $f_i$ to $tf_i$. This defines a global linear isometry $T$ on $L_p(0,1)
=\oplus_i L_p(A_i)=\oplus_i L_p(B_i)$ which extends $t$.
\fprucl

 We use the above notation, where $t$ is an isometric embedding of $F$ onto $G$ in $L_p([0,1])$, and assume for example that $F$ does not have full support and $G$ has. By Lemma \ref{approx}, there exists an isometric embedding $t'$ of $F$ onto some $G'$ without full support with $\|t-t'\| \leq \epsilon$. By the Claim, $t'$ extends to a global isometry on $L_p(0,1)$ and we are done.
\fprue

\nota\label{oio3ui23343r4353}
Note that by Lemma \ref{singleisometry} the isometry  $T$ considered in the proof of Claim \ref{liwjriwejriwe}  is simple preserving when the $f_j$'s and $t f_j$ are  simple functions. 
\fnota

We conclude with the following very strong amalgamation property. 

\prop\label{3l4kirjio34riji43} Let $1\le p<\infty$. Then
\begin{enumerate}[1)]
\item the class $\{\ell_p^n\}_{n\in \N}$ is an amalgamation class with modulus independent of the dimension. 
\item $L_p(0,1)$ is the Fraïssé limit of $\{\ell_p^n\}_n$.
  
\end{enumerate}  

\fprop
For the proof we use the following remarkable result by G. Schechtman  (as observed by D. Alspach \cite{Als}).
\begin{theorem}[Schechtman \cite{Sch}]\label{kdmfksmfkdsw}
For any $1 \leq p<\infty$ there is a modulus of stability $\varpi_p:]0,\infty[\to ]0,\infty[$ such that    
$$\Emb_{\de}(\ell_p^n, L_p(\mu))\con (\Emb(\ell_p^n, L_p(\mu)))_{\varpi_p(\de)}.$$
for every $n\in \N$, $\de>0$ and finite measure $\mu$.   Consequently, for every $d,m\in \N$ there is $n\in \N$ such that 
$$\iota_{m,n}\circ \Emb_{\de}(\ell_p^d,\ell_p^m)\con (\Emb(\ell_p^d,\ell_p^n)_{\varpi_p(\de)},$$
where $\iota_{m,n}:\ell_p^m\to \ell_p^n$ is the canonical isometric embedding $(a_j)_{j<m}\in \mbb F^m\mapsto (a_0,\dots,a_{m-1},0,\dots,0)\in \mbb F^n$. 
\end{theorem}  
\prue[{\sc Proof of Proposition \ref{3l4kirjio34riji43}}]
{\it 1)}: We may assume $p \neq 2$ and let $\mc F$ be the collection of finite dimensional simple subspaces of $L_p(0,1)$ without full support. 
\clam	
$\mc F$ has the amalgamation property. Consequently, $\{\ell_p^n\}_n$ has the amalgamation property.
\fclam
This, combined with Schechtman's result readily gives the proposition. 
\prucl
Let $E,F,G\in \mc F$, and let $\ga\in \Emb(E,F)$, $\eta\in \Emb(E,G)$. By the Claim \ref{liwjriwejriwe} (see the Remark \ref{oio3ui23343r4353}) there is a simple preserving $T\in \iso_\diamond(L_p(0,1))$ such that $T\circ \ga=\eta$.   Let $J\in \Emb(G,L_p(0,1)$ be simple preserving and  such that $TF+ J G$ does not have full support,  and let $V\in \mc F$ be  containing $T F+ J G$, let $I: F\to V$ be the restriction $I:=T\rest F$. Then clearly $I\circ \ga=J\circ \eta$.

{\it 2)}: From {\it 1)} and the Proposition \ref{ultracopylp}, we obtain that $L_p(0,1)$ is $\{\ell_p^n\}_n$-Fraïssé. Since $L_p(0,1)\in [\{\ell_p^n\}_n]$, we obtain from Theorem \ref{oi43hjiio434371} that $L_p(0,1)=\lim\{\ell_p^n\}_n$. 
\fprucl
\fprue

\section{The \sauhp property of $L_p(0,1)$: Approximate equimeasurability in $L_p$ spaces} \label{oi43894589754789}
The main result of this section is the following. 
\begin{theorem}\label{auh_for_L_p}
If $1\le p<\infty$, $p$ not even, then $L_p(0,1)$ is \sauhp.
\end{theorem}
Lusky's  proof of the approximate ultrahomogeneity of the spaces $L_p(0,1)$ for  $1\le p <\infty$, $p\notin 2\N$ is a consequence of the following result, known as the {\em equimeasurability principle}, proved independently by A. I. Plotkhin \cite{Plo1}   and  W. Rudin  \cite{Ru}. For the field $\mbb F=\R,\C$, recall that given a finite measure space $(\Om, \Sig,\mu)$,  we denote by $L_p(\Om,\Sig,\mu;\mbb F)$, or simply $L_p(\mu)$, the Lebesgue space of all $p$-integrable functions $f: \Om\to \F$.  
\begin{theorem}[Equimeasurability principle] \label{oi4jiorjeiwrew44}
Suppose that $1\le p<\infty$,  $p\notin 2\N$, $(\Om_0, \Sig_0,\mu_0)$, $(\Om_1, \Sig_1,\mu_1)$ are finite measure spaces, and $f_0,\dots,f_{n-1}\in L_p( \mu_0)$ and  $g_0,\dots,g_{n-1}\in L_p(\mu_1)$ are such that  
 $$\int  \left| 1 + \sum_{j<n} a_j f_j(\om) \right|^p d\mu_0(\om)= \int  \left| 1 + \sum_{j<n} a_j g_j(\om) \right|^p d\mu_1(\om)$$
  for every scalars $a_0,\dots, a_{n-1}\in \mbb F$. 
Then, $F=(f_j)_j$ and $G=(g_j)_j$ are   {\em jointly equidistributed}; that is, for every Borel subset $A\con \mbb F^n$, 
$$F_*\mu_0(A)=\mu_0(\conj{\om\in \Om_0}{(f_j(\om))_{j<n}\in A})=  G_*\mu_1(A)=\mu_1(\conj{\om\in \Om_1}{(g_j(\om))_{j<n}\in A}).$$
\end{theorem}
In other words,  if $\nrm{ \mathbbm{1}_{\Om_0} +\sum_j a_j f_j}_{L_p(\mu_0)}=\nrm{ \mathbbm{1}_{\Om_1} +\sum_j a_j g_j}_{L_p(\mu_1)}$  
 for every scalars $a_0,\dots, a_{n-1}\in \mbb F$, then   the {\em pushforward} Borel measures $F_* \mu_0$ and $G_*\mu_1$ are equal. Our demonstration relies on an  extension of the equimeasurability principle.  In order to state it we will use the well-known  {\em L\'evy-Prokhorov metric}. Let $(X,d)$ be a separable metric space.  Let $\mc B(X)$ be the $\sig$-algebra of Borel subsets of $X$, and let $\mc M(X)$ be the collection of Borel measures on $X$.   
\begin{definition}[Lévy-Prokhorov metric]
The {\em Lévy–Prokhorov metric} $ d_{\mc{LP}} :{\mathcal  {M}}(X)^{{2}}\to [0,+\infty )$ is defined by setting the distance between two finite  measures $\mu$  and $ \nu$  to be
$$ d_{\mc{LP}} (\mu ,\nu ):=\inf \left\{\varepsilon >0~|~\mu (A)\leq \nu (A_{{\varepsilon }})+\varepsilon \ {\text{and}}\ \nu (A)\leq \mu (A_{{\varepsilon }})+\varepsilon \ {\text{for all}}\ A\in {\mathcal  {B}}(X)\right\}.$$
\end{definition} 
This metric defines the complete convergence on $\mc M(X)$ (see \cite[Section 6]{Bi}).   
 We are in disposition to present the continuity statement associated to the equimeasurability principle. 
In the next,  $z_j: \mbb F^n\to \mbb F$ is the $j^\mr{th}$-projection mapping $(a_0,\dots,a_{n-1})\mapsto a_j$, and given a measure space $(\Om,\Sig,\mu)$ and  $\vphi\in L_p(\mu)$ positive, let $\vphi d\mu$ be the  measure on $\Sig$ having density $f$ with respect to $\mu$, that is, $(\vphi d\mu)(A):=\int_A \vphi d\mu$.

\begin{theorem}[Approximate equimeasurability principle]\label{oi4jiri4rjer}
Suppose that $1\le p<\infty$, $p\notin 2\N$,  and $(\Om_0, \Sig_0,\mu_0)$ is a  finite measure space. Then for every $\vep>0$, $I\con [0,p]$ finite and  $f_0,\dots,f_{n-1}\in L_p(\mu_0)$ there is $\de>0$ such that if  $(\Om_1,\Sig_1,\mu_1)$ is a  finite measure space and $g_0,\dots,g_{n-1}\in L_p(\mu_1)$ are such that 
$$\frac1{1+\de}\left\|  \mathbbm{1}_{\Om_1} +\sum_j a_j g_j\right\|_{L_p(\mu_1)}\le \left\| \mathbbm{1}_{\Om_0} +\sum_j a_j f_j\right\|_{L_p(\mu_0)}\le (1+\de)\left\|  \mathbbm{1}_{\Om_1} +\sum_j a_j g_j\right\|_{L_p(\mu_1)}$$ for every scalars $a_0,\dots, a_{n-1}$, then, 
\begin{equation*}  
  \max_{\al\in I, j<n}d_{\mc{LP}}(|z_j|^\al d(F_*\mu_0), |z_j|^\al d(G_*\mu_1))< \vep,
\end{equation*}  
that is, for every $\al\in I$, every $j<n$, and every Borel subset  $A\con \F^n$,
\begin{align*}
   \int_{F^{-1}(A)} |f_j(\om)|^\al d\mu_0(\om)\le &  \int_{G^{-1}(A_\vep)} |f_j(\om)|^\al d\mu_1(\om) +\vep  \text{ and } \\
    \int_{G^{-1}(A)} |f_j(\om)|^\al d\mu_1(\om) \le  &  \int_{F^{-1}(A_\vep) } |f_j(\om)|^\al d \mu_0(\om) +\vep ,
\end{align*}  
where, as before,  $F=(f_j)_{j<n}$ and $G=(g_j)_{j<n}$.
\end{theorem}
Here we are using the standard euclidean metric on $\mbb F^n$, and  the fact  that  $F_* \mu_0,G_*\mu_1\in \mc M(\mathbb F^n)$. Also, we are using that  that $\int \vphi(z) d(F_*\mu)(z)= \int \vphi(F(\om))d\om$ for every positive Borel function $\vphi$, that can be easily proved by first assuming that $\vphi$ is a characteristic function of a Borel subset of $\mbb F^n$, use linearity of the integral, and then use the monotone convergence theorem.

 A  well-known consequence of the equimeasurability principle is that if $\ga: X\con L_p(\mu_0)\to L_p(\mu_1)$ is an isometric embedding, and $u\in X$ has {\em full support} in $X$, then $\ga u$ has full support in  $\ga X$. Similarly,  Theorem \ref{oi4jiri4rjer} will provide   the corresponding approximate result. Given  $\vep>0$, we say that $f\in X$ has $\vep$-full support if $\nrm{P_{\{u=0\}}\rest X}\le \vep$. We have then the following.
\begin{theorem}[Approximate full support]\label{kjnnjguytfyyuoipo}
Suppose that $X$ is a finite dimensional subspace of $L_p(\mu)$ with $1\le p<\infty$, $p\notin 2\N$, $u\in X$ has full support in $X$ and suppose that $\vep>0$. Then there is some $\de>0$ such that if $\ga\in \Emb_\de(X, L_p(\mu))$ then   $\ga(u)$ has $\vep$-full support in $\ga (X)$.  
\end{theorem}
The study of the approximate equimeasurability will be done in the next Subsection \ref{approx_equim}, first by finding the topological correspondence determining the convergence of $p$-characteristics (Theorem \ref{continuity_theorem}), and finishing with  its metric version in terms of the Lévy-Prohorov metric (Theorem \ref{quantitative_continuity}).  Before, we   present how to  use Theorem \ref{oi4jiri4rjer} and Theorem \ref{kjnnjguytfyyuoipo} to prove the \sauhp property of $L_p(0,1)$ for $1\le p<\infty$, $p\notin 2\N$.  
\subsection{The proof of Theorem \ref{auh_for_L_p}} 
We have already seen that $\age(L_p(0,1))$ is always compact, so we just need the weak-Fraïssé property of $L_p(0,1)$ for $1\le p<\infty$, $p\notin 2\N$. Let us explain the strategy of the proof.  Suppose that $X\in \age(L_p(0,1))$ and  $\ga: X\to L_p(0,1)$  is a  $\de$-isometric embedding, $Y:=\im \ga$. Suppose that $\mathbbm 1_{[0,1]}\in X$ and  choose a basis  $(f_j)_{j<n}$  of $X$.  In Definition \ref{io43ioio3448b} we introduce appropriate partitions $\mc R$ of $\mbb F^n$ whose pieces are products $\prod_j I_j$ where each $I_j\con \F$  either has small  diameter or is an unbounded set such that $\int_{I_j} |f_j|^p$ is small enough.  By considering pullback partition $F^{-1} (\mc R)$, $F:[0,1]\to \mbb F^n$, $x\mapsto F(x):=(f_j(x))_{j<n}$,   we can naturally  almost embed $X$ into the finite dimensional $L_p$-space $\langle \mathbbm 1_{F^{-1}R}\rangle_{R\in \mc R}$ (Proposition  \ref{koemoirmwiomrwe}).  By the approximate equimeasurability principle, $Y$ is almost embedded into   $\langle \mathbbm 1_{G^{-1}(R)}\rangle_{R\in \mc R}$ and the mapping $\mathbbm 1_{F^{-1}(R)}\mapsto\mathbbm 1_{G^{-1}(R)}$  linearly defines   an almost isometry that almost extends $\ga$, and that can be easily perturbed to become a surjective isometry $I$ (Lemma \ref{ksmfkwewe}). Now we use the extension result for isometric embeddings defined on $\ell_p^m$ (Proposition \ref{ultracopylp}) to find the isometry almost extending $\ga$.  If  $\mathbbm 1_{[0,1]}\notin X$, then we choose some normalized $u\in X$ of full support in $X$, and by the approximate full support principle we know that $\ga(u)$ will have approximate full support of $Y$. We can now rotate both $u$ and $\ga(u)$ to $\mathbbm 1$, and use the case when $X$ was unital.  We start by introducing appropriate partitions. 

\begin{definition}[Appropriate partitions]\label{io43ioio3448b}
Let  $F=(f_j)_{j<n}$  be a sequence of functions in $L_p(\Om,\Sig,\mu)$, $0<\vep\le 1$, and let  $K>0$   such that  $\max_j\int_{|f_j| \ge K} |f_j|^p < \frac{\vep^p}3$. 
We say that a finite partition $\mc R$ of $\mbb F^n$ is   {\em $(\vep,K)$-
appropriate} for $F$ when  each $P\in \mc R$ is of the form $P= \prod_{j<n} I_j$  where   each $I_j$ is either  a Borel subset $I_j\con \mbb F$ of  diameter strictly less than $\vep/(3\nrm{\mu})^{1/p}$, or else equal to $\mbb F\setminus B(0,K)$.   
\end{definition}
Associated to such partition $\mc R$, we consider the finite partition $\mc P:=F^{-1}(\mc R)=\conj{F^{-1} R}{R\in \mc R}$ of $\Om$, and the corresponding {\em conditional expectation} $\mbb E(\cdot; \mc P):L_p(\Om,\Sig,\mu)\to L_p(\Om,\mc P,\mu)$, 
 $$\mbb E(f; \mc P)= \sum_{\mu(F^{-1}(R))>0}\left( \frac{1}{\mu_0(F^{-1}(R))}\int_{F^{-1}(R)} f d\mu_0\right)  \mathbbm 1_{F^{-1}(R)}.$$
It is a consequence of Jensen's inequality that   the conditional expectations is a norm one projection.   Given an $(\vep, K)$-appropriate partition $\mc R$ for $F:=(f_j)_{j<n}$, let 
\begin{align*}
 \mc R_+^F:=&\conj{R\in \mc R}{F_*\mu(R)>0}, \\
 \mc R_0^F:=& \mc R^F\setminus \mc R_+^F,\, R_0^F:=\bigcup \mc R_0^F\text{ and for each $k<n$,}\\
 \mc R_{+,k}^F:=&\conj{R\in \mc R_+^F}{\pi_k(R) \text{ is bounded}},  \, R_{+,k}^F:=\bigcup \mc R_{+,k}^F  \text{ and} \\ 
 \mc R_{\infty,k}^F:=&\mc R_0^F\setminus \mc R_{+,k}^F, \, \text{and } R_{\infty,k}^F:=\bigcup \mc R_{\infty,k}^F.
\end{align*}  

\begin{proposition}\label{koemoirmwiomrwe} Suppose that $\mc R$ is $(\vep,K)$-appropriate for $F=(f_j)_{j<n}$. Then, for every $k<n$,
\begin{enumerate}[(a)]
\item  $\nrm{ \mathbb E(f_k;\mc P)\rest F^{-1}(R_{\infty,k}^F)   }_{L_p(\mu)}^p\le \nrm{f_k\rest F^{-1}(R_{\infty,k}^F)}_{L_p(\mu)}^p<\vep^p/3$,

\item  $\nrm{ \mathbb E(f_k;\mc P)\rest R_{+,k}^F - f_k\rest R_{+,k}^F }_{L_p(\mu)}^p\le \vep^p/3$, 
\item $\nrm{f_k-\mbb E(f_k; \mc P)}_{L_p(\mu)}\le \vep$.  

\end{enumerate}

\end{proposition}
\prue
The proof is standard. To simplify the notation we will avoid the superindex ${}^F$. 
Fix $k<n$.  The fact that $\nrm{ \mathbb E(f_k;\mc P)\rest F^{-1}(R_{\infty,k})   }_{L_p(\mu)}^p\le \nrm{f_k\rest F^{-1}(R_{\infty,k})}_{L_p(\mu)}^p$ follows from Jensen's inequality; Now observe that $R_{\infty,+}\con \{|f_k|\ge K\}$, hence,
$$\nrm{f_k\rest F^{-1}(R_{\infty,k})}_{L_p(\mu)}^p\nrm{f_k\rest \{|f_k|\ge K\} }_{L_p(\mu)}^p <\frac{\vep^p}{3}.$$ 
Given $R\in \mr R_{+,k}$, and given 
  $\om\in F^{-1}C$,
$$\left|\frac{1}{\mu(F^{-1}(R))}\int_{F^{-1}(R)} f_k d\mu - f_k(\om)\right|\le \frac{1}{\mu(F^{-1}(R))}\int_{F^{-1}(R)}\left| f_k(\al) - f_k(\om)\right|d\mu (\al)\le \diam(I_k)\le \frac{\vep}{\sqrt[p]{3\nrm{\mu}}}.
$$ 
Hence, 
 $$\int_{F^{-1}R} \left| \frac{1}{\mu(F^{-1}(R))}\int_{F^{-1}(R)} f_k d\mu   -f_k(\al)\right|^pd\mu(\al)\le \frac{\vep^p}{3\nrm{\mu}}\mu(F^{-1}(R)).$$
Putting all together,
\begin{align*}
\nrm{ \mathbb E(f_k;\mc P)\rest R_{+,k} - f_k\rest R_{+,k} }_{L_p(\mu)}^p=& \sum_{R\in \mc R_{+,k}} \int_{F^{-1}R} \left| \frac{1}{\mu(F^{-1}(R))}\int_{F^{-1}(R)} f_k d\mu   -f_k(\al)\right|^pd\mu(\al)  \le \\
 \le & \frac{\vep^p}{3\nrm{\mu}} \sum_{\mu(F^{-1}(R))>0,\, \pi_k(R)\neq B(0,K)^c} \mu(F^{-1}(R))  \le \frac{\vep^p}3.
 \end{align*} 
 {\it(c)} follows easily from {\it(a)} and {\it(b)}.
\fprue
Recall that $z_j: \mbb F^n\to \mbb F$ is the $j^\mr{th}$-projection mapping $(a_0,\dots,a_{n-1})\mapsto a_j$, and that  given $\vphi\in L_p(\mu)$ positive,  $\vphi d\mu$ is  the measure on $\Sig$ defined by $(\vphi d\mu)(A):=\int_A \vphi d\mu$.   The following is easy to prove.
\begin{proposition}
Suppose that $\mc R$ is $(\vep,K)$-appropriate for $F$. There is $\de>0$ such that  if $G=(g_j)_j$ in $L_p(\nu)$ is such that $\max_j d_{\mc LP}( |z_j|^p d(F_*\mu), |z_j|^pd(G_*\nu))\le \de$, then $\mc R$ is also $(\vep,K)$-appropriate for $G$.  \qed
\end{proposition}

For the next lemma, recall that a $\mu$-measurable set is of {\em $\mu$-continuity} when $\mu(\partial A)=0$, where $\partial A$ is the topological boundary. 

\begin{lemma}\label{ksmfkwewe}
For every $X\con L_p(\mu_0)$ finite dimensional containing $\mathbbm 1_{\Om_0}$ and every $\vep>0$ there are $\de>0$,  $\mc E(X;\vep)\con L_p(\mu_0)$ and $\xi_{X,\vep}: X\to \mc E(X;\vep)$ such that
\begin{enumerate}[(a)]
\item $\mc E(X;\vep)$ is isometric to some $\ell_p^m$;
\item $\nrm{\xi_{X,\vep}-i_X}_{\mu_0}\le \vep$;
\item for every  unital $\de$-isometric embedding $\ga: X\to L_p(\mu_1)$ there is an isometric embedding $I: \mc E(X;\vep)\to L_p(\mu_1)$ such that $\nrm{I\circ \xi_{X,\vep}-\ga}\le \vep$. 
 
\end{enumerate} 
We will say that the pair $(\mc E(X;\vep), \xi_{X,\vep})$ is an {\em $\vep$-envelope of $X$}. 
\end{lemma}
\prue
 Fix $\vep_0$, and let $F:=(f_j)_{j<n}$ be an Auerbach basis of $X$, that is $\max_j |a_j|\le \nrm{\sum_j a_j f_j}_{L_p(\mu_0)}$.  Let $\mc R$ be an $(\vep/(6n),K)$ appropriate partition for $F$, consisting of $F_*\mu_0$-continuity sets and let $\mc P:=F^{-1}(\mc R)$. It follows from Proposition \ref{koemoirmwiomrwe} that $\max_j\nrm{f_j - \mbb E(f_j; \mc P)}\le \vep/6n$, and since $F$ is an Auerbach basis, we obtain that $\nrm{f- \mbb E(f; \mc P)}\le \vep/6$. Let then $\mc E(X;\vep):= \langle \mathbbm 1_{P}\rangle_{P\in \mc P}$ and let $\xi_{X,\vep}: X\to \widehat{X}^\vep$ be the restriction to $X$ of the conditional expectation $\mathbb E(\cdot; \mc P)$. We use  Theorem \ref{oi4jiri4rjer} of approximate equimeasurability to find $\de>0$ such that if $T:X\to L_p(\mu_1)$ is a $\de$-embedding such that $T(\mathbbm 1_{\Om_0})=\mathbbm 1_{\Om_1}$, and setting $G:=(g_j)_{j<n}$, $g_j:=T(f_j)$, then 
 \begin{enumerate}[(i)]
 \item $(1+\vep_0)^{-1}\le \mu_1(G^{-1}R) /\mu_0(F^{-1}R) \le 1+\vep_0$  for $R\in \mc R_+^F$,  where $\vep_0:= \vep/(4n K (3 \nrm{\mu_0})^{1/p} )$;
 \item$\max_j\int_{G^{-1}R_0^F} |g_j|^pd\mu_1\le \vep_1=\vep^p/(n^p 6)$;  
 \item $\mc R$ is  $(\vep/2n,K)$-adequate; 
\item $G$ is a 2-biorthogonal sequence of $T X$, that is $\max_j |a_j|\le 2\nrm{\sum_j a_j g_j}_{L_p(\mu_1)}$.
   \end{enumerate} 
(i) is possible because each $C$ is a  $F_*\mu_0$-continuity set, and (ii) and (iii) are possible because we can force each $|z_j|^pd(F_*\mu_0)$ and $|z_j|^pd(G_*\mu_1)$  to be close enough with respect to the Lévy-Prokhorov metric. 
Suppose that $\ga:X\to L_p(\Om_1,\mu_1)$ is a unital $\de$-embedding. Let $G:= (g_j)_{j<n}$, $g_j=\ga(f_j)$, $\mc Q:=G^{-1}\mc R$. Observe that by {\rm(iv)} and {\rm (v)} we have that 
$\nrm{i_{\ga X}- \mbb E(\cdot; \mc Q)\rest \ga X}_{L_p(\Om_1,\mu_1)}\le \vep$. 
 Let 
$I: \mc E(X;\vep)  \to  L_p(\Om_1,\mu_1)$ be linearly defined for each $R\in  \mc R_+^G$   by
$$I(\mathbbm 1_{F^{-1} R}):= \frac{\mu_0(F^{-1} R)}{\mu_1(G^{-1} R)}\mathbbm 1_{G^{-1}R}.$$
Note that by {\rm(i)} above,  
$\mc R_+^F\con \mc R_+^G$;   hence, for scalars $(a_R)_{R\in \mc R^F_+}$, 
\begin{align*}
\nrm{I(\sum_{R\in \mc R_+^F} a_R \mathbbm 1_{F^{-1}R})}_{L_p(\mu_1)}&=
\nrm{\sum_{R\in \mc R_+^F} a_R\frac{\mu_0(F^{-1} R)}{\mu_1(G^{-1} R)}}_{L_p(\mu_1)}= 
\sum_{R\in \mc R_+^F} a_R \mu_0(F^{-1} R)=  \nrm{\sum_{R\in \mc R_+^F} a_R \mathbbm 1_{F^{-1}R}}_{L_p(\mu_0)}.
\end{align*}
So, $I$ is an isometric embedding.  The proof will be finished once we establish that   
$$\nrm{I \circ \xi_{X,\vep} -\mbb E(\cdot;\mc Q)\circ \ga} \le \vep.$$  
Fix $k<n$.    We have that
 \begin{align*}
 I (\xi_{X,\vep}(f_k)) =& I(\mathbb E(f_k; \mc P)) =\sum_{R\in \mc R_+^F}\left( \frac{1}{\mu_1(G^{-1}(R))}\int_{F^{-1}(R)} f_k d\mu_0\right)  \mathbbm 1_{G^{-1}(R)};\\
 \mathbb E(g_k;\mc Q) =& \sum_{R\in \mc R_+^F}\left( \frac{1}{\mu_1(G^{-1}(R))}\int_{G^{-1}(R)} g_k d\mu_1\right)  \mathbbm 1_{G^{-1}(R)}.
   \end{align*}
  Fix $R\in \mc R_{+,k}^F$, and let   $\al\in \pi_k(C)$. Then we have that 
 \begin{align*}
 \left| \int_{F^{-1}(R)} f_k d\mu_0-\int_{G^{-1}(R)} g_k d\mu_1    \right| \le &   \int_{F^{-1}(R)} \left| f_k-\al \right| d\mu_0 +    \int_{G^{-1}(R)} \left| g_k - \frac{\mu_0(F^{-1}(R) }{\mu_1(G^{-1}(R)} \al \right| d\mu_1    \le \\
 \le & \frac{\vep}{6n \sqrt[p]{3\nrm{\mu_0}}}\mu_0(F^{-1}(R)) +\left( \frac{\vep}{6n \sqrt[p]{3\nrm{\mu_0}}}+ \vep_0 K\right)\mu_1(G^{-1}(R)) \le \\
 \le & \left( \frac{3\vep}{4n \sqrt[p]{3\nrm{\mu_0}}}  \right) \mu_1(G^{-1}(R)).
 \end{align*} 
 Hence,
 \begin{align*}
 &\nrm{I(\xi_{X,\vep}(f_k))\rest G^{-1}(R_{+,k}^F)- \mathbb E(g_k;\mc Q)\rest G^{-1}(R_{+,k}^F)}_{L_p(\mu_1)}^p = \\
 &= \sum_{R\in \mc R_{+,k}^F} \int_{G^{-1}(R)} \left|\frac1{\mu_1(G^{-1}(R))}\left( \int_{F^{-1}(R)} f_k d\mu_0-\int_{G^{-1}(R)} g_k d\mu_1\right)   \right|^p \le \\
 &\le  \sum_{R\in \mc R_{+,k}^F} \left( \frac{3\vep}{4n \sqrt[p]{3\nrm{\mu_0}}}  \right)^p \mu_1(G^{-1}(R))\le 
 \sum_{R\in \mc R_{+,k}^F} \left( \frac{3\vep}{4n }  \right)^p \frac{\mu_1(G^{-1}(R))}{{3\nrm{\mu_0}}}.
 \end{align*} 
 On the other hand, by {\rm(iii)} above, 
  \begin{align*}
\nrm{\mathbb E(g_k;\mc Q)\rest G^{-1}(R_{0}^F)}_{L_p(\mu_1)}^p \le & \nrm{g_k \rest G^{-1}(R_0^F)}_{L_p(\mu_1)}^p\le  \nrm{g_k \rest G^{-1}(R_0^G)}_{L_p(\mu_1)}^p\le   \frac{\vep^p}{6n^p }.
 \end{align*} 
 Finally, since $\mc R$ is $(\vep/6n,K)$-adequate for both $F$ and $G$, we obtain from (a) in Proposition \ref{koemoirmwiomrwe} that 
  \begin{align*}
\nrm{\mathbb E(g_k;\mc Q)\rest G^{-1}(R_{\infty,k}^F)}_{L_p(\mu_1)}^p \le &  \frac{\vep^p}{3(6n)^p }.
 \end{align*} 
Since 
  \begin{align*}
I(\xi_{X,\vep}(f_k)\rest G^{-1}(R_{\infty,k}^F))=I(\mathbb E(f_k;\mc P) \rest F^{-1}(R_{\infty,k}^F))
 \end{align*} 
and $I$ is an isometry, we obtain, using again (a) in Proposition \ref{koemoirmwiomrwe} that 
  \begin{align*}
\nrm{I(\xi_{X,\vep}(f_k)\rest G^{-1}(R_{\infty,k}^F))}_{L_p(\mu_1)}^p &= \nrm{\mathbb E(f_k;\mc P) \rest F^{-1}(R_{\infty,k}^F)}_{L_p(\mu_0)}^p\le  \frac{\vep^p}{3(6n)^p }.
 \end{align*} 
 Putting all together,
 \begin{align*}
  \nrm{I \circ \xi_{X,\vep}(f_k) -\mbb E(g_k;\mc Q)}^p = &   \nrm{(I \circ \xi_{X,\vep}(f_k) -\mbb E(g_k;\mc Q)) \rest  G^{-1}(R_{+,k}^F)}^p + \nrm{\mathbb E(g_k;\mc Q)\rest G^{-1}(R_{0}^F)}_{L_p(\mu_1)}^p + \\
  + & \nrm{(I \circ \xi_{X,\vep}(f_k) -\mbb E(g_k;\mc Q)) \rest  G^{-1}(R_{\infty,k}^F)}^p\le     \\
\le &  \sum_{R\in \mc R_{+,k}}  \left( \frac{3\vep}{4n }  \right)^p \frac{\mu_1(G^{-1}(R))}{{3\nrm{\mu_0}}} +  \frac{\vep^p}{3(6n)^p } + 2\frac{\vep^p}{3(6n)^p }\le   \frac23\left( \frac{3\vep}{4n }  \right)^p +\frac{\vep^p}{(6n)^p}\le 
 \frac{\vep^p}{n^p}. 
 \end{align*} 
 Since $(f_k)_{k<n}$ is an Auerbach basis of $X$, we obtain that $\nrm{I \circ \xi_{X,\vep} -\mbb E(\cdot;\mc Q)\circ \ga}_{L_p(\mu_0),L_p(\mu_1)}\le \vep.$  
\fprue

We are ready for the proof of the \auhp property of $L_p(0,1)$ for $1\le p<\infty$, $p\notin 2\N$:
\prue[{\sc Proof of Theorem \ref{auh_for_L_p}}] 
Suppose that $X\in \age(L_p(0,1))$, $0<\vep\le 1$ and denote by $S_X$ the unit sphere of $X$, that is, the points in $X$ of norm one. Let $A$ be the support of $X$. By making a small perturbation if needed, we assume that $\la(A)<1$. Let $u\in S_X$ be  such that $\supp u=A$, and let   $\theta_X\in \iso(L_p(A))$ be such that $\theta_X(u)=\la(A)^{-1/p}\mathbbm 1_A$ (see for example Lemma \ref{singleisometry}). Let now $X_0:=\theta(X)$. We apply   Lemma \ref{ksmfkwewe}  to $X_0$ and $\vep$ to find the corresponding $0<\de_0\le \vep$,  and  then Theorem \ref{kjnnjguytfyyuoipo}  to $u$  $X$ and $\de_0/2>0$ to find the corresponding $0<\de_1\le \de_0$.  Let $0<\de\le \de_1$ be such that $(1+\de)^2\le (1+\de_0)(1-\de_0/2)$.   We claim that such $\de$ works: For suppose that $\ga: X\to L_p(0,1)$ is a $\de$-embedding. It follows that $\ga(u)$ has $\de_0/2$-full support in $\ga(X)$. Let $B:=\supp \ga(u)$, that without loss of generality we assume that $\la(B)<1$, and let $\ga_0:X\to L_p(0,1)$ be 
$$\ga_0:= \nrm{\ga(u)}^{-1} P_B\circ \ga.$$ By the choice of $\de$, we have that $\ga_0\in \Emb_{\de_0}(X,L_p(B))$ and $\nrm{\ga_0-\ga}\le 3\vep$.    Let now $\Phi: L_p(B)\to L_p(A)$ be an isometry onto.  Let $Y:= \Phi(P_B(\ga(X))):=\conj{\Phi(g\rest B)}{g\in \ga(X)}$. We consider now $\ga_1: X\to Y\con L_p(A)$, $$\ga_1:= \Phi\circ \ga_0,$$ and note that $\ga_1\in \Emb_{\de_0}(X, L_p(A))$.  Set $v:=\ga_1(u)$, that is normalized and has support $A$, and let  $\theta_Y\in \iso(L_p(A))$ be such that $\theta_Y(v)=\la(A)^{-1/p} \mathbbm 1_A$.  We set now $\ga_2: X_0\to \theta_Y(Y)=Y_0$, 
$$\ga_2:= \theta_Y\circ \ga_1 \circ \theta_X^{-1} \rest X_0.$$ Observe that $\mathbbm 1_A\in X_0\cap Y_0$, and that $\ga_2(\mathbbm 1_A)= \mathbbm 1_A$, so by the choice of $\de_0$, there is some   $\widehat{X}\equiv\ell_p^m$, $\xi: X_0\to \widehat{X} $ and an isometric embedding $I: \widehat{X}\to L_p(A)$ such that $\nrm{\xi -i_{X_0}}\le \vep$, and $\nrm{I\circ \xi - \ga_2}\le \vep$.  By Proposition \ref{ultracopylp}, there is an isometry $\theta\in \iso(L_p(A))$ such that $\nrm{\theta\rest \widehat{X}- I}\le \vep$.  
Then,   $\nrm{\theta  - \ga_2}\le 4\vep$, or, equivalently, 
\begin{equation*}
\nrm{\theta_2  -  \ga_0}\le 4\vep,
\end{equation*}
where $\theta_2:=\Phi^{-1}\circ\theta_Y^{-1}\circ\theta \circ\theta_X: L_p(A)\to L_p(B)$ is an isometry onto.  Let $J\in \iso(L_p(0,1))$ be extending $\theta_2$.  Let us see that this is possible: since we have that 
$0<\la(A),\la(B)<1$, we have that $L_p(A^c)$ and $L_p(B^c)$ are isometric, so we fix a surjective isometry $\Theta: L_p(A^c)\to L_p(B^c)$, and then $J(f):= \theta_2(f\rest A) +\Theta(f\rest A^c)$ makes the job. Since we know that $\nrm{\ga-\ga_0}\le 3\vep$, putting all together we obtain that
\[\nrm{J \rest X - \ga}\le 7\vep. \qedhere \]
  
\fprue

\begin{corollary}
For $1\le p<\infty$,  $p\neq 4,6,\dots$ the class $\mr{Age}(L_p(0,1))$   of finite dimensional subspaces of $L_p(0,1)$ is a Fraïssé class whose limit is $L_p(0,1)$. 
\end{corollary}
In a personal communication I. Ben Yaacov \cite{BeYa0} mentioned to us that there might be  connections of this result with the  Ryll-Nardzewski Theorem and   quantifier elimination in the context of continuous logic (see \cite[Sections 12 and 13]{BYBeHeUs}) .

\subsection{The approximate equimeasurability principle}\label{approx_equim}
As mentioned before,  Plotkhin     and Rudin    independently  proved that for $1\le p<\infty$,  $p \notin 2\N$,  the transform $ \widehat{\mu}^{(p)}(a):= \nrm{1+ \langle a,z\rangle}_{L_p(\mu)}$ determines the measure  $\mu$ (for which $\mbb E_\mu(|z|^p)<\infty$),  much like the    {\em Fourier-Stieltjes
transform}  $\widehat{\mu}(a):=\int e^{i \Re \langle a, u \rangle}d\mu(u)$ does for an arbitrary measure $\mu$.  In this case, there is also a continuity aspect of it,  called {\em L\'evy's continuity theorem}, stating that  for finite measures, if $ \widehat{\mu_n}\to \widehat{\mu}$ converges uniformly on compacta, then $\mu_n\to_n \mu$ completely.    The goal now is to see   that  a similar statement   holds for   the transform $\widehat{\mu}^{(p)}$. In fact, the distance between the transforms $\widehat{\mu}^{(p)}$ and $\widehat{\nu}^{(p)}$   will also determine the distance between  the finite measures $|z|^\al d\mu$ and  $|z|^\al d\nu$ for every $0\le \al\le p$.   Our proof follows some ideas of the  standard proof of Levy's continuity theorem (see for example \cite[Theorem 2.6.8]{Cu})  and the proof of  the equimeasurability principle given by C. D. Hardin in \cite[Theorem 1.1a]{Har} (see  also \cite[Theorem 3.3.2]{FJ1}).  

We start by recalling some basic concepts in measure theory. We refer the reader to \cite{Ha} or \cite{Cu} for more details. 
Given a separable metric space $(X,d)$, let $\mc M(X)$ be the collection of all finite   measures  on $\mc B(X)$, the class of Borel subsets of $(X,d)$. Given $\mu\in \mc M(X)$, $A\in \mc B(X)$ is called a {\em $\mu$-continuity set} if $\mu(\partial A)=0$, where $\partial A$ is the topological boundary of $A$. Recall that a sequence $(\mu_n)_n$ of finite   measures {\em converges weakly}  to $\mu\in \mc M(X)$   when $\mu_n(A)\to_n\mu(A)$  for every bounded continuity set $A\in \mc B(X)$. The sequence $(\mu_n)_n$ converges {\em completely} to $\mu$ if it converges weakly and $\nrm{\mu_n}\to_n \nrm{\mu}$.   It is well-known that $(\mu_n)_n$ converges completely to $\mu$ exactly when $(\mu_n)_n$ converges weakly to $\mu$ and $(\mu_n)_n$ is {\em tight}, or, equivalently, when   $\mu_n(A)\to_n \mu(A)$ for every continuity set $A\in \mc B(X)$. Recall that $(\mu_n)_n$ is tight if for every $\vep>0$ there is a compact set $K\con X$ such that $\sup_n \mu_n(X\setminus K)\le \vep$.  Given a function $f:X\to \R$ that is integrable with respect to $\mu\in \mc M (X)$ one defines the (signed) measure $fd\mu$ by $fd\mu(A):=\int_A f d\mu$.  The {\em Fourier-Stieltjes transform} of $\mu\in \mc M(\F^n)$ is the function $\widehat{\mu}:\F^n\to \C$, 
$$\widehat{\mu}(a)=\int e^{i\Re(\langle a,\bar b\rangle)} d\mu(b),$$ 
where $\Re(\al)$ is the real part of $\al$, and $\overline{ (b_1,\dots,b_n)}=(\bar b_1,\dots,\bar b_n)$ is the sequence of conjugates of $b_1,\dots,b_n$.  Notice that when $\mbb F=\C$, the definition above coincides with the standard definition of the Fourier-Stieltjes transform for measures on $\R^{2n}$, via the canonical identification of $\C$ with $\R^2$.  Recall also that given $\mu\in \mc M(\mbb F^n)$ and given $f:\mbb F^n\to \R$ that is integrable with respect to $\la$ and $\mu$,  one defines the convolution $f* \mu: \mbb F^n\to \R  $, 
$$(f*\mu)(b):=\int f(b-a)d\mu(a),$$
that corresponds to the convolution $(fd\la) * \mu$ of the measures $fd\la$ and $\mu$.  Notice that this is the case when $f: \F^n\to \R$ is continuous, bounded and $\la$-integrable.    A basic property we will use is that  $\widehat{f * \mu}=\widehat{f}   \cdot \widehat{\mu}=\widehat{f d\la}\cdot \widehat{\mu}$.

\begin{definition}
Let $\mc M^{(p)}(\mbb F^n)$ be the collection of all Borel measures on $\mbb F^n$ such that $|z|^pd\mu\in \mc M_n(\mbb F^n)$, that is,  $\int |z|^p d\mu(z)<\infty$. 
\end{definition}
In the previous definition $|z|$ is the euclidean norm of the vector $z$. It follows that $\mc M^{(0)}(\mbb F^n)=\mc M(\mbb F^n)$ is the collection of finite Borel measures on $\mbb F^n$,  and that $\mu\in \mc M^{(p)}(\mbb F^n)$ if and only if $\int |z_j|^pd\mu(z)<\infty$ for every $1\le j\le n$, where each $z\in \mbb F^n\mapsto z_j\in \F$ is the canonical $j^\mr{th}$-projection.    Recall that given a measurable function  $T:(\Om_0,\Sig_0)\to (\Om_1,\Sig_1)$ and a measure $\mu$ on $\Sig_0$ one defines the {\em pushforward measure} $T_*\mu $ on $\Sig_1$ by 
$(T_*\mu)(A):=\mu(T^{-1}(A))$ for $A\in \Sig_1$.  In particular, each sequence $F=(f_1,\dots,f_n)$ of elements of a Lebesgue space $ L_p(\Om,\Sig,\mu)$ defines the measure  $F_*\mu$    on  $\mbb F^n$, where $F$ is interpreted as the measurable function $F: \Om\to \mbb F^n$, $F(\om)=(f_1(\om),\dots,f_n(\om))$.  Notice that for $H(z_1,\dots,z_n)\in L_p(F_*\mu )$ one has that $\int H(z_1,\dots,z_n) d(F_*\mu)(z_1,\dots,z_n)=\int H(f_1(\om),\dots, f_n(\om))d\mu(\om)$.

\begin{definition}[$p$-characteristics]
Given $\mu\in \mc M^{(p)}(\mbb F^n)$, we define the \emph{$p$-characteristics}  $\widehat{\mu}^{(p)}:\mbb F^n\to [0,\infty[$ of $\mu$  by 
$$\widehat{\mu}^{(p)}(a):=\nrm{1+\langle a,z\rangle}_{L_p(\mu)}=\left(\int |1+ \langle a ,z\rangle|^p d\mu(z)\right)^\frac{1}{p} \hfill \text{for every $a\in \mbb F^n$.}$$
\end{definition}
With this terminology,   Plotkin   and Rudin  results  can be restated as follows.
\begin{theorem}[Uniqueness of the $p$-characteristics]
Suppose that $1\le p<\infty$, $p\notin 2\N$.
If $\mu,\nu\in \mc M^{(p)}(\mbb F^n)$ are such that $\widehat{\mu}^{(p)}=\widehat{\nu}^{(p)}$, then $\mu=\nu$. 
\end{theorem}
 This is the corresponding continuity statement for    $p$-characteristics.
\begin{theorem}[Continuity of the $p$-characteristics]\label{continuity_theorem}
Suppose that $1\le p<\infty$, $p\notin 2\N$.  The following are equivalent for a sequence $(\mu_k)_k $  and a measure $\mu$ all   in   $\mc M^{(p)}(\mbb F^n)$:
\begin{enumerate}[1)]
\item $(|z|^\al d\mu_k)_k$ converges completely to $|z|^\al d\mu$ for all $0\le \al \le p$;

\item $(|z|^p d\mu_k)_k$ converges completely to $|z|^p d\mu$ and $\nrm{\mu_k}\to_k \nrm{\mu}$;
\item $(\mu_k)_k$ converges completely to $\mu$ and  $(|z|^pd\mu_k)_k$ is tight;

 \item $(\widehat{\mu_k}^{(p)})_k$ converges to $\widehat{\mu}^{(p)}$ uniformly in all compacta of $\mbb F^n$. 

\end{enumerate}
\end{theorem}
  The proof of each implication  is done in several steps, being ${\it 4)}\implies {\it 3)}$ the more interesting  one.

 \prop\label{jjheorhthrhjtr}
 Let $f:(X,d)\to \R^+$ be a continuous function and suppose that $(\mu_k)_k$ and $\mu$ satisfy that $\int f d\mu_n,\int f d\mu <\infty$. If   $(\mu_k)_k$ converges weakly to $\mu$  and $(fd\mu_k)_k$ is tight, then  $(\mu_k)_k$ is tight and    $(f d\mu_k)_k$ converges completely to $fd\mu$.  Consequently ${\it 3)}\implies {\it 1)}$  of Theorem  \ref{continuity_theorem} holds for every $p$. 
 \fprop
\prue
By one of the several characterizations of weak convergence (Portmanteau Theorem, \cite[Theorem 2.1]{Bi}) we have to check that $\limsup_k \int_C f d\mu_k \le \int_C fd\mu$ for every closed subset $C\con X$. Fix $\vep>0$. and let $K\con X$ be a compact subset such that $\sup_k(fd\mu_k)(K^c)<\vep/2$. Now the function $g:= \mathbbm{1}_{K\cap C} \cdot f$ is upper semicontinuous and bounded, since $K\cap C$ is closed. Since $\mu_k\to_k \mu$ weakly, it follows that $\limsup_k \int g d\mu_k\le \int gd\mu$. Consequently,
$$\limsup_k \int_C f d\mu_k \le  \limsup_k \int g d\mu_k +\limsup_k \int_{K^c} f d\mu_k \le  \int g d\mu +\vep \le \int_C f d\mu +\vep.
$$ 
Since $\vep$ is arbitrary, we are done. 
\fprue

\begin{proposition}
Suppose that $(|z|^pd\mu_k)_k$ converges completely to $|z|^pd\mu$ and suppose that $\nrm{\mu_k}\to_k\nrm{\mu}$. Then $(\mu_k)_k$ converges completely to $\mu$. Consequently ${\it 2)}\implies {\it 3)}$  of Theorem  \ref{continuity_theorem} holds for every $p$. 
\end{proposition}
\prue
The fact that $(|z|^pd\mu_k)_k$ converges completely gives  that it is tight, and since $p\ge 0$,  we obtain that $(\mu_k)_k$ is tight.  Since this implies that every subsequence of $(\mu_k)_k$ has a further completely  converging subsequence, the next claim implies that $(\mu_k)_k$ converges completely to $\mu$.
\clam
Every complete converging subsequence of $(\mu_k)_k$ converges to $\mu$.
\fclam
\prucl 
Suppose that $(\mu_k)_{k\in M}$ converges completely to $\nu$.  It follows from Proposition \ref{jjheorhthrhjtr} that $|z|^p d\mu_k\to_{k\in M} |z|^p d\nu(z)$ completely, so $|z|^pd\nu=|z|^p\mu$. We are going to see that this implies that $\nu=\mu$: 
Fix $\de>0$. Let $0<\ga\le \de$ be such that $(1+\de) (1-p\ga/\de)\ge 1$, and let $\mc P$ be a countable partition of $K_\de:=\mbb F^n \setminus B(0,\de)$ of (Borel) subsets of diameter at most $\ga$.  A simple computation gives that if $A\con K_\de$ as diameter at most $\ga$, then $\sup_{a\in A} |a|^p \le (1+\de)\inf_{a\in A}|a|^p$.  This implies that given  $P\in \mc P$ one has that 
$$\inf_{a\in P} |a|^p \mu(P)\le \int_{P} |z|^p d\mu(z)= \int_{P} |z|^pd\nu(z)\le \sup_{a\in P} |a|^p \nu(P)\le (1+\de)\inf_{a\in A}|a|^p\nu(P),$$
 and hence $\mu(P)\le (1+\de)\nu(P)$.  It follows that   given $E\con K_\de$ and  setting $\mc P_E:=\conj{P\in \mc P}{P\cap E\neq \buit}$ we have that $E\con \bigcup_{P\in \mc P_E} P\con (E)_\ga$, and  consequently,
$$\mu(E) \le \mu(\bigcup_{P\in \mc P_E} P)= \sum_{P\in \mc P_E} \mu(P)\le (1+\de) \sum_{P\in \mc P_E} \nu(P) \le (1+\de) \nu((E)_\ga)\le (1+\de) \nu((E)_\de).$$ 
Since $\de$ was arbitrary, we obtain that $\mu(E)\le \nu(\bar E)$ and $\nu(E)\le \mu(\bar E)$ for every Borel $E\con \F^n\setminus \{0\}$.  In particular $\mu(K_\de)=\nu(K_\de)$ for  every $\de>0$, so $\mu(\F^n\setminus \{0\})=\nu(\F^n\setminus \{0\})$. Using that $\nrm{\mu}=\nrm{\nu}$ we obtain that $\mu(\{0\})=\nu(\{0\})$.   Now fix a Borel continuity subset  $E\con \F^n$  for $\mu$ and $\nu$, and let $(\de_n)_n$ be a decreasing sequence of strictly positive real numbers  with limit 0 such that $\mu(\partial B(0,\de_n))=\nu(\partial B(0,\de_n))=0$. Since $\partial (E\cap K_{\de_n})\con \partial E \cup \partial B(0,\de_n)$, it follows that $E\cap K_{\de_n}$ is a continuity set, so $\mu(E\cap     K_{\de_n})=\nu(E\cap     K_{\de_n})$ for all $n$. Hence, $\mu(E\setminus \{0\})=\lim_{n\to \infty}\mu(E\cap     K_{\de_n})=\lim_{n\to \infty}\nu(E\cap     K_{\de_n})=\nu(E\setminus \{0\})$ and consequently 
$\mu(E)=\nu(E)$.  
\fprucl
 \fprue 		
We  need   the following simple estimates. 
\prop\label{oihir43iljkerfkpokgtriojre}
Suppose that  $\mu,\nu\in \mc M^{(p)}(\mbb F^n)$. Then
\begin{enumerate}[(a)]
\item $0\le \widehat{\mu}^{(p)}(a) \le    \nrm{\mu}^{1/p} + |a| \cdot \nrm{\,|z|^p d\mu}^{1/p}  $ for every $a\in \mbb F^n$;
\item $\widehat{\mu}^{(p)}$ is uniformly continuous; in fact, $|\widehat{\mu}^{(p)}(a)-\widehat{\mu}^{(p)}(b)|\le |a-b| \cdot \nrm{|z|^p d\mu}^{1/p}$. 
\end{enumerate}

\fprop
\prue
{\it (a)}:  Using Cauchy–Schwarz, 
$\widehat{\mu}^{(p)}(a)=\nrm{1+ \langle a, z\rangle}_{L_p(\mu)}\le \nrm{1}_{L_p(\mu)} + |a| \nrm{|z|}_{L_p(\mu)} =\nrm{\mu}^{\frac1p}+ |a| \nrm{|z|^p d\mu}^{\frac1p} .$
 {\it (b)}: $|\widehat{\mu}^{(p)}(a)-\widehat{\mu}^{(p)}(b)|\le \nrm{\langle a-b,z\rangle}_{L_1(\mu)} \le 
 |a-b| \nrm{|z|}_{L_1(\mu)}= |a-b| \cdot \nrm{|z|^p d\mu}^{1/p}$.  
\fprue

\prop
Suppose that we have the complete convergence $\mu_k\to_k \mu$ and   $|z|^p d\mu_k\to |z|^p d\mu$. Then  $(\widehat{\mu_k}^{(p)})_k$ converges to $\widehat{\mu}^{(p)}$ uniformly in all compacta of $\mbb F^n$.   Consequently ${\it 1)}\implies {\it 4)}$  of Theorem  \ref{continuity_theorem} holds for every $p$. 

\fprop
\prue
Fix $M>0$ and $\vep>0$.   We know that $(|z|^p\mu_k)_k$ is tight, so there is some $K\ge 1$ such that $\int_{|z|>K}|z|^pd\mu_k(z)\le \vep.$ Let $D$ be a finite $\vep$-dense subset of $B(0,M)$, and let $k_0$ be such that 
$$ \left|\int_{|z|\le K} |1+ \langle a,z\rangle|^p  d\mu_k(z)  -\int_{|z|\le K} |1+ \langle a,z\rangle|^p  d\mu(z)   \right|\le \vep$$ for every $a\in D$ and every $k\ge k_0$. This is possible since $\mathbbm 1_{|z|\le K} |1+\langle a,z\rangle|^p$ is bounded and      $\mu_k\to \mu$ completely.    Then for such $k$ and  $a$,  
\begin{align*}
|\widehat{\mu_k}^{(p)}(a)-\widehat{\mu}^{(p)}(a)| \le & 
\left|\int_{|z|\le K} |1+ \langle a,z\rangle|^p  d\mu_k(z)  -\int_{|z|\le K}  |1+ \langle a,z\rangle|^p  d\mu(z)   \right|    
+ \mu(B(0,K)^c)  +\\ 
+& \mu_k(B(0,K)^c)+    |a|^p((|z|^p d\mu)(B(0,K)^c)  + (|z|^p d\mu_k)(B(0,K)^c))\le   \vep(3+ 2|a|^p)\le \\
\le &\vep(3+2 M^p).  
\end{align*}
Using this and Proposition \ref{oihir43iljkerfkpokgtriojre} {\it(b)}, for every $|a|\le M$ and $k\ge k_0$, 
  \begin{align*}
|\widehat{\mu_k}^{(p)}(a)-\widehat{\mu}^{(p)}(a)| \le & \vep(\nrm{|z|^pd\mu} +\nrm{|z|^pd\mu_k} + 3+2 M^p). 
\end{align*}
This shows that $\widehat{\mu_k}^{(p)}\to \widehat{\mu}^{(p)}$ uniformly on $B(0,M)$.  
\fprue

We have already seen that    {\it 1)}, {\it 2)} and {\it 3)}  of Theorem \ref{continuity_theorem}   are equivalent ({\it 1)} implies {\it 2)} trivially  and also that ${\it 3)}\implies {\it 4)}$. We conclude by showing that   ${\it 4)}\implies {\it 3)}$. We start with the following interesting criteria for complete convergence extending the proof given by Hardin in \cite[Theorem 1.1a]{Har} of the equimeasurability principle for $n=1$.
\begin{lemma}\label{jwejrjw23dddqwe}
Suppose that $(\mu_k)_k$ and $\mu$ are measures in $\mc M(\mbb F)$ such that $\nrm{\mu_k}\to_k \nrm{\mu}$  and such that  there is a continuous  and bounded $0\neq f\in L_1(\mbb F,\la)$ such that $\int f(a+bz)d\mu_k(z)\to_k \int f(a+bz)d\mu(z)$ for every $(a,b)\in\F\times \F$. Then $(\mu_k)_k$ converges completely to $\mu$.  
\end{lemma}
\prue
 We fix all data. Notice that $f\in L_p(\nu)$ for every $\nu\in \mc M(\mbb F)$, because $f$ is bounded and continuous, and $\nu(\mbb F)$ is finite.  
  \clam	
 $(\mu_k)_k$ is tight. 
 \fclam
 \prucl	
 Fix $\vep>0$.   Let $z_0$ be such that $|f(z_0)|=\nrm{f}_\infty$, and let $\de>0$ be such that $(1-\vep)|f(z)|\le |f(z_0)|$ for $|z-z_0|\le \de$. Fix $K>0$  such that $\mu(B(0,K)^c)\le \vep$ and     let $g(z):= |f(z_0)|^{-1} f((\de/K)z +z_0)$. Notice that $\int g(a+bz)d\mu_k\to_k \int g(a+bz)d\mu$ for every $(a,b)\in \mbb F\times \mathbb F$, that $1=g(0)=\max_z g(z)=\nrm{g}_\infty$, and that $1-\vep\le g(z)\le 1$ for every $|z|\le K$.  Let $L\ge K$ be such that $\max_{|z|\ge L} |g(z)|\le \vep$. 
Let $k_0$ be such that  $|\int g d\mu_k-\int  g d\mu|, |\nrm{\mu_k}-\nrm{\mu}|\le \vep$ for every $k\ge k_0$. Then for such $k$,  
\begin{align*}
\mu_k(B(0, L))\ge & \int_{|z|\le L} g(z) d\mu_k(z) \ge  \int g(z) d\mu_k(z) -\vep\nrm{\mu_k}\ge       \int g(z)  d\mu(z) - \vep(1+ \nrm{\mu_k}) \ge \\
\ge &  \int_{|z|\le K} g(z)  d\mu(z)    - (2+\nrm{\mu_k} )\vep \ge   (1-\vep)\mu(B(0,K))  - (2+\nrm{\mu_k} )\vep\ge\\
\ge &  (1-\vep)(\nrm{\mu}-\vep)   - (2+\nrm{\mu_k} )\vep\ge (1-\vep)(\nrm{\mu_k} -2\vep)  - (2+\nrm{\mu_k} )\vep = \\
=  & \nrm{\mu_k}  - (4+2\nrm{\mu_k})\vep.
\end{align*}
So, $\mu_k(B(0,L)^c)\le   (4+2\nrm{\mu_k}) \vep ,$
 and since $(\mu_k)_k$ is bounded, the previous inequality shows that  $(\mu_k)_k$ is tight. 
 \fprucl
 \clam \label{eiorwerewi44}
 Suppose that $\nu$ is such that $\nrm{\mu}=\nrm{\nu}$ and that   $\int f(a+bz)d\mu(z)=\int f(a+bz)d\nu(z)$ for every $a,b\in \mbb F$. Then $\mu=\nu$.
 \fclam
\prucl
This is essentially  proved in \cite{Har}. For the sake of completeness we give a proof.  Given $c\in \mbb F$, let $f_c(z):= f(cz)$.     For  $a,c\in \mbb F$ we have 
$(f_c*\mu)(a)=\int f_c(a-u) d\mu(u)=\int f(ca-cu) d\mu(u)=\int f(ca-cu) d\nu(u)=\int f_c(a-u) d\nu(u)=(f*\nu)(a)$, so, 
\begin{equation}\label{lwij3rijewirewrewr}
\widehat{f_c}\widehat{\mu}=\widehat{f_c*\mu}=\widehat{f_c*\nu}=\widehat{f_c}\widehat{\nu}
\end{equation} 
We prove that $\widehat{\mu}=\widehat{\nu}$, that, by the uniqueness of Fourier-Stieltjes transform, implies that $\mu=\nu$. So, fix  $ a\in \mbb F$. First of all, if $a=0$, then  
$\widehat{\mu}(0)=\nrm{\mu}=\nrm{\nu}=\widehat{\nu}(0)$.  So, suppose that $a\neq 0$; let $0\neq a_0\in \mbb F$ be such that $\widehat{f}(a_0)\neq 0$. This is possible because $f\neq 0$, and by the continuity of $\widehat{f}$. Set $c:=a/a_0$. It follows from      \eqref{lwij3rijewirewrewr} that
$$\frac{1}{c}\widehat{f}(a_0)\widehat{\mu}(a) = \widehat{f_c}(a)\widehat{\mu}(a)= \widehat{f_c}(a)\widehat{\nu}(a) = \frac{1}{c}\widehat{f}(a_0)\widehat{\nu}(a), $$
hence $\widehat{\mu}(a)=\widehat{\nu}(a)$. 
\fprucl
 We are going to see now  that every subsequence of $(\mu_k)_k$ has a further subsequence converging completely to $\mu$, which proves that $(\mu_k)_k$ converges completely to $\mu$: Fix a subsequence $(\mu_k)_{k\in M}$ of $(\mu_k)_k$;  by Helly's first Theorem, there is a further subsequence $(\mu_k)_{k\in N}$ of   $(\mu_k)_{k\in M}$ converging completely to $\nu$.  Since for each $a,b\in \mbb F$ the function $f(a+bz)$ is continuous and bounded, we have that $\int f(a+bz)d\mu_k(z)\to_{k\in N} \int f(a+bz) d\mu(z)$ for every $a,b\in \mbb F$.  Hence, $\int f(a+bz)d\mu(z)=\int f(b+bz)d\nu(z)$ for every $a,b\in \mbb F$, and this implies, by  Claim \ref{eiorwerewi44}, that $\nu=\mu$. 
\fprue

We are going to use Lemma \ref{jwejrjw23dddqwe} to show the implication ${\it 4)}\implies {\it 3)}$, so we have to find the appropriate function.

\begin{proposition}\label{oiuioui78547878555}
For every $p\in [1,\infty[\setminus  2\N$ and every   $m\ge 2\lfloor p \rfloor+6$ there exists a sequence $(a_j)_{j\le m}$ of real numbers such that the function $f(z):=\sum_{j=0}^m a_j |z+j|^p$ satisfies
\begin{enumerate}[(a)]
\item $f\neq 0$; 
 \item $f\in L_1(\la;\mbb F)$ and is bounded;
 \item $\lim_{z\to 0} \frac{f(z)}{|z|^p}=a_0$.  
\end{enumerate}

\end{proposition}

\prue We give the proof assuming that $\mbb F=\C$, but notice that the same function $f$  when restricted to the real numbers has the desired properties for  $\mbb F=\mbb R$. 
  In the next, the complex power $z^\al:=\exp(\al \log z)$ is defined on $U:=\C\setminus \{z\le 0\}$, and   $\log z$ is the principal brach of the logarithm.  
  Recall that for $|z|>m$ in $U$ and $0\le j\le m$  one has that 
\begin{equation*}
|z+j|^p=(z+j)^\frac{p}{2} (\overline{z}+j)^\frac{p}2 =\sum_{k,l=0}^\infty \binom{\frac{p}2}{k}\binom{\frac{p}2}{l} z^{\frac{p}2-k} \overline{z}^{\frac{p}2-l}  j^{k+l};
\end{equation*}
while if $|z|<1$ is in $U$,
\begin{equation*} 
|z+j|^p=(z+j)^\frac{p}{2} (\overline{z}+j)^\frac{p}2 =\sum_{k,l=0}^\infty \binom{\frac{p}2}{k}\binom{\frac{p}2}{l} j^{p-(k+l)} z^k \overline{z}^l;
\end{equation*}
Let $(a_j)_{j\le m}$ be a non-trivial solution of the system
$$
\begin{cases}
\sum_{j=0}^m a_j j^k=0 & \text{ $0 \le k\le \lfloor p \rfloor +2$}\\
\sum_{j=1}^m a_j j^{p-l}=0 & \text{ $0 \le l\le \lfloor p\rfloor+1$}.
\end{cases}
$$
Let $f(z):=\sum_{j=0}^m a_j |z+j|^p$. 
It follows   that for $|z|>m$ and in $U$,
\begin{align*}
|f(z)|=&\left|\sum_{j=0}^ma_j |z+j|^p\right|=\left|\sum_{j=0}^ma_j   \sum_{k,l=0}^\infty \binom{\frac{p}2}{k}\binom{\frac{p}2}{l}  z^{\frac{p}2-k} \overline{z}^{\frac{p}2-l}  j^{r}\right| = \\
=&
\left|\sum_{r\ge 0}  \left(\sum_{j=0}^ma_j j^{r}\right)    \sum_{k+l=r}  \binom{\frac{p}2}{k}\binom{\frac{p}2}{l}   z^{\frac{p}2-k} \overline{z}^{\frac{p}2-l}   \right|= \\
=&\left|\sum_{r= \lfloor p\rfloor +3}^\infty  \left(\sum_{j=0}^ma_j j^{r}\right)    \sum_{k+l=r}  \binom{\frac{p}2}{k}\binom{\frac{p}2}{l}   z^{\frac{p}2-k} \overline{z}^{\frac{p}2-l} \right| \le 
\sum_{r= \lfloor p\rfloor +3}^\infty  \left|\sum_{j=0}^ma_j j^{r}\right|     \sum_{k+l=r}  \left|\binom{\frac{p}2}{k}\binom{\frac{p}2}{l}\right|  |z|^{p-r} = \\
=&O(|z|^{\lfloor p\rfloor -p+2}),
\end{align*}
and by continuity, of $f$, the previous is true for all $|z|>m$.
Since $\lfloor p\rfloor -p+2>1$, the previous inequality shows that $f\in L_1(\la)$, and that $f$ is bounded.  On the other hand, for $|z|<1$ in $U$, and similarly as before,
\begin{align*}
|f(z) -a_0|z|^p|=&\left|\sum_{j=1}^ma_j |z+j|^p\right|=\left|\sum_{j=1}^ma_j   \sum_{k,l=0}^\infty \binom{\frac{p}2}{k}\binom{\frac{p}2}{l}  j^{p-k-l} z^k\overline{z}^{l} \right| = \\
=&
\left|\sum_{r\ge 0}  \left(\sum_{j=1}^ma_j j^{p-r}\right)    \sum_{k+l=r}  \binom{\frac{p}2}{k}\binom{\frac{p}2}{l}   z^{k} \overline{z}^{l}   \right|= \\
=&\left|\sum_{r= \lfloor p\rfloor +2}^\infty  \left(\sum_{j=1}^ma_j j^{p-r}\right)    \sum_{k+l=r}  \binom{\frac{p}2}{k}\binom{\frac{p}2}{l}   z^{k} \overline{z}^{l} \right| \le 
\sum_{r= \lfloor p\rfloor +2}^\infty  \left|\sum_{j=1}^ma_j j^{p-r}\right|     \sum_{k+l=r}  \left|\binom{\frac{p}2}{k}\binom{\frac{p}2}{l}\right|  |z|^{r} = \\
=&O(|z|^{\lfloor p\rfloor +1}) \text{ as $z\to 0$},
\end{align*}
and again by continuity of $f$, the previous inequality is true for all $|z|<1$, hence $\lim_{z\to 0}  |f(z) -a_0|z|^p|/|z|^p \le \lim_{z\to 0} O(|z|^{\lfloor p\rfloor +1})/|z|^p=0$. 
 \fprue 

\begin{lemma}
Suppose that $(\widehat{\mu_k})_k$  converges to $\widehat{\mu}$ uniformly on compacta of $\mbb F^n$. Then $(|z|^pd\mu_k(z))_k$ is tight and $(\mu_k)_k$ converges completely to $\mu$. Consequently,   ${\it 4)}\implies {\it 3)}$  in Theorem  \ref{continuity_theorem} holds.  
\end{lemma}
\prue We start with the following: 
\clam
 $(|z|^pd\mu_k(z))_k$ is tight.
\fclam
\prucl
For each $1\le j \le n$, let $\mu_k^{(j)}:= (\pi_j)_*\mu_k  \in \mc M_p(\mbb F)$, where $\pi_j:\mbb F^n \to \mbb F$ is the canonical projection  $\pi_j(z_1,\dots,z_n)=z_j$.  Observe that for each $K>0$, 
\begin{equation}
\label{lokij4iorejtijoirtrete}\int_{\nrm{z}_\infty>K} \nrm{z}_\infty^pd\mu_k(z) \le \sum_{j=1}^n \int_{K<\nrm{z}_\infty=|z_i|} |z_i|^pd\mu_k(z)\le\sum_{j=1}^n \int_{K<|t|} |t|^pd\mu_k^{(j)}(t),
\end{equation}
so, it suffices to show that each $(|t|^pd\mu_k^{(j)}(t))_k$ is a tight sequence for each $1\le j\le n$. We fix one of such $j$, and to simplify the notation we set $\nu:=(\pi_j)_*\mu$ and $\nu_k:=\mu_k^{(j)}$ for every $k$.  Note that $\widehat{\nu_k}^{(p)}\to_k \widehat{\nu}$, so in particular $\nrm{\nu_k}\to_k \nrm{\nu}$. Let $f$ be a function as in Proposition \ref{oiuioui78547878555} for $m:=2\lfloor p \rfloor+6$. Then for each $a,b\in \mbb F$ one has that $\int f(a+bz)d\nu_k\to _k \int f(a+bz) d\nu$, so it follows from Lemma \ref{jwejrjw23dddqwe} that   $\nu_k\to_k \nu$ completely.  Since $\widehat{\nu_k}^{(p)}\to \widehat{\nu}$, in particular one has that $\nrm{|t|^p d\nu_k(t)}\to_k \nrm{|t|^p d\nu(t)}$.  Set $F(z):=f(z)/|z|^p$.  For each $a,b \in \mbb F$ we have that $F(a+bz) |z|^p$ is  bounded and continuous, so it follows that $\int F(a+bt) |t|^p d\nu_k(t) \to_k \int F(a+bt) |t|^p d\nu(t)$. Again using    Lemma \ref{jwejrjw23dddqwe} we obtain that $|t|^p d\nu_k(t) \to_k |t|^p d\nu(t)$ completely, so in particular $(|t|^p d\nu_k(t))_k$ is tight, as desired.  
\fprucl
Let us prove now that $\mu_k\to \mu$ completely. Since $(\mu_k)_k$ is tight, by Helly's first Theorem, it suffices to show that each completely convergent subsequence of $(\mu_k)_k$ converges completely to $\mu$; so, fix one such completely convergent subsequences $\mu_{k_l}\to_l \nu$; Since $(|z|^pd\mu_{k_l}(z))_l$ is in addition tight, it follows from the implication   ${\it 3)}\implies {\it 4)}$  of Theorem  \ref{continuity_theorem} that $\widehat{\mu_{k_l}}^{(p)}\to_l \widehat{\nu}^{(p)}$; in particular  $\widehat{\nu}^{(p)}= \widehat{\mu}^{(p)}$, and by the uniqueness theorem of the $p$-characteristics, we have that $\nu=\mu$, as desired. 
\fprue

\subsubsection{Inversion formulas for $\mbb F=\mbb R$ and $p\in 2\N+1$}
When working on real numbers and $p$ an odd integer, there is a more direct and elementary proof of $3)\implies 2)$  of Theorem \ref{continuity_theorem}.  
  
\begin{lemma}\label{inversion}
For $a\neq 0$ we have
$$\mu(]-\infty,a])=  \frac{1}{2}+\frac{1}{2\cdot(p!)  } \lim_{\vep\to 0^+} \frac{1}{\vep^p}\sum_{j=0}^p(-1)^{j+1} \binom{p}{j}\left|a+j\vep\right|^p \left(\widehat{\mu}^{(p)}\left(\frac{-1}{a+j\vep}\right)\right)^p  $$
\end{lemma}
Observe that the previous formula shows that a sequence $(\mu_k)_k$ in $\mc M^{(p)}(\R)$ converges completely to $\mu\in \mc M^{(p)}(\R)$ when $(\widehat{\mu_k}^{(p)})_k$ converges to $\widehat{\mu}^{(p)}$ uniformly on compacta. 
Given $a\in \R$ and $\vep>0$, let  
\begin{align*}
G_{p}(x,a,\vep):=&\frac12+\frac{1}{2(p!)   \vep^p} \sum_{j=0}^p(-1)^{j+1} \binom{p}{j} |x-(a+j\vep)|^p.
\end{align*}
\begin{center}
\begin{figure}[h]
\includegraphics[scale=.4]{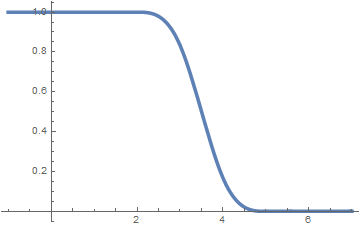}
\caption{$G_3(x,2,1)$}
\end{figure}
\end{center}

\begin{proposition}
$G_p(x,a,\vep) = 1$ if $x\le a$, $G_p(x,a,\vep)(x) = 0$ if $x\ge a+\vep p$ and  $ 0\le G_p (\cdot, a,\vep) \le 1$; consequently, 
\begin{equation}
\mu(]-\infty,a]) \le \int G_p(x,a,\vep) d\mu(x) \le  \mu(]-\infty,a+\vep p]);  
\end{equation} 
\end{proposition}			
\prue
 {\it (a)}: Suppose that $x\le a$; then $|x-(a+j\vep)|^p=(a-x+j\vep )^p$, so 
$$G_p(x,a,\vep)=\frac12+\frac12\frac{1}{p!  \vep^p} \sum_{k=0}^p \binom{p}{k} (a-x)^{p-k}(-\vep)^k\left( \sum_{j=0}^p(-1)^{j+1} \binom{p}{j} j^k\right) =1,$$
because $\sum_{j=0}^p(-1)^{j+1} \binom{p}{j} j^k=0$ for $0\le k<p$ and $\sum_{j=0}^p(-1)^{j+1} \binom{p}{j} j^p=p!$, because $p$ is odd.   The estimate for $x \geq a+\vep p$ is similar. Then note that $G_p(x,a,\vep)+G_p(2a+p\vep-x,a,\vep)=1$ so it is enough to check the inequality  
$0\le G_p (x, a,\vep) \le 1$ for $a \leq x \leq a+\vep p/2$. 
For $1 \leq r < p/2$ we consider the $2r$-derivative of $G_p(x,a,\vep)$ with respect to $x$,
$$G_p^{(2r)}(x,a,\vep)=\frac12\frac{1}{(p-2r)!  \vep^p} \sum_{j=0}^p(-1)^{j+1} \binom{p}{j} |x-(a+j\vep)|^{p-2r}.$$
and note that it assumes the value $0$ in $x=a$ and in  $x=a+\vep p/2$, and that
when $p=2r+1$ then it is affine on each interval $[a+j\vep,a+(j+1)\vep], 0 \leq j \leq p-1$ and therefore admits at most $\frac{p+1}{2}$ zeros on $[a,a+\vep p/2]$ (an easy computation shows that actually exactly $\frac{p+1}{2}$ zeros are obtained).
By standard analysis we also note that if for $r \geq 1$, $G_p^{(2r)}(x,a,\vep)$ admits at most $m$ zeros (including $a$ and $a+\vep p/2$) in $[a,a+\vep p/2]$,  then $G_p^{(2r-1)}(x,a,\vep)$ admits at most $m-1$ zeros (including $a$), and then again $G_p^{(2r-2)}(x,a,\vep)$ admits at most $m-1$ zeros in the same interval. From these two facts we deduce that
 $G_p^{\prime}(x,a,\vep)$ admits at most one zero in $[a,a+\vep p/2]$, which must be $a$, and therefore $G_p(x,a,\vep)$ is monotonous (actually decreasing) there and in particular takes values between $\frac 12$ and $1$. This concludes the proof. 
\fprue
Now Lemma \ref{inversion} follows from this Proposition, simply using that for $c\neq 0$ one has that  
$$\int |x+ c|^pd\mu(x)= |c|^p \left(\widehat{\mu}^{(p)}\left(\frac{1}{c}\right) \right)^p.$$

\subsubsection{Metrics on $\mc M(\mbb F^n)$}
We give now a quantitative version of the Continuity Theorem \ref{continuity_theorem}. 

\begin{definition}
 Given $\mu,\nu\in \mc M^{(p)}(\mbb F^n)$, we define 
 $$\partial_p(\mu,\nu):=\inf\conj{K\ge 1}{\frac{1}{K}\widehat{\nu}^{(p)}(a) \le  \widehat{\mu}^{(p)}(a) \le K\widehat{\nu}^{(p)}(a) \text{ for all $a\in \mbb F^n$}}.$$
\end{definition}  
Such $K$ always exists because the  basic sequence  $(\mathbbm 1_{\mbb F^n}, z_0,\dots,z_{n-1}\rangle)$ in  $L_p(\mu)$ and  in $\con L_p(\nu)$ must be equivalent. The function $\partial_p$ is a multiplicative pseudometric, that because of the   uniqueness of the $p$-characteristics, it is a multiplicative metric.  It is easy to see that $\partial_p$ defines the compact convergence on $\mc M^{(p)}(\mbb F^n)$.

We give now a quantitative version of the continuity theorem of the $p$-characteristics. 
Given $\theta:\N\to \N$, let 
$$\mc T_\theta:=\conj{\mu\in \mc M(\mbb F^{n})}{\nrm{\mu}\le \theta(0) \text{ and } \mu(\mbb F^n\setminus B(0,\theta(m+1)))\le \frac1{2^{m}}  \text{ for every $m\in \N$}}.$$
 Notice that $\bigcup_{\theta\in ^{\N}\N}\mc T_\theta=\mc M(\mbb F^n)$, where $^{\N}\N$ is the collection of mappings $\theta: \N\to \N$. As a consequence of the continuity theorem, we obtain the following.

\begin{corollary}[Quantitative Continuity Theorem]\label{quantitative_continuity} 
Let $\theta:\N\to \N$ and let $\vep>0$. 
\begin{enumerate}[(a)]
\item Given $0\le \al\le p$ there is some $\de>0$ such that if $|z|^pd\mu \in   \mc M^{(p)}(\mbb F^n)\cap \mc T_\theta$, $\nrm{\mu}\le \theta(0)$ and $\nu\in \mc M^{(p)}(\mbb F^n)$    is such that $\partial_p(\mu,\nu)<1+\de$, then $d_{\mc{LP}}(|z|^\al d\mu,|z|^\al d\nu)<\vep$.  

\item  There is some $\de>0$ such that if $\mu,\nu\in \mc M^{(p)}(\mbb F^n)$  are such that $\nrm{\mu},\nrm{\nu}\le \theta(0)$,  $|z|^pd\mu,|z|^pd\nu\in \mc T_\theta$ and     $d_{\mc{ LP}}(\mu,\nu)<\de$, then $\partial_p( \mu, \nu)<1+\vep$.  
\end{enumerate}	
\end{corollary}	
\prue
{\it (a)}:  Suppose otherwise; we can find  sequences $(|z|^pd\mu_k)_k$ in  $\mc M^{(p)}(\mbb F^n)\cap \mc T_\theta$  and with $\sup_k \nrm{\mu_k}\le \theta(0)$, and  $(\nu_k)_k$ in  $\mc M^{(p)}(\mbb F^n)$ such that $\lim_k \partial_p(\mu_k,\nu_k)=1$ and $\inf_k d_{\mc{LP}}(|z|^\al d\mu_k, |z|^\al d\nu_k)\ge \vep$.   Since $(\mu_k)_k$ is bounded in norm and tight, there is a completely convergent subsequence to $\mu$. Without loss of generality, we assume that $(\mu_k)_k$ converges completely to $\mu$, Since $(|z|^pd\mu_k)_k$ is tight, it follows from the implications  $2. \implies 3.$ and $2. \implies 1.$  in Theorem \ref{continuity_theorem} that $\lim_k \partial_p(\mu_k,\mu)=1$ and that $\lim_k d_\mc{LP}(|z|^\al d\mu_k, |z|^\al d\mu)=0$. Resuming, $\lim_k \partial_p(\mu,\nu_k)=1$ and $\inf_k d_\mc{LP}(|z|^\al d\mu,|z|^\al d\nu_k)\ge \vep$, contradicting the implication $3. \implies 1.$  in Theorem \ref{continuity_theorem}.  The proof of {\it (b)} is similar. We leave the details to the reader.
\fprue

 \subsubsection{Approximate full support}  We finish the section by  proving the approximate full support principle in Theorem \ref{kjnnjguytfyyuoipo}. 
 We assume that  $L_p(\mu)$ is separable.
Recall that for a $\mu$-measurable subset $A$, $P_A$ denotes the Boolean projection on $L_p(\mu)$  defined by $P_A(f):= f\cdot \mathbbm{1}_A$, and that  given a subspace $X$  of $L_p(\mu)$, and let $\vep>0$, we say that $f$ has {\em $\vep$-full support in $X$}  if  $\nrm{P_{\{f=0\}}\rest X}\le \vep$. We recall that Theorem   \ref{kjnnjguytfyyuoipo}  states that if $u\in X$ has full support in $X$  then for every $\vep>0$  there is some $\de>0$ such that if $\ga\in \Emb_\de(X, L_p(\mu))$ then   $\ga u$ has $\vep$-full support in $\ga X$. 
We will follow the same strategy as in \cite[Section 3]{Har}.  
We need the following preliminary result. 
\lema\label{oij4t3u9887yuhifgitfr}
Let $D\con L_p(\mu)$ be countable. For every $0<\vep<1$ there are $a_1,a_2\in [0,1]$ such that for every $f_1,f_2\in D$ one has that $a_1f_1+a_2 f_2$ has full support in $\langle f_1,f_2\rangle$ and $1-a_1, a_2\le \vep$.  
  \flema

\prue
Fix $g_1,g_2\in L_p(\mu)$. Let $S:=\supp g_1\cup \supp g_2$ and let  
$g:S\times [0,1]\to \mbb F$, $g(\om,t):= (1-\vep)t g_1(\om)+\vep (1-t)g_2(\om)$.  We consider the product measure $\mu\times \la$ on $S\times [0,1]$, where $\la$ is the Lebesgue measure on $[0,1]$.   Let $A\con S\times [0,1]$ be the set $A:=\conj{(\om,t)\in S\times [0,1]}{ g(\om,t)=0}$. Notice that given $\om\in S$ we have that $\la(A_\om)=0$, where $A_\om:=\conj{t\in [0,1]}{(\om,t)\in A}$: We have that $t\in A_\om$ if and only if $t((1-\vep)g_1(\om)-\vep g_2(\om))=-\vep g_2(\om)$,  and since $\om\in S$, we obtain that $(1-\vep)g_1(\om)-\vep g_2(\om)\neq 0$, so $A_\om=\{-\vep g_2(\om) ((1-\vep)g_1(\om)-\vep g_2(\om))^{-1}\}$ is a singleton. We define also, given $t\in [0,1]$,  $A^t=\conj{\om \in S}{(\om,t)\in A}$.  Now using Fubini Theorem, 
$$0=\int_{\om\in S} \la(A_\om)d\mu(\om)=             (\mu\times \la)(A)=  \int_{t\in [0,1]} \mu(A^t) d\la(t).$$
So $H_{g_1,g_2}:=\conj{t\in [0,1]}{\mu(A^t)=0}$ has Lebesgue Measure 1 in $[0,1]$.  Notice that for $t\in H_{g_1,g_2}$ we have that $\mu(S\setminus \supp g(\cdot,t))=0$, so $g(\cdot,t)$ has full support in $\langle g_1,g_2\rangle$.  Therefore 
$\bigcap_{\{g_1,g_2\}\con D} H_{g_1,g_2}$ has measure 1, and $a_1=(1-t) \vep$,  $a_2=t(1-\vep) $ for $t$ in that intersection will work.   
\fprue

\lema\label{lkq2hiuherpotjtt}
Let  $X$ be a finite dimensional subspace of $L_p(\mu)$ and suppose that  $u\in X$  has full support in $X$ and it is normalized. Suppose that $T:X\to L_p(\mu)$ is an isomorphic embedding such that $Tu$ has full support in $TX$.  Given $f\in X$ of norm one,  let $\nu:= (f/u)_*( |u|^p d\mu)$ and    $\eta:=(Tf/Tu)_*(|Tu|^p d\mu)$. Then $\nu,\eta\in \mc M^{(p)}(\mbb F)$ and $\partial_p(\nu,\eta) \le  \max\{\nrm{T},\nrm{T^{-1}}\}$. 

\flema
\prue
This follows from the simple observation that $\widehat{\nu}^{(p)}(a)= \nrm{u + f}_{L_p(\mu)}$. 
 \fprue

\prue[{\sc Proof of Theorem \ref{kjnnjguytfyyuoipo}}]
Fix $0<\vep\le 1$, and a  finite dimensional subspace $X\con L_p(\mu)$.  Let $f_0\in X$ with full support. Without loss of generality we assume that $f_0$ is normalized.    Let $(f_1,\dots,f_n)$ be a normalized basis of $X$, and let $D$ be a countable dense subset of the centered ball $B(0,2)$ of $L_p(\mu)$ that contains $f_0,f_1,\dots,f_n$, and let $a_1,a_2\in [0,1]$ be the result of applying Lemma \ref{oij4t3u9887yuhifgitfr} to $D$ and $\vep_0$ such that $(4\vep_0)^p +\vep_0^p \nrm{\mu}+3\vep_0\le \vep$.         For each $1\le j\le n$, let  $\vphi_j:= a_0 f_0 +a_1 f_j$; notice that 
$\mu(\supp f_0\setminus \supp \vphi_j)=0$, so $ f_0/\vphi_j\in L_p( |\vphi_j|^p d\mu)$.  For each $j=1,\dots,n$, let $\mu_j:= (f_0/\vphi_j)_*(|\vphi_j|^pd\mu)$.  Observe that 
$$\mu_j(\{0\})= \int_{\frac{f_0}{\vphi_j}=0} |\vphi_j|^pd\mu=\int_{f_0=0} |\vphi_j|^pd\mu=0,$$
because $f_0$ has full support in $X$.  Let $0<\vep_1\le \vep_0$ be such that $\mu_j(B(0,2\vep_1))<\vep_0$ for every $j=1,\dots,n$, and let $\vep_2\le \vep_1$ be such that for $j=1,\dots,n$,  if $\nu\in \mc M^{(p)}(\mbb F)$ is such that $\partial_p(\mu_j,\nu)\le 1+\vep_2$, then $d_{\mc LP}(\mu_j,\nu)\le \vep_1$.   We claim that $\de:=\vep_2/2$ works. For suppose that $T:X\to L_p(\mu)$ is such that $\nrm{T},\nrm{T^{-1}}\le 1+\de$.   We will show that $\int_{Tf_0=0}|Tf_j|^p\le \vep$: 
Fix $1\le j\le n$. We assume that $f_j$ and $f_0$ are linearly independent, since otherwise $\supp Tf_0=\supp Tf_j$. 
Let   $\ga_0,\ga_j\in D$ and $0<\vep_3\le \vep_2$ be such that
\begin{enumerate}[(a)]
\item $  \vep_3 <\sqrt{\vep_2} \vep_1^{(2p+1)/p}$; 
\item $\int_{|Tf_j|>\frac{1}{(\vep_3)^{1/2}}} |Tf_j|^pd\mu<\vep_0$;
\item  $\nrm{Tf_0-\ga_0}, \nrm{Tf_j-\ga_j}, \nrm{T(\vphi_j) -\psi_j}\le \vep_3$ where $\psi_j=a_0 \ga_0 +a_1\ga_j$;
\item  the linear mapping $U$ defined by $U(f_0):=\ga_0$ and $U(f_j)=\ga_j$ satisfies that $\nrm{U},\nrm{U^{-1}}\le 1+\vep_2$
\end{enumerate}
Since $\psi_j:=a_0 \ga_0+a_1\ga_j$,   has full support in $\langle \ga_0,\ga_j\rangle$, it follows that    $\nu_j:= (\ga_0/\psi_j)_*(|\psi_j|^p d\mu) \in \mc M^{(p)}(\mbb F)$ and $\partial_p(\mu_j,\nu_j)\le 1+\vep_2$, by (d) above. Hence,  by Lemma \ref{lkq2hiuherpotjtt}, $d_{\mc LP}(\mu_j,\nu_j)\le \vep_1$.   It follows that  
\begin{equation}
\int_{|\frac{\ga_0}{\psi_j}|\le \vep_1} |\psi_j|^p d\mu = \nu_j(B(0,\vep_1))\le  \nu_j(B(0,2\vep_1)) +\vep_1\le  2\vep_0.
\end{equation}
Now observe that
\begin{align}
\int_{|\ga_0|\le \vep_1^2} |\psi_j|^p d\mu \le & \int_{|\ga_0|\le \vep_1^2 \& |\psi_j|\le \vep_1} |\psi_j|^p d\mu+  \int_{|\ga_0|\le \vep_1^2 \& |\psi_j|\ge \vep_1} |\psi_j|^p d\mu \le \vep_1^p \nrm{\mu} + 2\vep_0. 
\end{align}
Since  $0\le a_0\le \vep_0$ and $ 1-\vep_0  \le a_1\le 1$, it follows that $\nrm{Tf_j-\psi_j}\le \nrm{Tf_j-\ga_j}+ \nrm{\psi_j -\ga_j}\le \vep_3+\vep_0 (\nrm{\ga_0}+ \nrm{\ga_j})\le \vep_3+ 3\vep_0$. Hence, 
\begin{equation}
\int_{|\ga_0|\le \vep_1^2} |Tf_j|^p d\mu\le ( \vep_3+ 3\vep_0)^p + \vep_1^p \nrm{\mu} + 2\vep_0.
\end{equation}
Now,
\begin{equation}	\label{iiojhiuhuih441}
\int_{Tf_0=0} |Tf_j|^p d\mu \le  ( \vep_3+ 3\vep_0)^p + \vep_1^p \nrm{\mu} + 2\vep_0 + \int_{Tf_0=0 \& |\ga_0|> \vep_1^2} |Tf_j|^p d\mu.
\end{equation}
Since
\begin{equation}
\vep_0^{2p}\mu(Tf_0=0 \& |\ga_0|> \vep_1^2)\le \int_{Tf_0=0 \& |\ga_0|> \vep_1^2 } |\ga_0|^p d\mu \le  \nrm{Tf_0-\ga_0}^p\le  \vep_3^p
\end{equation}
it follows that 
\begin{align}
 \int_{Tf_0=0 \& |\ga_0|> \vep_1^2} |Tf_j|^p d\mu \le &  \int_{Tf_0=0 \& |\ga_0|> \vep_1^2 \& |Tf_j|\le 1/\sqrt{\vep_2}} |Tf_j|^p d\mu + \int_{Tf_0=0 \& |\ga_0|> \vep_1^2 \& |Tf_j|> 1/\sqrt{\vep_2}} |Tf_j|^p d\mu  \nonumber\\
 \le &    \left(\frac{\vep_3}{\sqrt{\vep_2} \vep_1^2}\right)^p +\vep_1 \le 2\vep_1 \label{iiojhiuhuih44}
\end{align}
Combining \eqref{iiojhiuhuih441} and \eqref{iiojhiuhuih44} we obtain that 
\[
\int_{Tf_0=0} |Tf_j|^p d\mu \le  ( \vep_3+ 3\vep_0)^p + \vep_1^p \nrm{\mu} + 2\vep_0 +2\vep_1 \le \vep. \qedhere
\]
\fprue

\section{Approximate Ramsey properties of $L_p$ spaces}\label{kemrkwmekorwew}
The {\em approximate Ramsey property (ARP)} is an extension of the near amalgamation property that is known to characterize the extreme amenability of the isometry group of \auh Banach spaces. This is a particular instance of the  Kechris-Pestov-Todorcevic (KPT)   correspondence (see Proposition   \ref{iuhuih43985445}) for Banach spaces.   We will give a proof of  the \arp of the class $\{\ell_p^n\}_n$, and we will relate it with some approximate Ramsey principles of certain regular partitions.  
Our proof uses a discrete  form of the method of concentration of measure applied to these partitions. 
 
\subsection{Approximate Ramsey properties of classes of finite dimensional spaces} 
 We start by recalling some combinatorial useful concepts and terminology.  Let $(A,d_A)$ be a metric space. Given $r\in \N$, an $r$-coloring of a set $A$ is simply a mapping $c:A\to r=\{0,1,\dots,r-1\}$. A \emph{monochromatic} set of an $r$-coloring $c$ of $A$ is a subset $B$ of $A$ on which $c$ is constant.      We say that $B\con A$ is $\vep$-monochromatic, $\vep\ge 0$, if there is some $ \widehat r \in r$ such that $B\con (c^{-1}\{\widehat r\})_\vep$.  A \emph{continuous coloring} of $A$ is a 1-Lipschitz mapping $c: A\to [0,1]$.

\defi
Let $\mc F$ be a  family of finite dimensional  normed spaces.
\begin{enumerate}[{\it(a)}]
\item $\mc F$  has the \emph{Approximate Ramsey Property \arp} when for every $X$ and $Y$ in $\mc F$  and every $\vep>0$ there exists $Z\in \mc F$ such that  every  continuous coloring $c$ of $\mr{Emb}(X,Z)$    $\vep$-stabilizes on  $\ga\circ \mr{Emb}(X,Y)$ for some $\ga\in \mr{Emb}(Y,Z)$, that is, such that
$$\mr{osc}(c\rest \ga\circ \mr{Emb}(X,Y))=\sup_{\psi,\eta\in \mr{Emb}(X,Y)}|c(\ga\circ \psi)-c(\ga\circ \eta)|<\vep;$$

\item $\mc F$  has the \emph{Approximate Ramsey $\mathit{Property}^+$ \arpp} when for every $X$ and $\vep>0$ there is $\de>0$ such that  for any $Y$ in $\mc F$  there exists $Z\in \mc F$ such that  every  continuous coloring $c$ of $\mr{Emb}_\de(X,Z)$    $\vep$-stabilizes on  $\ga\circ \mr{Emb}_\de(X,Y)$ for some $\ga\in \mr{Emb}(Y,Z)$;

 \item $\mc F$  has the \emph{Steady Approximate Ramsey $\mathit{Property}^+$} \sarpp with modulus of stability $\varpi: \N\times [0,\infty[\to [0,\infty[$ when for every $X\in \mc F_k$, every $Y\in \mc F$, $\vep>0$ and  $\de\ge 0$   there exists $Z\in \mc F$ such that  every  continuous coloring $c$ of $\mr{Emb}_{\varpi(k,\de)}(X,Z)$    ($\varpi(k,\de)+\vep$)-stabilizes on  $\ga\circ \mr{Emb}_{\de}(X,Y)$ for some $\ga\in \mr{Emb}(Y,Z)$. 

It is defined in \cite{BaLALuMbo2} that when $\varpi$ does not depend on the dimension, $\mc F$ is said to have the {\em stable} approximate Ramsey $\mathit{Property}^+$.
\end{enumerate}

\fdefi
Up to now, the following classes are known to have approximate Ramsey properties.
\begin{example}
 The class  $\{\ell_2^n\}_n$  of finite dimensional euclidean spaces has the \sarpp.  First of all, M. Gromov and V. Milman \cite{GrMi} proved that the unitary group $\iso(\ell_2)$ with its strong operator topology is a Lévy group, so it is extremely amenable. Since $\ell_2$ is \uh,  this last fact is equivalent to saying that $\age(\ell_2)$ has the \arp (see Theorem \ref{iuhuih43985445} below). Moreover, $\ell_2$ is \sauhp, hence $\age(\ell_2)$ is an amalgamation class. This implies that $\age(\ell_2)\equiv\{\ell_2^n\}_n$ has the \sarpp (see Proposition    \ref{oiu43rio34hjoi443}). 
\end{example}

\begin{example}
For $1\le p<\infty$, $p\notin 2\N$, the class $\age(L_p(0,1))$ has the \sarpp:  The \arp of $\age(L_p(0,1))$ is a consequence of the fact that $\iso(L_p(0,1))$ is extremely amenable, proved by  T. Giordano and V. Pestov  \cite{GiPe}, and that those $L_p$ spaces are \auh.  Moreover, we proved  that  these  spaces  $L_p(0,1)$ are  \sauhp, so, $\age(L_p(0,1))$ is an amalgamation class. 
\end{example}
	
\begin{example}
For all $1\le p\neq 2<\infty$ the class $\{\ell_p^n\}_n$ has the \sarpp:  We give a direct proof   in Section \ref{kemrkwmekorwew}    of the \arp of $\{\ell_p^n\}_n$. This, and the fact that $\{\ell_p^n\}_n$ has the amalgamation property, proved in Proposition \ref{3l4kirjio34riji43}   gives the desired \sarpp of $\{\ell_p^n\}_n$, and that also gives another proof of the extreme amenability of the isometry group $\iso(L_p(0,1))$.      
\end{example}

\begin{example}
The classes $\{\ell_\infty^n\}_n$, the finite dimensional {\em polyhedral spaces}, and all finite dimensional normed spaces have the \sarpp (proved by D. Bartošová, M. Lupini, B. Mbombo and the second author of this paper, in  \cite{BaLALuMbo2}; see also \cite{BaLALuMbo1}). 
\end{example}

The \arp has the following   reinterpretation in terms of finite colorings. 
\prop	 \cite{BaLALuMbo2}
For a class $\mc F$   of finite dimensional spaces the following are equivalent:
\begin{enumerate}[1)]
\item $\mc F$ has the approximate Ramsey property;
\item $\mc F$ has the discrete approximate Ramsey property, that is, for every $X$ and $Y$ in $\mc F$, every $r\in \N$ and every $\vep>0$ there exists $Z\in \mc F$ such that  every $r$-coloring of $\mr{Emb}(X,Z)$ has a $\vep$-monochromatic set of the form $\ga\circ \mr{Emb}(X,Y)$ for some $\ga\in \mr{Emb}(Y,Z)$. 
 
\end{enumerate}
\fprop
Similar equivalences are true for the \arpp and the \sarpp.	
\prue	 For the sake of completeness we sketch the proof. ${\it 2)} \implies {\it 1)}$: Given any continuous coloring $c:\Emb(X,Z)\to [0,1]$ and given $\vep>0$, one can induce the discretization of $c$, $\widehat{c}:\Emb(X,Z)\to r$, where $r$ is chosen such that there is a partition of $[0,1]$ into $r$-many disjoint intervals of diameter less than $\vep$. Since this assignment does not depend on $X$ or $Z$, we can use {\it 2)}  to deduce {\it 1)}.  
The proof of ${\it 1)} \implies  {\it 2)}$ is done by induction on the number of colors $r$: Given an $r+1$-coloring $c:\Emb(X,Z)\to r+1=\{0,1,\dots,r\}$, we can define the induced continuous coloring $\widehat{c}(\ga):=(1/2)d(\ga,c^{-1}(r))$, and then use the inductive hypothesis for $r$ and {\it 1)}.
\fprue

Similarly to the case of discrete structures, approximate Ramsey properties extend the corresponding amalgamation properties. 
\prop
Suppose that $\mc F$ has the \jep.
If $\mc F$ has the \arp, \arpp, \sarpp, then $\mc F$ has the  \nap, is a weak amalgamation class, is an amalgamation class, respectively.
\fprop
\prue
We only prove that \arpp   implies \napp; the other implications are proved in a similar way. 
Suppose that $X\in \mc F$ and $\vep>0$.  We claim that  $0<\de\le 1$  witnessing the \arpp for the initial parameters $X$ and $\mc F$ works. For suppose that $\ga \in \Emb_\de(X,Y)$,    $\eta \in \Emb_\de(X,Z)$. Let $V\in \mc F$, and $f\in \Emb(Y, V)$, and $g\in \Emb(Z,V)$. By the \arpp, we can find $W\in \mc F$ such that for the particular coloring $c:\Emb_\de(X,W)\to [0,1]$, $c(h):=(1/4)d(h, \Emb(V,W)\circ f\circ  \ga)$ we can find $\ro\in \Emb(V,W)$ for which $\osc(c\rest \ro \circ \Emb_\de(X,V))\le \vep$. Observe that $c(\ro \circ f\circ \ga)=0$, so there is $\nu\in \Emb(V,W)$ such that $\nrm{\nu\circ g \circ \eta-\ro\circ f \circ \ga}\le \vep$, as desired.
 \fprue 
\begin{problem}
Does there exist a   Fra\"\i ss\'e class of finite dimensional spaces not having the \arp?
\end{problem}

\begin{proposition}\label{oiu43rio34hjoi443}
Let $\mc F$ be a class of finite dimensional normed spaces. Then, 
\begin{enumerate}[1)]
 \item $\mc F$ has the \arpp if and only if  $\mc F$ has   the \arp and \napp.
  \item $\mc F$ has   the \sarpp if and only if $\mc F$ has  the \arp and it is an amalgamation class.
\end{enumerate}   
\end{proposition}  
\begin{proof}
We use the following, that has to be compared with    Lemma \ref{oi788798434}  and that is proved similarly. 
\clam
$\mc F$ has  \napp  if and only if for every $\vep>0$ and  $X\in \mc F$  there is $\de>0$ such that  for  every $Y,Z\in \mc F$   there is some $V\in \mc F$ and some $I\in \Emb(Y,V)$ such that for every $\ga\in \Emb_{\de}(X,Y)$ and $\eta\in \Emb_\de(X,Z)$ there is $J\in \Emb(Z,V)$ such that $\nrm{I\circ \ga - J \circ \eta}\le \vep$.   \qed
\fclam
Now suppose that $\mc F$ has both the \arp and  \napp.  Fix $\vep>0$ and $X\in \mc F$. We use first the claim to find the corresponding $\de$ for $\vep/3$. Now given $Y\in \mc F$ we use the property of $\de$ to find $V\in \mc F$  and $I\in \Emb(Y,Z)$ such that, in particular, $I \circ \Emb_\de(X,Y)\con (\Emb(X,V))_{\vep/3}$.  Now we use the \arp of $\mc F$ applied to $X,V$ and $\vep$ to find $Z$, that we claim that it works for our purposes: For suppose that $c:\Emb_\de(X,Z)\to [0,1]$ is a continuous coloring.  By the \arp of $\mc F$, there is $J\in \Emb(V,Z)$ such that $\osc(c\rest J \circ \Emb(X,V))\le \vep/3$. Let us see that $\osc(c\rest J\circ I \circ \Emb_\de(X,Y))\le \vep$: For suppose that $\ga,\eta\in \Emb_\de(X,Y)$. There are $\iota,\xi\in \Emb(X,V)$ such that $\nrm{\iota-I\circ \ga}, \nrm{\xi-I\circ \eta}\le \vep/2$. Hence $|c(J\circ I\circ \ga)-c(J\circ I\circ \eta)|\le |c(J\circ I\circ \ga)-c(J\circ \iota)| +|c(J\circ \iota)-c(J\circ \xi)|+|c(J\circ \xi)-c(J\circ I\circ \eta)| \le \vep$. {\it2)} is proved similarly.
\end{proof}

The following  connects the approximate Ramsey property of $\age(E)$ and the extreme amenability of $\iso(E)$.  It is a slight extension of the correspondence given in  \cite{BaLALuMbo2}, and a  particular case of  the metric KPT correspondence for  metric structures (see \cite[Theorem 3.10]{MeTsa}). 
\begin{theorem}[KPT correspondence]\label{iuhuih43985445}
 Suppose that $E$ is \auh. The following   are equivalent:
 \begin{enumerate}[1)]
 \item  The group $\iso(E)$ with the strong operator topology is extremely amenable; that is, every continuous action of $\iso(E)$ on a compact space has a fixed point.
 \item $\age(E)$ has the \arp.

 \end{enumerate}    
 Suppose that $\mc G\preceq \age(E)$ is an amalgamation class such that $E\in [\mc G]$ (See Definition \ref{wlkrjwjrew87755}).  
   Then the previous statements  are equivalent to
 \begin{enumerate}[1)]\addtocounter{enumi}{2}
 \item  $\mc G$ has the \arp (equivalently \sarpp).
   
 \end{enumerate}  
\end{theorem}  
Before giving the proof, we state two interesting consequences.
\cor
 Suppose that $\mc G$ is an hereditary family with the \sarpp.   
 \begin{enumerate}[1)]
 \item  There is a unique  separable Fraïssé Banach space $E$ such that $\age(E)\equiv \bmc{\mc G}$ and such that its isometry group is extremely amenable. Moreover $E=\flim\bmc{\mc G}$.
  \item $\bmc{\mc G}$ is \sarpp.  
 \end{enumerate}   
\fcor  
\begin{proof}
{\it1)} and {\it2)}: Set $\mc H:=\bmc{\mc G}$. We know by Proposition \ref{oiu43rio34hjoi443} {\it 2)}    that $\mc G$ is an amalgamation class, hence, by Proposition  \ref{loi34iooi547889659866} {\it2)},   also $\mc H$ is an amalgamation class. Hence, $\mc H $ is a Fraïssé class. Let $E$ be its Fraïssé limit $\flim\mc H$. The family $\mc G$ fulfills the conditions in the  last part of Theorem \ref{iuhuih43985445}, so $\iso(E)$ is extremely amenable and $\mc H$ has the \arp.  
\end{proof}  

\prue[Proof of Theorem \ref{iuhuih43985445}]
 We prove that {\it 1)} implies {\it 2)}:   We need to introduce some concepts. Given  two metric spaces $(A,d_A)$ and $(B,d_B)$, let $\mr{Lip}(A,B)$ be the collection of 1-Lipschitz mappings from $A$ to  $B$. When $A$ is compact, we endow it with the uniform metric $d(c,d):=\sup_{a\in A} d_B(c(a),d(a))$. Observe that when $B$ is also compact, $(\mr{Lip}(A,B),d)$ is also compact. 
For each $W\in \age(E)$, let $\langle W\rangle :=\conj{X\in \age(E)}{W\con X}$.  Note that $\{\langle W\rangle\}_{W\in \age(E)}$ has the finite intersection property.  Let $\mc U$ be a non-principal ultrafilter on $\age(E)$  containing all $\langle W\rangle$.   Define the ultraproduct 
$$\mr{Lip}_\mc U(\Emb(X,E),[0,1]):= \left(\prod_{X\con Y\in \age(E)} \mr{Lip}(\Emb(X,Y),[0,1]    )  \right)/\sim_\mc U,$$ where $(c_Y)_Y\sim_\mc U (d_Y)_{Y}$ if and only if for every 
$(\ga_j)_{j<n}$ in  $ \Emb(X,E)$, and every $\vep>0$ one has that  $\conj{Y\in \langle \sum_{j<n} \im \ga_j\rangle}{  |\max_{j<n}  |c_Y(\ga_j) -d_Y(\ga_j)|\le \vep }\in \mc U$.  We consider the canonical action of $\iso(E)$ in  $\mr{Lip}(\Emb(X,E), [0,1])$, $(g \cdot c)(\ga):= c (g \circ \ga)$, and the  corresponding  (algebraic) action $\iso(E)\acts\mr{Lip}_\mc U(\Emb(X,E),[0,1])$,  $g\cdot [(c_Y)_Y]_\mc U =[(d_Y)_Y]_\mc U $,  where each  $d_Y(\ga):= c_{g(Y)}(g \circ \ga)$. 
Finally, let  $\Phi:\mr{Lip}(\Emb(X,E), [0,1])\to \mr{Lip}_\mc U(\Emb(X,E),[0,1])$, $\Phi(c) =(c_Y)_{Y}$, where 
$c_Y(\ga):=c( \ga)$. 
\clam
$\Phi$ is a $\iso(E)$-bijection. 
\fclam
\prucl
Suppose that   $\Phi(c)= [(c_Y)_Y]_\mc U$  and $\Phi(g \cdot c)= [(d_Y)_Y]_\mc U$. Then for each $Y$ and $\ga\in \Emb(X,Y)$, $c_Y(\ga)=c(\ga)$ and  $d_Y(\ga)= (g\cdot c)(\ga)=c( g\circ \ga)$, so $g\cdot [(c_Y)_Y]_\mc U= [(d_Y)_Y]_\mc U$.  It is easy to see that $\Phi$ is 1-1. We prove that  $\Phi$ is onto: Fix $[(c_Y)_Y]_\mc U$. We are going to find $c$ such that $\Phi(c)=[(c_Y)_Y]_\mc U$. Fix $\ga\in \Emb(X,E)$. Since $\langle \im \ga\rangle\in \mc U$ and since $(c_Z(\ga))_{W\con Z}$ is a bounded sequence,   the $\mc U$-limit $c(\ga):=\lim_{Y\to \mc U} c_Y(\ga)$ exists. It is ease to see that $c\in \mr{Lip}(\Emb(X,E),[0,1])$ and that $\Phi(c)=[(c_Y)_Y]_\mc U$. 
\fprucl
Now suppose that  $\iso(E)$ is extremely amenable, and let us prove the \arp of $\age(E)$:  Fix $X,Y\in \age(E)$ and $\vep>0$, and let $c: \Emb(X,E)\to [0,1]$ be 1-Lipschitz. Let $d\in \overline{\iso(E) c}$ be such that $g \cdot d=d$ for every $g\in \Iso(E)$, i.e., $d(\ga)=d(g\circ \ga)$ for every $\ga\in \Emb(X,E)$. Since we are assuming that $E$ is \auh, it follows from this that $d$ is a constant function.    Now, since  $\Emb(X,Y)$  is compact, we can find $g\in \iso(E)$ such that $\sup_{\ga\in \Emb(X,Y)}|g\cdot c(\ga) , d(\ga)| \le \vep/2$.  Let us see that $\osc( c\rest g\circ \Emb(X,Y))\le \vep$:    For suppose that $\ga,\eta\in \Emb(X,Y)$;  Then, $|c(g\circ \ga)-c(g\circ \eta)|\le |c(g\circ \ga)-d(\ga)|+ |c(g\circ \eta)-d(\eta)|\le \vep$.  
Since $\Phi$ is a $\iso(E)$-bijection,  given $(c_Z)_{Z}\in \prod_{Z\in \age(E)} \mr{Lip}(\Emb(X,Z),[0,1])$ one has that 
$$\conj{Z\in  \age(E)}{  \text{there is $\ga\in \Emb(Y,Z)$ such that } \osc(c_Z\rest \ga\circ  \Emb(X,Y))\le \vep}\in \mc U,$$    
and consequently, 
$$\conj{Z\in  \age(E)}{  \text{$\forall c\in \mr{Lip}(\Emb(X,Z),[0,1])$   $\exists\ga\in \Emb(Y,Z)$ with} \osc(c \rest \Emb(X,Y))\le \vep}\in \mc U$$    

We prove that {\it 2)} implies  {\it 1)}. We use the following  known characterization of existence of a fixed point.
\clam
Let $G$ be a topological group,    $G\acts K$, and suppose that $p\in K$ has dense orbit.  The following are equivalent.
\begin{enumerate}[a)]
\item  There is a fixed point for the action  $G\acts K$.
\item For every entourage $U$ in $K$ and every finite set $F\con G$ there is some $g\in G$ such that $F \cdot( g \cdot  p)$ is $U$-small, that is, for every $f_0,f_1\in F$ one has that $(f_0 \cdot(g \cdot p), f_1\cdot( g \cdot p))\in U$.
\end{enumerate}

\fclam	
\prucl
We assume that all entourages considered are symmetric.	
For suppose that $q\in K$ is a fixed point; Fix $F\con G$ finite and an entourage $U$; let $V$ be an entourage such that $V\circ V\con U$. Using that $g\cdot : K\to K$ is uniformly continuous, we find an entourage  $W$ such that $gW\con V$ for every $g\in F$. Let $h\in G$ be such that $(h\cdot p,q)\in W$. It follows that $(g h \cdot p, q)=(g h \cdot p, g q)\in V$ for all $g\in F$; hence $(g h \cdot p, g' h\cdot p)\in U$.  Suppose now that { b)} holds, and for every finite set $F$ and entourage $U$ choose $g_{F,U}\in G$ such that $(F\cup\{e\})\cdot (g_{F,U}\cdot p)$ is $U$-small, hence $\{f\cdot (g_{F,U} \cdot  p)\}_{f\in F}\con U[g_{F,U}\cdot p]$ for every  $F$ and $U$. Then any accumulation point $q$ of $\{g_{F,U}\cdot p\}_{F,U}$ is a fixed point.  
\fprucl
 Now  suppose that $\age(E)$ has the \arp. Again, we assume that all entourages considered are symmetric.	 Fix  $\iso(E)\acts K$, $p\in K$, an entourage $U$ and a finite set $F\con \iso(E)$. Set $H:= F ^{-1}$, and  let $V$ be an entourage such that $V\circ V\circ V\circ V\con U$. For each $X\in \age(E)$, we define the pseudometric $d_X$ on $\iso(E)$ by $d_X(g,h):=\nrm{g\rest X-h\rest X}$. The family $\mc D=\{d_X\}_{X\in \age(E)}$ defines the left-uniformity of $\iso(E)$, hence its SOT.  Now the mapping $\iso(E)\to K$, $g\mapsto g^{-1} p$ is uniformly continuous with respect to the family $\mc D$, so there is some $X\in \age(E)$ and some  $\de>0$ such that $d_X(g,h)\le \de$ implies that $(g^{-1}\cdot p, h^{-1} \cdot p)\in V$. Let $Y:=\sum_{g\in H} g X$. Finally, let $\{x_j\}_{j<r}\con K$ be such that $K=\bigcup_{j<r} V[x_j]$, where $V[x]:=\conj{y\in K}{(x,y)\in V}$.  We apply the \arp of $\age(E)$ to $X$, $Y$, $\de/3$ and $r$ to find the corresponding $Z$.  We define the coloring $c:\Emb(X,Z)\to r$ for $\ga\in \Emb(X,Z)$ by choosing $g\in \iso(E)$ such that $\nrm{g\rest X-\ga}\le \de/3$, and then by declaring $c(\ga)=j$ if $j$ is (the first) such that $g^{-1} p \in V[x_j]$. By the Ramsey property of $Z$, we can find $\ro\in \Emb(Y,Z)$ and $j<r$ such that, in particular, for every $\eta\in \Emb(X,Y)$ there is some $g_\eta\in \iso(E)$  such that $(g_\eta)^{-1}\cdot p\in V[x_j]$ and $\nrm{\ro\circ \eta - g_\eta}\le 2\de/3$.   Choose $h\in \iso(E)$ such that $\nrm{h\rest Y-\ro}\le \de/3$. Then, for every $f\in H$, setting $\eta:= f\rest X_n$, then $d_X(h \circ f ,g_\eta)\le \de$, and $g_\eta^{-1} \cdot p\in V[x_j]$. Consequently, $(f_0 \circ h^{-1} \cdot p, f_1 \circ h^{-1} \cdot p) \in U$ for every $f_0,f_1\in F$, as desired.

Suppose that $\mc G \preceq \age(E)$ is an amalgamation class such that  $E\in [\mc G]$, that is, such that the collection of subspaces of elements of $\mc G_E$ is $\La_E$-dense in $\age(E)$.  Then the  family of pseudometrics $\{d_X\}_{X\in\mc G_E}$ defines the SOT of $\iso(E)$.    The strategy of the proof  used to see that {\it 2)} implies {\it 1)}, with the natural modifications, works here. We leave the details to the reader. 

Suppose now that  $\age (E)$ has the \arp. Fix
  $X,Y\in \mc G$, $r\in \N$, and $\vep>0$.  Let $0<\vep_0<\vep/3$ and let  $Z\in \age(E)$ be such that every $r$-coloring of $\Emb(X,Z)$  has an $\vep_0$-monochromatic set of the form $\ga\circ \Emb(X,Y)$ for some $\ga\in \Emb(Y,Z)$. Since $E\in [\mc G]$ and $\mc G$ is an amalgamation class, we can find $0<\de<\vep/(3\vep_0)-1$,   some $V\in \mc G$  and some   $\theta\in \Emb_\de(Z, V)$ such that for every $\ga\in \Emb(X,Z)$ and $\eta\in \Emb(Y,Z)$ there are isometric embeddings $i\in \Emb(X,V)$ and $j\in \Emb(Y,V)$ such that $\nrm{\theta\circ \ga-i},\nrm{\theta\circ \eta-j}\le \vep_0$. We claim that $V$ works for our purposes: For suppose that $c:\Emb(X,V)\to r$. We induce the coloring $\widehat{c}: \Emb(X,Z)\to r$ by choosing for each $\ga\in \Emb(X,Z)$ some $i\in \Emb(X,V)$ such that $\nrm{i- \theta\circ \ga}\le \vep_0$ and declare $\widehat{c}(\ga):=c(i)$. Let $\eta\in \Emb(Y,Z)$  and $\bar r<r$ be such that $\eta\circ \Emb(X,Y)\con (\widehat{c}^{-1}(\bar r) )_{\vep_0}$. Then one can show that   $j\circ \Emb(X,Y)\con (c^{-1}(\bar r) )_{\vep}$, where $j\in \Emb(Y,V)$ is such that $\nrm{j - \theta\circ \eta}\le \vep_0$. 
\fprue
Observe that the previous requirement on $\mc G$ is satisfied when $E=\flim \mc G$. 
Observe also that in the proof we are not assuming that $E$ is necessarily separable, and since for Fraïssé spaces we have that $\age(E)\equiv \age(E_\mc U)$ for any ultrafilter $\mc U$,  we obtain the following.
\cor	
Suppose that $E$ is a   Fraïssé space such that $\age(E)$ has the \arp then $\iso(E_\mc U)$ is extremely amenable for every ultrafilter $\mc U$ on $\N$. In particular, $\iso(\mbb G_\mc U)$ and $\iso(L_p(0,1)_\mc U)$ are extremely amenable non separable groups.  \qed
\fcor

\subsubsection{Multidimensional Borsuk-Ulam}\label{Borsuk-Ulam}
The approximate Ramsey property of the family $\{\ell_p^n\}_n$  has a natural reinterpretation  as a version of a multidimensional {\em Borsuk-Ulam Theorem}. Recall that the equivalent reformulation by  L. Lyusternik and  S. Shnirel’man \cite{LuSch} (see also \cite[Theorem 2.1.1]{Ma})   of the Borsuk-Ulam theorem states that  if the sphere $S_{\ell_2^{n+1}}$ is covered by $n + 1$ open sets, then one of these sets contains a pair $(x, -x)$ of antipodal points.   Given $0< p<\infty$,  let $\mathbf{n}_{\bm{p}}(d,m,r,\vep)$ be  the minimal number $n$ such that for every coloring  
$c:\mr{Emb}(\ell_p^d,\ell_p^n)\to r$
 there is  $\ga\in \mr{Emb}(\ell_p^m,\ell_p^n)\text{ and $i<r$} $ such that 
$\ga\circ\mr{Emb}(\ell_p^d,\ell_p^m)\con (c^{-1}\{i\})_\vep$. The \arp of $\{\ell_p^n\}_n$ is exactly the statement saying that   $\mathbf{n}_{\bm p}(d,m,r,\vep)$ exists. 

Recall that given $\de\ge 0$ and  a subset $A$ of a metric space $(X,d)$, one defines the (closed) $\de$-fattening $A_{\le \de}:=\conj{x\in X}{d(x,A)\le \de}$ and the (open) $\de$-deflating $A_{-\de}:=X\setminus ((X\setminus A)^c)_{\le \de}$. It is easy to see that $(A_{-\de})_{\le \de}\con A$.  We will say that an open covering $\mc U$ of a metric space $X$ is called $\vep$-\emph{fat}  when   $\{V_{-\vep}\}_{U\in \mc U}$ is also a covering on $X$.
 
 The following is a sort of Lebesgue's Number Lemma. 

 \prop
 Suppose that $X$ is a compact  metric space. Then every open covering of $X$ is an $\vep$-open covering for some $\vep>0$.   
 \fprop
 \prue
 Suppose that for some open covering $\mc U$ of $X$ 	such $\vep>0$ does not exists. For each $n\in \N$ we can find a point $x_n$ of $X$ not in $\bigcup_{U\in \mc U} U_{-2^{-n}}$. Since $X$ is compact, there is a subsequence $(x_{n_k})_k$ converging to some $x\in X$. Choose $U\in \mc U$ such that $x\in U$, and also $\de>0$ such that  the ball $B(x,2\de)\con U$. Observe that 
 $B(x,\de)\con (B(x,2\de))_{-\de}\con U_{-\de}  $. Now let $k\in \N$ be such that $\de n_{k}\ge 1$ and such that $d(x_{n_k},x)<\de$. It follows then that 
 $x_{n_k}\in U_{-\de}\con U_{-{2^{-n_k}}},$
 a contradiction. 
 \fprue
   
\defi
Given  $0<p\le \infty$, integers  $d,m$, and $r$ and $\vep>0$, let $\mathbf{n}_{\mathbf{BU,}\bm{p}}(d,m,r,\vep)$  be the minimal integer $n$ such that for every $\vep$-fat open covering $\mc U$ of $\mr{Emb}(\ell_p^d,\ell_p^n)$ with at most $r$ many pieces there exists $\ga\in \mr{Emb}(\ell_p^d,\ell_p^m)$ and some $U\in \mc U$ such that 
\begin{equation*}
\ga \circ \mr{Emb}(\ell_p^d,\ell_p^m)  \con U.
\end{equation*} 
\fdefi
Notice that by assigning to  each  $x\in S_{\ell_p^n} $   the embedding $1\mapsto x$ we can identify topologically $S_{\ell_p^n}$ and $\mr{Emb}(\ell_p^1,\ell_p^n)$. Since $\mr{Emb}(\ell_p^1,\ell_p^1)=\{\pm\mr{Id}\}$,   it follows that $\mathbf{n}_{\mathbf{BU,}\bm{p}}(1,1,r,\vep)$ is the minimal integer such that for every open covering of $S_{\ell_p^n}$ of cardinality $r$ there exists $U\in \mc U$ containing some pair of antipodal vectors.  
 Hence, by Borsuk-Ulam,  $\mathbf{n}_{\mathbf{BU,}\bm{p}}(1,1,r,\vep)\le r$ for every $\vep>0$. In this way, we have  Borsuk-Ulam Theorem is the following statement. 
\begin{theorem}[Lusternik and  Shnirel’man]
$\mathbf{n}_{\mathbf{BU,}\bm{p}}(1,1,r,\vep)\le  r $  for every $\vep>0$. 
\end{theorem}
   We have the following relation, easy to prove.  
 \prop
 For every $0< p\le \infty$ every $d,m$, $r$ and every $0<\vep<\de$ one has that 
 \begin{equation*}
\mbf{n}_{\bm p}(m,d,r,\de)\le  \mbf{n}_{\mbf{BU},\bm p}(d,m,r,\vep)\le \mbf{n}_{\bm p}(m,d,r,\vep). 
 \end{equation*} 
 \fprop
\begin{problem}
Is always  $\mbf{n}_{\mbf{BU},\bm p}(d,m,r,\vep)$ independent of $\vep$?
\end{problem}

\subsection{The \arp of the family $\{\ell_p^n\}_n$}\label{arp_of_ell_p_equisurjections}
We give a direct proof of  the  Approximate Ramsey property of  the family $\{\ell_p^n\}_n$ for $1\le p\neq 2<\infty$  and then of the \arpp of $\age(L_p(0,1))$ for $p\notin 2\N$.

\teor\label{it43784343}
 For $1\le p \neq 2<\infty$ the family $\{\ell_p^n\}_n$  has the \arp.
 \fteor
Its proof is done by relating this \arp with  the \arp of approximate equipartitions, that is shown to be true by an instance of the phenomenon of concentration of measure.  First of all, we try to reduce Theorem \ref{it43784343} to the case of $p=1$.  Notice that it follows from the Banach-Lamperti Theorem on isometries of $L_p(0,1)$, $p\neq 2,\infty$, that all isometry groups for those $p$'s are topologically isomorphic, as was observed and used in \cite{GiPe}. There is a similar fact concerning embeddings using the {\em Mazur mapping}. 
 
 \defi
Given $0< p,q<\infty$, and given $n$, let $\mr{M}_{p,q,n}:\ell_p^n\to \ell_q^n$, be the \emph{Mazur map} defined for $x\in \ell_p^n$ by
\begin{equation}
\mr{M}_{p,q,n}(x):=\sum_{\xi<n}  \mr{sign}(x(\xi)  ) |x(\xi) | ^{\frac{p}q} u_\xi.
\end{equation} 
\fdefi
The following facts are known and easy to prove:
\begin{enumerate}[(a)]
\item  $\mr{M}_{p,q,n}$ preserves the support and the signs of the coordinates; in fact, if $x$ and $y$ are disjointly supported then 
$\mr{M}_{p,q,n}(\la x+\mu y)=\mr{sign}(\la)|\la|^{p/q}\mr{M}_{p,q,n}(x)+\mr{sign}(\la) |\la|^{p/q} \mr{M}_{p,q,n}(y).$
\item   $\mr{M}_{q,p,n}\circ \mr{M}_{p,q,n}=\mr{Id}_{\ell_p^n}$.
\item  $\mr{M}_{p,q,n}$ is a uniform homeomorphism between the corresponding unit spheres (note that $\nrm{\mr{M}_{p,q,n}(x)}_q^q=\nrm{x}_p^p$).  If $\tau_{p,q}$ is  the modulus of uniform continuity of $\mr{M}_{p,q}$,  then
\begin{align*}
\tau_{p,q}(t) \le & 
\begin{cases}
\frac{p}{q}t & \text{ if $p\ge q$}\\
 c_{p,q} t^\frac{p}{q} & \text{ if $p<q$}.
\end{cases}
\end{align*}

\end{enumerate}  
For $1\le p, q<\infty$, $p,q\neq 2$, the Mazur mapping naturally extends to  $\mr{M}_{p,q,n}^d:\mr{Emb}(\ell_p^d,\ell_p^n)\to \mr{Emb}(\ell_q^d,\ell_q^n)$   defined for $\ga\in \mr{Emb}(\ell_p^d,\ell_p^n)$ by  
\begin{equation*}
\mr{M}_{p,q,n}^d({\ga}):=\mr{M}_{p,q,n}\circ \ga \circ\mr{M}_{q,p,d}.
\end{equation*} 
$\mr{M}_{p,q,n}^d({\ga})$ is a linear isometric embedding because $\ga$ sends disjointed supported to disjointed supported, and in fact, 
\begin{equation*}
\mr{M}_{p,q,n}^d({\ga})(\sum_{i<d}a_i u_i)=\sum_{i<d}a_i\mr{M}_{p,q,n}(\ga(u_i)).
\end{equation*} 
\prop
$\mr{M}_{p,q,n}^d$ is an uniform homeomorphism with modulus of continuity $\tau_{p,q}$ with inverse   $\mr{M}_{q,p,n}^d$.
\fprop
\prue
\begin{align*}
\nrm{\mr{M}_{p,q,n}^d(\ga) -\mr{M}_{p,q,n}^d(\eta)}_{q,q}=&\max_{x\in S_{\ell_q^d}}\nrm{\mr{M}_{p,q,n}(\ga(\mr{M}_{q,p,d} (x) ))- \mr{M}_{p,q,n}(\eta(\mr{M}_{q,p,d} (x) ))}_{q}=\\
=& \max_{y\in S_{\ell_p^d}}\nrm{\mr{M}_{p,q,n}(\ga(y ))- \mr{M}_{p,q,n}(\eta(y ))}_{q}\le 
\ \max_{y\in S_{\ell_p^d}}\tau_{p,q}(\nrm{\ga(y )- \eta(y )}_{p})= \\
=& {\om}_{p,q}(\nrm{\ga-\eta}_{p,p}), 
 \end{align*} 
 because $\tau_{p,q}$ is increasing. 
\fprue
Given $p,d,m\in \N$ and $\vep>0$, the integer $\mbf{n}_{\bm p}(d,m,r,\vep)$ is the minimal integer $n$ witnessing the \arp of $\{\ell_p^k\}_k$ for the initial parameters $d,m$ and $\vep$. 
We obtain the following
\prop\label{oio2injrkdfhiwehuiriuwbds}
$\mathbf{n}_{\bm p}(d,m,r,\vep)=\mathbf{n}_{\bm q}(d,m,r,\tau_{p,q}(\vep))$ for every $0< p,q<\infty$, $p,q\neq 2$.
\fprop

\prue
Fix $p,q$ as above, and fix all the parameters. Let $n:=  \mathbf{n_p}(d,m,r,\vep)$, and let $c:\mr{Emb}(\ell_q^d,\ell_q^n)\to r$. Let $\widehat{c}:\mr{Emb}(\ell_p^d,\ell_p^n)\to r$ be the induced coloring 
$\widehat{c}= c \circ \mr{M}_{p,q,n}^d$. Then, let $\ga\in \mr{Emb}(\ell_p^m,\ell_p^n)$, and $i<r$ be such that
$\ga\circ \mr{Emb}(\ell_p^d,\ell_p^m)\con (\widehat{c}^{-1}\{i\})_\vep$. 
 Let $\widehat{\ga}=\mr{M}_{p,q,n}^m(\ga)$. We claim that $\widehat{\ga}$ and $i<r$ do the job. Fix $\sig\in \mr{Emb}(\ell_q^d,\ell_q^m)$. Then $\bar \sig := \mr{M}_{q,p,m}^d(\sig)\in \mr{Emb}(\ell_p^d,\ell_p^m)$, so there is some  $\psi \in \mr{Emb}(\ell_p^d,\ell_p^n)$ with $c(\mr{M}_{p,q,n}^d(\psi))=i$ and $\nrm{\ga\circ \bar \sig - \psi}_{p,p}\le \vep$. Hence $\nrm{\mr{M}_{p,q,n}^d(\ga\circ \bar \sig )-  \mr{M}_{p,q,n}^d(\psi)}_{q,q}\le \tau_{p,q}(\vep)$, and since $\mr{M}_{p,q,n}^d (\ga\circ \bar \sig )=\mr{M}_{p,q,n}^d (\ga\circ \mr{M}_{q,p,m}^d (\sig) )= \mr{M}_{p,q,n}^m(\ga)\circ \sig$, we are done. 
\fprue
So, in order to have the \arp of $\{\ell_p^n\}_n$, $p\neq 2,\infty$, it suffices to prove the \arp of $\{\ell_1^n\}_n$, and this is what we do next.  The proof we give is a byproduct of an extension of the fact that  $\mathbf{n}_1(1,m,r,\vep)$ exists, proved by J. Matou\v{s}ek and V.  R\"{o}dl  \cite{Ma-Ro},   and the existence of the Ramsey number corresponding to unital embeddings, i.e. those embeddings preserving the unit.  The proof of the existence of  $\mathbf{n}_1(1,m,r,\vep)$ was done, as we mentioned before, by   Matou\v{s}ek and   R\"{o}dl using combinatorial methods (spread vectors), and, independently,    by   E. Odell, H. P. Rosenthal and Th. Schlumprecht in \cite{Od-Ro-Schl}, by using  tools of Banach space theory, such as different type of bases.

We introduce some notation. In $\ell_1^n$, we  denote by $\mathbbm 1$ to the sequence $(1/n)\sum_{j<n}^n u_j=(1/n,\dots,1/n)\in \mbb F^n$.   We will denote by $\Emb((\ell_1^d,\mathbbm 1), (\ell_1^n,\mathbbm 1))$ the collection of unital isometric embeddings.  We have the following consequence of the equimeasurability principle of Plotkin and Rudin mentioned above in \S \ref{oi43894589754789}    ( Theorem \ref{oi4jiorjeiwrew44}).
\begin{proposition}
 An isometric embedding $\ga: \ell_1^d\to \ell_1^n$ is unital if and only if $d|n$ and  $\ga(u_j)= (d/n) \sum_{k\in s_j} u_k$ for $j<d$ such that  $\{s_j\}_{j<d}$ is a {\em $d$-equipartition} of $n$, that is,  $\#s_j=\#s_l= n/d$ for every $j,l<d$.   \qed
\end{proposition}

\defi\label{oi32rioj32oi234}
Given $d|n$ let $\mc{EQ}(n,d)$ be the set of \emph{equipartitions} of $n$ with $d$ many pieces.  Given in addition $d|m|n$, and $\mc R\in \mc{EQ}(n,m)$ let $\langle \mc R\rangle_d^{\mr{eq}}$ be the collection of $d$-equipartitions of $n$ coarser than $\mc R$. 

Partitions of $n$ with $d$-many pieces are canonically identified with  surjections $F: n\to d$ such that $\min F^{-1}(i)<\min F^{-1}(j)$ for every $i<j<d$, called {\em rigid surjections}. In this way, $d$-equipartitions are identified with rigid surjections $F: n\to d$ such that $\# F^{-1}(j)= n/d$ for every $j<d$. 
\fdefi

It follows that $d$-dimensional unital subspaces of $\ell_1^n$ that are isometric to $\ell_1^d$ are the of the form  $\langle (d/n) \sum_{k\in s_j} u_k\rangle_{j<d}$ for $d$-equipartitions $\{s_j\}_{j<d}$ of $n$.  This means that the following are equivalent:
\begin{enumerate}[(i)]
\item The class $\{(\ell_1^n,\mathbbm{1})\}_n$ has the {\em Structural Ramsey Property}, that is, for every $d,m$ and every $r$ there is $n$ such that every $r$-coloring of the collection $\binom{(\ell_1^n,\mathbbm{1})}{(\ell_1^d,\mathbbm{1})}$ of unital subspaces of $\ell_1^n$ isometric to $\ell_1^d$ has a monochromatic set of the form $\binom{ X}{(\ell_1^d,\mathbbm{1})}$ for some unital $X\in \binom{(\ell_1^n,\mathbbm{1})}{(\ell_1^m,\mathbbm{1})}$.
\item The collection of equipartitions have the Structural  Ramsey Property, that is,  for every $d,m$ and $r$ there is $n$ such that  every $r$-coloring  of $\mc{EQ}(n,d)$ has a monochromatic set of the form $\langle \mc R\rangle_{d}^\mr{eq}$     for some             $R\in \mc{EQ}(n,m)$.   

\end{enumerate}

It is interesting to compare the previous equivalence between Ramsey properties with what happens in the case of $p=\infty$:  Observe that   $\binom{(\ell_1^n,\mathbbm{1})}{(\ell_1^d,\mathbbm{1})}$  is exactly the collection of  unital {\em sublattices} of $\ell_1^n$ of dimension $d$.  Now  we define  $\binom{(\ell_\infty^n,\mathbbm{1})}{(\ell_\infty^d,\mathbbm{1})}$ in the same way as the collection of $d$-dimensional unital (i.e. containing the unit $\sum_{j<n}u_j$ of $\ell_\infty^n$) sublattices of $\ell_\infty^n$. Then we have the following similar equivalence. 
\begin{enumerate}[(i)]\addtocounter{enumi}{2}
\item The class $\{(\ell_\infty^n,\mathbbm{1})\}_n$ has the {\em Structural} Ramsey Property, that is, for every $d,m$ and every $r$ there is $n$ such that every $r$-coloring of the collection $\binom{(\ell_\infty^n,\mathbbm{1})}{(\ell_\infty^d,\mathbbm{1})}$ of unital subspaces of $\ell_\infty^n$ isometric to $\ell_\infty^d$ has a monochromatic set of the form $\binom{ X}{(\ell_\infty^d,\mathbbm{1})}$ for some unital $X\in \binom{(\ell_\infty^n,\mathbbm{1})}{(\ell_\infty^m,\mathbbm{1})}$.
\item The collection of partitions have the Structural  Ramsey Property, that is,  for every $d,m$ and $r$ there is $n$ such that  every $r$-coloring  of the collection $\mc{E}(n,d)$  of partitions of $n$ with $d$ many pieces has a monochromatic set of the form $ \langle \mc P\rangle$, the collection of $d$-partitions coarser than $\mc P$,     for some             $\mc P\in \mc{E}(n,m)$.   

\end{enumerate}
This latter statement is the well-known  {\em Dual Ramsey Theorem} of Graham  and Rothschild \cite{GrRo}, that was recently used in \cite{BaLALuMbo2}  to prove the \arp of $\{\ell_\infty^n\}_n$.  We have the following open problem posed by A. S. Kechris,  M. Sokić and S. Todorcevic in \cite{KeSoTo}. 
\begin{problem} 
Does the collection of equipartitions have the structural Ramsey property?
\end{problem}
We now present a positive answer for the approximate version of this problem, and that will be used to show the \arp of $\{\ell_1^n\}_n$. First,  endow $\mathcal{EQ}(n,d)$ with the normalized {\em Hamming} metric as follows: We identify each equipartition $\mc P$ with the corresponding rigid surjection $F_\mc P$, and then we define 
$$d_\mr{H}(\mc P,\mc G):=d_\mr{H}(F_\mc P,F_\mc G):=\frac{1}{n} \#[ F_\mc P\neq F_\mc Q]=\frac{1}{n} \#\conj{j<n}{F_\mc P(j)\neq F_\mc Q(j)}$$
We will prove that
\begin{theorem}[Approximate Structural Ramsey Property of Equipartitions]\label{ASR_for_equip}
  For every $d|m$, every $r\in \N$ and every $\vep>0$ there is a multiple $n$ of $m$ such that for every $r$-coloring of $\mathcal{EQ}(n,d)$ has an $\vep$-monochromatic set of the form $\langle \mc R\rangle_d^\mr{eq}$ for some $\mc R\in \mathcal{EQ}(n,m)$. 
\end{theorem}  
To prove this, we introduce the notion of {\em approximate equipartitions and equisurjections} and using a discrete case of the method of concentration of measure we prove the approximate Ramsey result in Theorem \ref{i34rh34uh43uier} below, that easily implies Theorem \ref{ASR_for_equip}. 
\defi
Given two finite sets $S$ and $T$, let  ${}^TS$  be the set of mappings from $T$ to $S$, and let $\mr{Epi}(T,S)$ be the subset of the surjective ones.  Given $\de\ge 0$, let  $\eqp_\de(T,S)$ be the collection of all  \emph{$\de$-equisurjections}  $F:T\to S$; that is, those $T$  such that 
\begin{equation}
\frac{\#T}{\#S}(1-\de)\le \# F^{-1}(s) \le \frac{\#T}{\#S}(1+\de) \text{ for all $s\in S$}.
\end{equation} 
\fdefi
So $\de$-equisurjections are ``up-to $\de$'' {\em equisurjections}, that is, surjections $F:T\to S$ such that $\#F^{-1}(s)$ is always the same. 
Notice that when $\de<1$, we have 
\begin{equation}
\frac{1-\de}{1+\de}\le \frac{\#F^{-1}(s)}{\#F^{-1}(t)}\le \frac{1+\de}{1-\de}
\end{equation} 
for every  $F\in\eqp_\vep(T,S)$  and  $s,t\in S$.  The set $\mr{Equi}_0(T,S)$ will be denoted by $\mr{Equi}(T,S)$ and its elements equisurjections instead of $0$-equisurjections.

Finally, observe also that $\eqp_{\de_0}(S,R)\circ \eqp_{\de_1}(T,S)\con \eqp_{\de}(T,R)$ if $\de$ is such that    $(1-\de)\le (1-\de_0)(1-\de_1) \le  (1+\de_0)(1+\de_1)\le (1+\de)$.    We consider also ${}^TS$ as a metric space endowed with the normalized \emph{Hamming} distance 
\begin{equation}
d_\mr{H}(F,G):=\frac{1}{\#T} \#(\conj{t\in T}{F(t)\neq G(t)}).
\end{equation}
We will prove the following slight  generalization of the \arp of equisurjections. 
\teor[Approximate Ramsey property for $\de$-equisurjections] \label{i34rh34uh43uier} 
Let  $d|m$, $r\in \N$,   $\de\ge 0$ and $\vep>0$. There is  a multiple $n$  of $m$ such that    every $r$-coloring of $\eqp_{\de}(n,d)$ has an $(\de+\vep)$-monochromatic set of the form $\eqp_{\de}(m,d)\circ R$ for some  $R\in \eqp(n,m)$.
\fteor
It follows for example from the approximate equimeasurability principle that if $\ga:\ell_1^d\to \ell_1^n$ is a unital  approximate isometric embedding, then  $\ga(u_j)$ and $\ga(u_k)$ are almost disjointly supported for $j\neq k<d$, so these unital quasi isometric embeddings can be approximated by those linear mappings of the form $u_j\mapsto (1/{\#s_j})\mathbbm 1_{s_j}$ where $(s_j)_{j<d}$ is an approximate equisurjection of $n$.  As a consequence,  the \arp of approximate equipartitions is exactly a reformulation of the \arp of unital approximate isometric embeddings between $\ell_1^n$'s. 
The proof of Theorem \ref{i34rh34uh43uier} will be given later on Subsection \ref{oi43r4354343399}, but before we come back to it, we introduce a combinatorial tool that will be used. 
We recall   the notion of {\em spread vector} introduced by Matou\v{s}ek and  R\"{o}dl to prove the existence of $\mbf{n}_{\bm p}(1,m,r,\vep)$.
\defi
Given a vector $\mathbf{a}=(a_j)_{j<n}\in \R^n$, we say that $v\in c_{00}$ is a \emph{spread} of $\mathbf{a}$ if $v=\sum_{j<n} a_j u_{m_j}$ for some increasing sequence $(m_j)_{j<n}$ of integers. In this case, we write that $v=\mr{Spread}(\mathbf{a}, \{m_j\}_{j<n})$.
\fdefi
\teor[The Spreading Vector Theorem]\cite{Ma-Ro}\label{6267yrtr90rt}
For every $\vep>0$ and $m$ there exists $k$ and a normalized vector $\mathbf a\in S_{\ell_1^k}$ with the following property:  for each $j<m$, let $$x_j=x_j ^{(\vep,m)} := \sum_{l<k} a_l u_{k^2 j +k(l+1)}.$$ Then for every 
$x =\sum_j b_j x_j$ of norm 1 there is $s\con k^2 m$ of cardinality $k$ such that $\nrm{x-\mr{spread}(\mathbf a, s)}_1<\vep$ and such that   $s\con \bigcup_{b_j\neq 0}   [k^2j+k/2, k^2(j+1)+k/2[$. Consequently, for every isometric embedding $\ga: \ell_1^m \to \langle x_j\rangle_{j<m}$ there is an isometric embedding $\eta:\ell_1^d \to \ell_1^{k^2m}$ such that $\nrm{\ga-\eta}_{1,1}\le \vep$ and such that each $\eta(u_j)=\mr{spread}(\mathbf a, s_j)$ for some $s_j$.  
\qed\fteor
This statement and the classical Ramsey theorem easily proves the existence of $\mathbf{n}_1(1,m,r,\vep)$.

\prue[\sc Proof of Theorem \ref{it43784343}]  By Proposition \ref{oio2injrkdfhiwehuiriuwbds}, it suffices to prove the case $p=1$. We use  the Spreading Theorem \ref{6267yrtr90rt} applied to $m$ and $\vep/2$ to find the corresponding $\mathbf{a}=(a_l)_{l<k}\in S_{\ell_1^k}$. 
 Fix $d,m,r\in \N$ and $\vep>0$. 
We use the Approximate Ramsey property for  equipartitions in Theorem  \ref{i34rh34uh43uier} (see the comment after the statement)  applied to  $d=k^2m$, $k^2m$, number of colors $r$, $\de=0$,  and admitted error $\vep/2$ to find the corresponding $n$ divided by $k^2m$, i.e., $n$ has the property that for every $r$-coloring of $\Emb((\ell_1^{k^2m}, \mathbbm 1),(\ell_1^{n}, \mathbbm 1))$ has an $\vep/2$-monochromatic set of the form      
$\ga \circ \iso(\ell_1^{k^2m})$ for some unital isometric embedding $\ga\in \Emb((\ell_1^{k^2m}, \mathbbm 1),(\ell_1^{n}, \mathbbm 1))$. Let us see that this  $n$ works. We fix $c:\Emb(\ell_1^d,\ell_1^n)\to r$.   Let $\tau: \ell_1^m\to \ell_1^{k^2m}$ be the  linear mapping  defined by $\tau(u_j):= x_j^{(\vep/2,m)}$ for every $j<m$, and let $\iota: \ell_1^d\to \ell_1^{m}$ be the canonical isometric embedding $\iota(u_i):=u_i$ for every $i<d$.  Now let 
$\widehat{c}:\Emb((\ell_1^{k^2m},\mathbbm 1),(\ell_1^n,\mathbbm 1))\to r$ be defined  by $\widehat{c}(\ga):= c(\ga \circ \tau\circ \iota)$.  By the \arp, there is some unital $\ro:\ell_1^{k^2m}\to \ell_1^n$  and some $s<r$ such that $\tau \circ \iso(\ell_1^{k^2m})\con (\widehat{c}^{-1}(s))_{\vep/2}$.  
Finally, we claim that $(\ro \circ \tau)\circ \Emb(\ell_1^d,\ell_1^m)\con (c^{-1}(s))_\vep$: For suppose that $\ga\in \Emb(\ell_1^d,\ell_1^m)$. By the property of $(x_j^{(\vep,m)})_{j<m}$   in the spreading vector Theorem,  there are pairwise disjoint subsets $(s_j)_{j<d}$ of $k^2m$, each  of cardinality $k$, such that $\nrm{\tau(\ga(u_j))- \mr{spread}(\mathbf a, s_j)}_1<\vep/2$ for every $j<d$. Let now $\theta\in \iso(\ell_1^{k^2m})$ be such that $\theta(x_j^{(\vep,m)})=\mr{spread}(\mathbf a, s_j) $ for every $j<d$. Note that is possible because   $x_j^{(\vep,m)}=\mr{spread}(\mathbf a, \{k^2j+k(l+1)\}_{l<k} )$ and $(s_j)_{j<d}$ is a pairwise disjoint sequence.  It follows that 
$$\nrm{\theta\circ \tau \circ \iota -\tau \circ \ga}_{1,1} =\max_{j<d} \nrm{\theta(x_j^{(\vep,m)})- \tau(\ga(u_j))}_1\le  
\max_{j<d} \nrm{\mr{spread}(\mathbf a, s_j) - \tau(\ga(u_j))}_1\le\frac{\vep}2,$$
hence $\nrm{\ro \circ \theta\circ \tau\circ \iota -\ro \circ \tau\circ \ga}_{1,1}\le \vep/2$. Let $\psi:\ell_1^{k^2m}\to \ell_1^n$ be unital such that 
$\widehat{c}(\psi)=s$ and such that $\nrm{\psi - \ro\circ \theta}_{1,1}\le \vep/2$. It follows that $c(\psi \circ \tau\circ \iota)=s$ and 
\[ \nrm{\psi \circ \tau\circ \iota - \ro \circ \tau \circ \ga}\le \nrm{\psi \circ \tau\circ \iota - \ro \circ \theta \circ \tau \circ \iota}+
 \nrm{ \ro \circ \theta \circ \tau \circ \iota -\ro \circ \tau \circ \ga}\le \vep.\qedhere\]
\fprue

\subsubsection{The proof of the \arp for approximate equisurjections}   \label{oi43r4354343399}
The goal now is to give a proof of the \arp of approximate equisurjections   in Theorem \ref{i34rh34uh43uier}. This statement is a consequence of the concentration of measure property that approximate equisurjections have. Let us recall some basic fact and definitions on this. 
\defi
Recall that an  \emph{mm}-space is a  triple $(X,d,\mu)$ where $(X,d)$ is a metric space and $\mu$ is a (probability) measure on $X$. Recall that the \emph{extended concentration function} $\al_X(\de,\vep)$ for $\vep,\de>0$ is defined by 
\begin{equation*}
\al_X(\de,\vep)=1-\inf\conj{\mu((A)_\vep}{\mu(A)\ge \de}.
\end{equation*} 
The \emph{concentration function} $\al_X(\vep)$ is $\al_X(1/2,\vep)$.  A sequence $(X_n)_n$ of mm spaces is called \emph{L\'evy} when 
\begin{equation*}
\al_{X_n}(\vep)\to_n 0 \text{ for every $\vep>0$}.
\end{equation*} 
The sequence $(X_n)_n$ is called \emph{normal L\'evy} when there are $c_1,c_2>0$ such that
\begin{equation}\label{k43ijti4jtrre}
\al_{X_n}(\vep) \le c_1\exp(-c_2 \vep^2 n). 
\end{equation} 
We say that $(X_n)_n$ is \emph{asymptotically normal L\'evy} when there are $c_1,c_2>0$ such that for every $\vep>0$ there exists $n_\vep$ such that for every $n\ge n_\vep$ the inequality in \eqref{k43ijti4jtrre} holds. 
\fdefi
\prop
Suppose that $((X_n,d_n,\mu_n))_n$ is a normal L\'evy sequence, and suppose that $A_n\con X_n$ for every $n$ is such that  $\inf_n \mu_n(A_n)>0$. Then $(A_n, d_n, \mu_n(\cdot|A_n))_n$ is an asymptotic normal L\'evy sequence.  
\fprop
The proof is based on the following simple fact. 
\prop\label{3jrjjerers}
$\al_X(\de,\ro+\vep)\le \al_X(\vep)$ for every $\de$ such that $\al_X(\ro)<\de$.
\fprop
\prue
Fix $A$ such that $\mu(A)\ge \de$. We see first that  $\mu(A_\ro)\ge 1/2$:  Otherwise, $\mu( X\setminus A_\ro)\ge 1/2$, hence $\mu(A)\le 1- \mu((X\setminus A_\ro)_\ro)\le \al_X(\ro)<\de$, and this is impossible.
So it follows that $\mu(A_{\ro+\vep})\ge 1-\al_X(\vep)$.   
\fprue

A standard way to estimate the concentration functions is by studying lengths of filtrations. Recall that given a measure space $(\Om, \Sig)$, a \emph{filtration} is an $\con$-increasing sequence $(\Sig_n)_n$ of $\sig$-subalgebras of $\Sig$. For what we are interested in, we assume that the filtration is finite, starting and finishing with the trivial subalgebras $\{\buit,\Om\}$ and $\Sig$.  
A finite metric space $(\Omega,d)$ is of \emph{length} $l$  (see for example \cite{Mi-Sch}) if there are numbers $a_0,\dots,a_n$ and a {filtration} $(\mc F_{k})_{k=0}^n$ such that 
$l=(\sum_i a_i^2)^{1/2}$ and such that
for every $k$ and every $A,B\in \mc F_k$ such that $A,B\con C\in \mc F_{k-1}$ there is a bijection $\theta: A\to B$ such that $\max_{a\in A}d(a,\theta(a))\le a_k$.    It follows then  (see \cite[Theorem 4.2]{Le}) that 
if $(\Om,d)$ is of length $l$, and if  $\mu_C$ denotes the normalized counting measure on $\Om$,  then, 
\begin{equation}\label{iio58893478678238}
\al_{(\Om,d,\mu_C)}(\vep)\le \exp( -\vep^2/(8l^2)) 
\end{equation} 
It is well known that the \emph{mm space} $\mc X_n:=({}^nS,d_\mr{H},\mu_{\mr C})$, where, ${}^nS$ is the collection of functions from $n$ to $S$, $d_\mr{H}$ is the normalized Hamming distance and $\mu_\mr{C}$ is the normalized counting measure,  is of \emph{length} $1/\sqrt{n}$, so it follows from  the inequality \eqref{iio58893478678238} that    
\begin{equation}\label{lij34jrij4ieee323223}
\al_{\mc X_n}(\vep)\le \exp(-\frac{\vep^2 n}{8}).
\end{equation}

The following is a well-known result that follows from the weak law of large numbers (see below). We give an alternative proof using concentration  of measure.
\prop\label{l3jrijjeferweew}
Suppose that $\de>0$ and $\#S\ge 2$. Then there exists some $n_\de$ such that for every $n\ge n_\de$ one has that 
\begin{equation}\label{kjnklmerkiore89895445}
\mu(\eqp_{\de} (n,S))\ge 1- \exp(-\frac{\de^2}{9 (\#S(\#S-1))^2}n ). 
\end{equation} 
\fprop
\prue
For each $s\in S$, let     $A_s:=\conj{F\in {}^nS}{\#F^{-1}(s)\ge {n}/{\#S}}$.  Notice that $ \mu_\mr{C}(A_s)\ge 1/\#S$:  Observe that $\mu_\mr C(A_s)=\mu_\mr C(A_t)$ because the transposition $\pi$ sending $s$ to $t$  transforms  $A_s$ into $A_t$ in a measure preserving way. Since ${}^nS=\bigcup_{s\in S}A_s$, we obtain the desired bound. Now we see that 
 \begin{equation*}\label{lkdslkjk438438rttg}
 \bigcap_{s\in S} (A_s )_\de \con \eqp_{\de\#S(\#S-1)}(n,S):
 \end{equation*} 
Given $F\in   (A_s)_\de $, let $G\in A_s$ be such that $d_\mr C(F,G)\le \de$. It follows that
\begin{equation}\label{j4ioioe4iu5ui5588}
\#F^{-1}(s)\ge \#G^{-1}(s) -d_\mr C(F,G)\ge   \frac{n}{\#S}(1- \de\#S).
\end{equation}  
So, if $F\in \bigcap_{s\in S} (A_s)_\de$, then, using \eqref{j4ioioe4iu5ui5588},   one has that 
\begin{align} \label{j4ioioe4iu5ui5581}
\#F^{-1}(s)=n -\sum_{t\neq s} \#F^{-1}(t) \le &  n -  \frac{n}{\#S}(1- \de\#S)(\#S -1)=\frac{n}{\#S}(1+\#S(\#S-1)). 
\end{align}
Since by hypothesis $\#S\ge 2$, it follows from \eqref{j4ioioe4iu5ui5588} and \eqref{j4ioioe4iu5ui5581} that $F\in \eqp_{\de\#S(\#S-1)}(n,S)$. 
Let $\ga>0$ be such that $8/9<(1-\ga)^2$, and   let $n_\de$ be such that for all $n\ge n_\de$ one has that  
 \begin{equation*}\label{opoigero76889gg}
 \exp(-\frac{\de^2 (1- \ga)^2}{8}n )\le\frac{ \exp(-\frac{\de^2 }{9}n )}{\#S} \text{ and } \exp(-\frac{(\de\ga)^2 }{8}n )< \frac{1}{\#S}.
 \end{equation*} 
Fix  $n\ge n_\de$.  Then, $\al_{\mc X_n}(\de \ga) <\frac{1}{\#S}$. It follows from this,   \eqref{lij34jrij4ieee323223}  and Proposition \ref{3jrjjerers}  that $\al_{\mc X_n}(\frac1{\#S}, \de ) \le \al_{\mc X_n}( \de(1-\ga) )\le\frac{ \exp(-\frac{\de^2 }{9}n )}{\#S}$. Hence, for every $s\in S$, $\mu_\mr C((A_s)_\de)\ge 1-\frac{ \exp(-\frac{\de^2 }{9}n )}{\#S}$,
 so,  we have that 
 $$\mu_\mr C(\eqp_{\de \#S(\#S-1)} (n,S) )\ge  \mu_\mr C(\bigcap_{s\in S} (A_s)_\de)\ge 1- \exp(-\frac{\de^2 }{9}n ).$$ 
Since $\de>0$ is arbitrary, we obtain   \eqref{kjnklmerkiore89895445}.
 \fprue
As a consequence we have the following. 
\cor
Given $\de>0$, the sequence $(\eqp_\de(n,S) ,d_{\mr{H}},\mu_\mr{C})_n$ is an asymptotically normal L\'evy sequence. \qed
\fcor

\prop\label{permut_equip}
Let $T,S$ be two finite sets, $\de,\de'\ge 0$. 
For every  $\phi\in \eqp_\de(T,S)$ and $\psi\in \eqp_{\de'}(T,S)$ there is some permutation $\pi$ of $T$ such that $d_\mr{H}(\psi\circ \pi, \phi)\le (\de+\de')/2$. In particular, if in addition $\#S |\#T$, then $\mr{Equi}_{\de}(T,S)\con (\mr{Equi}(T,S))_{\de/2}$. 
\fprop
\prue
For each $s\in S$, let $A_s:=\phi^{-1}(s)$ and $B_s:=\psi^{-1}(s)$. Let also $S_0:=\conj{s\in S}{\#A_s\ge \#B_s}$ and $S_1:=S\setminus S_0$. By symmetry, without loss of generality we assume that $2\#S_0\le \#S$.    We define $\pi: T\to T$ as follows: For every  $s\in S_0$, let $g_s:B_s\to A_s$ be an injection, and similarly for $s\in S_1$, let $f_s:A_s\to B_s$ be also an injection.  Let $\pi:T\to T$ be any bijection such that $\pi\rest A_s=f_s$ for $s\in S_1$, and such that $\pi\rest g_s(B_s)= g_s^{-1}$ for every $s\in S_0$. It follows that for every $s\in S_1$ and every $t\in A_s$ one has that  $\phi(t)=s=\psi(\pi(t))$,  while for $s\in S_0$ and $t\in g_s(B_s)$, $t=g_s(\bar t)$ one has that  $\phi(t)=\phi(g_s(\bar t))=s= \psi(\bar t)=\psi(\pi(t))$. This means that 
\begin{align*}
d_\mr{H}(\phi, \psi \circ \pi) \le & \frac{1}{\#T} \sum_{s\in S_0} \# ( A_s\setminus g_s(B_s)) = \frac{1}{\#T} \sum_{s\in S_0} (\# ( A_s)-\#(B_s) )\le \\
\le & \#S_0( (1+\de)\frac{1}{\#S} - (1-\de')\frac{1}{ \#S})\le \frac12( \de+\de').  \qedhere
\end{align*}
\fprue

\prop\label{hamming_composition}
Let $\phi_0,\phi_1\in \eqp_{\de_0}(T,S)$ and $\psi_0,\psi_1\in  \eqp_{\de_1}(S,R)$. Then $d_\mr{H}(\psi_0\circ \phi_0,\psi_1\circ \phi_0) \le (1+\de_0)  d_\mr{H}(\psi_0,\psi_1)$ and $d_\mr{H}(\psi_0\circ \phi_0,\psi_0\circ \phi_1) \le d_\mr{H}(\phi_0,\phi_1)$.
\fprop 
\prue
We have that $d_\mr{H}(\psi_0\circ \phi_0,\psi_1\circ \phi_0) = (\#T)^{-1}\sum_{\psi_0(s)\neq \psi_1(s)} \# (\phi_0^{-1}(s) ) \le (\#T)^{-1}\cdot d_\mr{H}(\psi_0,\psi_1) \cdot (\#S)  (1+\de_0) \# T(\#S)^{-1}  $.  The other inequality is easy to check.
   \fprue

The following particular case of Theorem  \ref{i34rh34uh43uier} will be used later to show the general case.  In the next lemma, $\mc S_X$ denotes the group of permutations of a set $X$. 

 \begin{lemma}\label{i34rh34uh43uier4322323}
For every finite set  $X$,  $\de,\vep>0$ and $r\in \N$ there is   $n\in \N$ such that   every $r$-coloring of $\mr{Equi}_\de(n,X)$ has an $\vep$-monochromatic set of the form $\mc S_{X}\circ F$ for some   $F\in \mr{Equi}_\de(n,X)$.
 \end{lemma}
\prue   
The group of permutations $\mc{S}_X$  of the set $X$ acts by bijections on $\mr{Equi}_\de(n,X)$, hence by $\mu_\mr{C}$-preserving transformations. This implies that if $A\con \mr{Equi}_\de(n,X)$ is such that  $\mu_\mr{C}(A)>1-(\#X!)$,  then $\bigcap_{\pi\in \mc{S}_X} \pi \cdot A\neq \buit$.  Let $n\in \N$ be such that $\al_{X_n}(1/r,\vep)<1/\#X!$, where  $X_n:=(\eqp_\de(n,X),d_\mr{H},\mu_\mr C)$.  Then $n$ works: Given a coloring $c:\eqp_\de(n,X)\to r$, let $j<r$ be such that $A=c^{-1}(j)$ has $\mu_\mr C$-measure $\ge 1/r$. Hence, $\mu_\mr C((A)_\vep)> 1-1/m!$. Let  $F\in  \bigcap_{\pi\in S_X} \pi ((A)_\vep)$. Then,  $\pi\cdot F\in (c^{-1}(j))_\vep$ for every $\pi \in \mc S_X$.  
\fprue

\begin{proof}[\sc Proof of Theorem \ref{i34rh34uh43uier}] 
Fix all parameters. We choose  $n\in \N$   given by Lemma  \ref{i34rh34uh43uier4322323} when applied to $m$, $\de,\vep$ and $r$, and such that $m|n$. We claim that such $n$ works. For suppose that $c:\mr{Equi}_\de(n,d)\to r$. Fix $\sig\in \mr{Equi}(m,d)$, and define $\widehat{c}:\mr{Equi}_\de(n,m)\to r$, by $\widehat{c}(\ro):= c(\sig\circ \ro)$.  By the choice of $n$, there is some $F\in \mr{Equi}_\de(n,m)$  and $j<r$ such that $S_m\circ F\con (\widehat{c}^{-1}(j))_\vep$.    We apply Proposition \ref{permut_equip}  to $F$ to find $F_0\in \mr{Equi}(n,m)$ such that $d_\mr{H}(F,F_0)\le \de/2$. We claim that $\mr{Equi}_\de(m,d)\circ F_0\con  (\widehat{c}^{-1}(j))_\vep$. For suppose that $\pi\in \mr{Emb}_\de(m,d)$.  We apply Proposition \ref{permut_equip} to  $\pi\in \mr{Emb}_\de(m,d)$ and $\sig\in \mr{Equi}(m,d)$  to find a permutation $\theta\in \mc S_m$ such that $d_\mr{H}(\sig\circ \theta, \pi)\le \de/2$.  Since  $S_m\circ F\con (\widehat{c}^{-1}(j))_\vep$, we can find $G\in \mr{Equi}_\de( n,m)$ such that $c(\sig\circ G)=\widehat{c}( G)=j$ and such that $d_\mr{H}(G, \theta\circ F)\le \vep$. We use Proposition \ref{hamming_composition} to conclude that 
\begin{align*}
d_\mr{H}(\pi\circ F_0, \sig \circ G)\le &  d_\mr{H}(\pi\circ F_0, \sig \circ \theta\circ F_0)+d_\mr{H}(\sig \circ \theta\circ F_0, \sig \circ G)\le  d_\mr{H}(\pi, \sig \circ \theta)+ d_\mr{H}(\theta\circ F_0, G)\le \\
\le & d_\mr{H}(\pi, \sig \circ \theta)+ d_\mr{H}(F_0, F)+ d_\mr{H}(\theta\circ F, G)\le \de +\vep. \qedhere
\end{align*}  
\end{proof}

 \section{Gurarij M-space}\label{gurarij_m_space}

 We conclude the paper by presenting a Fraïssé space in the category of $M$-spaces.
 Recall that a Banach lattice $X$ is called an $M$-space when $\nrm{\sum_{j<n} x_j}=\max_{j<n}\nrm{x_j}$ for every sequence $(x_j)_{i<n}$ of pairwise disjoint elements of $X$. In particular, for finite dimensional Banach lattices $X$ this   is equivalent to saying that $X$ is lattice isometric to $\ell_\infty^{\dim X}$.   We are going to prove that the class of  finite dimensional $M$-spaces is a Fraïssé class, and we will find the  corresponding Fraïssé $M$-space $\mbb G_\diamond$.  For this we revisit some results of  F. Cabello-Sanchez \cite{CS} on the existence of an 
almost transitive renorming of $C[0,1]$. To extend his results to the setting of extensions of partial isometries between finite dimensional subspaces,
we shall extend the notions of ultrahomogeneity which are specific to normed  lattices that were introduced in Section \ref{oi4iuuuuttyy}.   In this context, suppose that $A$ and $B$  are Banach lattices
 and  suppose that $\de>0$. Let $\ga:A\to B$ be a 1-1 linear mapping.   
\begin{enumerate}[(a)]
 \item $\ga$ is {\em $\de$-disjoint preserving} if $\nrm{|\ga(a_j)| \wedge |\ga(a_k)|} \leq \de$ for every $j \neq k$.
\item $\ga$ is {\em $\de$-positive} if  $ d(\ga(a), B_+)\le \de$ for every positive $a\in A_+$, where $A_+, B_+$ denote the corresponding set of positive elements. 
\item $\ga$ is a {\em $\de$-isometric disjoint-preserving embedding} when $\ga$ is   $\de$-disjoint preserving and $\de$-isometric embedding.
\item $\ga$ is a  {\em  $\de$-isometric lattice embedding} when it is  $\de$-disjoint preserving,    $\de$-positive and $\de$-isometric embedding.
\end{enumerate} 
Some of these definition already appear in \cite{OiTr}
Let $\Emb_\de^\diamond(A,B)$ and $\Emb_\de^\perp(A,B)$ be the collection of     $\de$-isometric lattice embeddings and  $\de$-isometric disjoint preserving embeddings, respectively. 
Let  $\mc E$ be a Banach lattice. We introduced in Definition \ref{oi23io3444} the notion of approximately lattice ultrahomogeneous  \latauh.  
Let $\age^\diamond(\mc E)$ be the collection of finite dimensional sublattices of $\mc E$. Similarly, let $\iso_\diamond(\mc E)$ and $\iso_\perp(\mc E)$  be the group of lattice   and  of disjoint-preserving isometries on $\mc E$, respectively. 
\begin{definition}		
We say that  $\mc E$ is  a {\em Fraïssé Banach  lattice} when there is a modulus of stability $\varpi:\N\times ]0,\infty[\to ]0,\infty]$ such that for every $X\in \age^\diamond_k(\mc E)$,   every $\de\ge 0$ and $\vep>0$ the canonical action $\iso_\diamond (\mc E)\acts \Emb_\de^\diamond(X,\mc E)$ is ($\varpi(k,\de)+\vep$)-transitive for every $k\in \N$ and every  $X\in \age_k^\diamond(\mc E)$. 
\end{definition}

It follows from Proposition \ref{ultracopylp} and Schechtman's Theorem \ref{kdmfksmfkdsw} that $L_p(0,1)$ is a Fraïssé  Banach lattice for every $1\le p<\infty$.   Now we present a new one. In \cite{CS}, F. Cabello-Sanchez defined $\mk X=(\Pi_{n=1}^{+\infty} L_{p_n}(0,1))_{\mathcal U}$, where $p_n \rightarrow +\infty$, and proved
that the non-separable Banach lattice $\mk X$ is an $M$-space that is transitive, meaning that the isometry group acts transitively on the unit sphere of $\eqs$. We prove with similar methods that his results may be improved as follows.

\begin{theorem}
   The lattice $\mk X$ is    a Fraïssé   Banach lattice with modulus $\varpi(k,\de)\le  3\cdot k \cdot \de$. In addition,    for every $X\in \age^\diamond_k(\mk X)$,   every $\de\ge 0$ and $\vep>0$ the canonical action $\iso_\perp (\mk X)\acts \Emb_\de^\diamond(X,\mk X)$ is ($3 \cdot k \cdot \de+\vep$)-transitive for every $k\in \N$ and every  $X\in \age_k^\diamond(\mk X)$.      

\end{theorem}  
The proof readily follows from the following two lemmas. 
\begin{lemma}
 The lattice $\mk X$ is   lattice $\mathit{ultrahomogeneous}$ and    disjointly $\mathit{ultrahomogeneous}$, that is, for every  $X\in \age^\diamond(\mk X)$  the action $\iso_\perp(\mk X)\acts \Emb^\perp(X,\mk X)$ is transitive.   
\end{lemma}  
\prue   
 Let  $F,G$ be finite dimensional sublattices of $\mk X$ and $t$ be an isometry from $F$ onto $G$.
Write $F=\langle f_j\rangle_{j<d}$ and $G=\langle tf_j\rangle_{j<d}$, where $(f_j)_{j<d}$ is a  normalized, positive and pairwise disjoint sequence.  
 The vectors  $f_0,\ldots,f_{d-1}$  may be represented
as $f_j=(f^j_n)_n$, where for each $n$ the $(f_j^n)_{j<d}$'s is a  normalized, positive and pairwise disjoint sequence in   $L_{p_n}(0,1)$ (see in \cite[the proof of Lemma 3.2]{CS}), and the same holds for each $tf_j=(g^j_n)_n$. If we call $t_n$ the isometric map sending $f^j_n$ to $g^j_n$  we know by Lemma \ref{ultracopylp}  on the approximate lattice ultrahomogeneity of $L_{p}$'s that there exists a lattice isometry $T_n$ on $L_{p_n}$ such that
$\|T_{n}\rest F_n-t_n\| \leq 2^{-n}$. We note that $T=(T_n)_n$ is a lattice isometry on $\mk X$ and
that $T(f_j)=(T_n(f_j^n))_n=(t_n(f_j^n))_n=(g^j_n)_n=t(f_j)$, for each $j<d$.
Disjoint ultrahomogeneity follows from a similar proof and the fact that each $L_{p_n}(0,1)$ is approximately disjointly ultrahomogeneous (see Proposition \ref{latt-disj}).
\fprue

It seems to remain open whether $\eqs$ is actually ultrahomogeneous as a Banach space.

\nota  Using the terminology from Banach space theory (see Remark \ref{j34irjw4ierw566e}), it follows   that $\eqs$ is also of lattice    disposition, in the sense that  for any $F \subset G$, where $F$ and $G$ are finite sublattices of $\mk X$, and for any lattice isometric embedding $t$ of $F$ into $M$, there exists a lattice isometric embedding $T$ of $G$ into $\eqs$ such that $T\rest {F}=t$. \fnota
The other lemma which is needed is a form of amalgamation property, in the lattice setting, of the class of finite dimensional $M$-spaces: 
\lema\label{impbis}
Suppose that $A$ and $B$ are two finite dimensional $M$-spaces and suppose that $\de\ge 0$. Then $\Emb_\de^\diamond(A,B)\con (\Emb^\diamond(A,B))_{3\de \dim A}$ and  $\Emb_\de^\perp(A,B)\con (\Emb^\perp(A,B))_{3\de \dim A}$. 
\flema
\prue We write $A=\langle a_j\rangle_{j<m}$, $B=\langle b_j\rangle_{j<n}$ both normalized pairwise disjointly supported, and suppose that $\ga\in \Emb_\de^\diamond(A,B)$.   Let $g:A\to B$ linearly defined by $g(a_j):=\sum_{|b_k^*(\ga(a_j))|>\de} b_k^*(\ga(a_j)) b_k$, where $b_k^*$ is the sequence of functionals dual to $b_k$.  The fact that $\ga$ is an $\de$-isometric lattice embedding implies that   $(g(a_j))_{j<m}$ is pairwise disjointly supported, positive and 
$\|g(a_j)-\ga(a_j)\| \leq \de$ for all $j<m$.  Hence, $|\nrm{g(a_j)}-1|\le 2\de$ for all $j<m$.  Let $\xi:A\to B$ linearly defined by
$\xi(a_j):= g(a_j)/\|g(a_j)\|$.  It is clear that $\xi$ is a lattice isometric embedding. also,  $\|\xi(a_j)-\ga(a_j)\| \leq 3\de$; hence, 
   $\|\ga-\xi\| \leq 3\de\dim A$. 
   \fprue 
 This lemma should be compared with the fact that without lattice constraints one has  $\Emb_\de(\ell_\infty^d, \ell_\infty^n)\con \Emb(\ell_\infty^d,\ell_\infty^n)_\de$.

\prop  The $M$-space $\eqs$ admits a separable sublattice $\mbb G_\diamond$ that is a Fraïssé Banach lattice with modulus of stability $\varpi(k,\de)\le 3\cdot \de\cdot k$, and such that for every finite dimensional sublattice $X$ of $\mbb G_\diamond$, $\de\ge 0$ and $\vep>0$ the canonical action $\iso_\perp(\mbb G_\diamond)\acts \Emb^\perp_\de(X,\mbb G_\diamond)$ is $3\cdot \de \cdot \dim X +\vep$-transitive.  \fprop

\prue 
We will find a separable sublattice $\mbb G_\diamond$ of $\mk X$ and some countable dense set $D\con \bigcup_{n}\Emb^\perp(\ell_\infty^n, \mbb G_\diamond)$ such that
\begin{enumerate}[a)]
\item  for every $\ga,\eta \in D\cap \Emb^\perp(\ell_\infty^n,\mbb G_\diamond)$ there is some $T\in \iso_\perp(\mbb G_\diamond)$ such that $\eta= T\circ \ga$.  
\item $D\cap   ( \bigcup_{n}\Emb^\diamond(\ell_\infty^n, \mbb G_\diamond))$ is dense in $    \bigcup_{n}\Emb^\diamond(\ell_\infty^n, \mbb G_\diamond)$.
\item  for every $\ga,\eta \in D\cap \Emb^\diamond(\ell_\infty^n,\mbb G_\diamond)$ there is some $T\in \iso_\diamond(\mbb G_\diamond)$ such that $\eta= T\circ \ga$.  
\end{enumerate}  
 This, together with Lemma \ref{impbis} gives the Fraïssé property of $\mbb G_\diamond$ with the corresponding modulus. 
Let $Y_0$ be a separable sublattice of $\mk X$ and let  ${ D}_0\con \bigcup_n \Emb^\diamond(\ell_\infty^n,Y_0)$  be countable and dense in $\bigcup_n \Emb^\perp(\ell_\infty^n,Y_0)$ such that $C_0\cap (\bigcup_n \Emb^\diamond(\ell_\infty^n,Y_0))$ is dense in $\bigcup_n \Emb^\diamond(\ell_\infty^n,Y_0)$.

By the lattice and disjoint  ultrahomogeneity of $\mk X$, given  $\ga,\eta\in  D_0$, $\ga,\eta\in \Emb^\perp(\ell_\infty^n,Y_0)$  it is possible to select a global disjoint isometry $T_{\ga,\eta}$ on $\mk X$ such that $\eta=T_{\ga,\eta}\ga$, in a way that if $\ga,\eta$ are positive, then $T_{\ga,\eta}$ is also positive. Let $G_0$ be the countable subgroup of $\iso(\mk X)$ generated by the disjoint isometries $\conj{T_{\ga,\eta}}{ \ga,\eta \in D_0\cap(\bigcup_{n\in \N}\Emb^\perp(\ell_\infty^n, Y_0)}$, and let $Y_1$ be the separable sublattice generated
by the spaces $TY_0$ with $T \in G_0$.  In this way, we can find   $\con$-increasing sequences $(Y_k)_k$, $(D_k)_k$ and $(G_k)_k$ where each $Y_k$ is a separable Banach sublattices   of $\mk X$,    $D_k$  is a countable dense subset  of $\bigcup_n \Emb^\perp(\ell_\infty^n, Y_k)$ such that $D_k$ is also dense in $\bigcup_n \Emb^\diamond(\ell_\infty^n, Y_k)$, and $G_k$ is a countable subgroup of $\iso_\perp(\mk X)$   such that  for every $\ga,\eta\in D_k\cap \Emb^\perp(\ell_\infty^n,Y_k)$ there is  $T\in G_k$ such that $\eta= T\circ \ga$, that is positive if $\ga,\eta$ are so, and such that $Y_{k+1}$ is the  sublattice generated by $\bigcup_{T\in G_k} T Y_k$.  Let $Y:=\bigcup_k Y_k$, and let   $\ga,\eta\in D_k\cap \Emb^\perp(\ell_\infty^n,Y_k)$. Choose $T\in G_k$ such that $\eta=T\circ \ga$, and such that it is positive if $\ga,\eta$ are so. Since $T Y_l\con Y_{l+1}$ for all $l\ge k$, we have that $T\rest Y: Y\to Y$ is an isometric embedding; since $G_k$ is a subgroup, similarly we have that $T^{-1}: Y\to Y$, so $T$ is a surjective isometry on $Y$. Then the closure  $\mbb G^\diamond$   of $Y$ is the desired sublattice.  
\fprue

\nota   $\mbb G^\diamond$ is of steady approximate lattice and disjoint preserving  disposition with modulus $\le 3 \cdot k\cdot \de$; that is, for every $X\con Y$ both in $\age^\perp(\mbb G_\diamond)$, $\de\ge 0$, $\vep>0$, and   $\ga\in \Emb^\perp_\de(X,\mbb G_\diamond)$ there is some $\eta\in \Emb^\diamond(Y,\mbb G_\diamond)$ such that $\nrm{\eta\rest X-\ga}\le 3\cdot \dim X \cdot \de +\vep$, and $\eta$ is can be chosen to be positive when $X$, $Y$ are lattices and $\ga$ is positive as well.
\fnota

Of course since the above construction depends on choices of subspaces and embeddings, the lattice $\mbb G^\diamond$ is not unique, but lattice isometrically $\mbb G^\diamond$ is unique. Concerning its proof,  is worth noticing that the proof for the uniqueness principle for \auh Banach spaces (Theorem \ref{oi43hjiio43437})  does not seem to work directly; the reason is that, in general, given two finite dimensional  sublattices $A$ and $B$ of a Banach lattice $X$, the lattice generated by $A$ and $B$ may  not be finite dimensional.   On the other hand,  the approach in Theorem \ref{k3whiueuiuhiewew}, with the obvious modifications,   shows the next uniqueness statement. We leave the details to the reader.

\begin{theorem}
Suppose that $X$ and $Y$ are two separable Fraïssé Banach lattices. The following are equivalent.
\begin{enumerate}[1)]
\item  $\age^\diamond(X)\equiv \age^\diamond(Y)$.
\item $X$ and $Y$ are lattice isometric. \qed
  
\end{enumerate}  
\end{theorem} 
\cor	
There is a   renorming of $C[0,1]$  that is a Fraïssé  Banach lattice.  
\fcor
\begin{proof}
According to \cite[Theorem 3.4]{CS}, every separable almost transitive infinite dimensional $M$-space is isomorphic to $C[0,1]$ (a consequence of Miljutin's Theorem).  
\end{proof}  

It is worth noting that $\mbb G_\diamond$,  although isomorphic to $C[0,1]$, cannot be  isometric to a $C(K)$ space itself. Indeed separable $C(K)$ spaces are not almost transitive apart from   the trivial case:

\nota If $K$ is a metrizable compact space with $|K| \geq 2$, then the group of  isometries on $C(K)$ is never almost transitive. \fnota

\prue Fix a compact metric space $(K,\mathbbm 1_K)$. First note that  isometries on $C(K)$ act by multiplication of a unimodular continuous function  with $f \mapsto f(\sigma)$ for some homeomorphism of $\sigma$,  so the orbit of $\mathbbm 1_K$ is the set of unimodular functions. So if $a \in K$, the function $x\mapsto f(x)=d_K(x,a)/\|d(x,a)\|_\infty$ is 
a norm $1$ function which is at distance $1$ for the orbit of $\mathbbm 1_K$. 
\fprue

We conclude this part by presenting new extremely amenable groups.
 
\teor The groups $\iso_\perp(\mbb G_\diamond)$ and $\iso_\diamond(\mbb G_\diamond)$  are extremely amenable.
  \fteor
We have the following correspondence.
\teor[KPT correspondence for Banach lattices]
Let $\mc E$ be a separable \auh Banach lattice. Then the following are equivalent
\begin{enumerate}[1)]
\item $\iso_\diamond(\mc E)$ is extremely amenable;
\item $\age^\diamond(\mc E)$ has the \arp.
 
\end{enumerate} 
 
\fteor
Its proof is exactly the obvious modification of that for Banach spaces in Theorem \ref{iuhuih43985445}.  
\teor\label{kjnuihuihuiuiur}  The class $\age_\diamond(\eqs)$ of finite dimensional sublattices with  lattice  or   disjoint preserving  isometric embeddings has the \arp.  
\fteor
In its proof we will use the Dual Ramsey theorem by Graham and Rothschild. Recall that given two finite  linear orderings $\mbf R=(R,<_R)$ and $\mbf S=(S,<_S)$, a mapping $\sig: S\to R$ is a {\em rigid surjection} when  $\sig$ is a surjection such that for every $r_0<_R r_1$ one has that $\min_{<_
S}\sig^{-1}(r_0)<_S \min_{<_
S}\sig^{-1}(r_1)$.  Let $\mr{Epi}(\mbf S,\mbf R)$ be the collection of rigid surjections from $\mbf S$ onto $\mbf R$. 
\begin{theorem}[Dual Ramsey Theorem \cite{GrRo}] 
For every finite linearly ordered sets $\mbf R$ and $\mbf S$, and every $r$ there is some linearly ordered set $\mbf T$ such that every $r$-coloring of $\mr{Epi}(\mbf T,\mbf R)$ has a monochromatic set of the form $\mr{Epi}(\mbf S,\mbf R)\circ \sig$ for some $\sig\in \mr{Epi}(\mbf T,\mbf S)$. 
\end{theorem}
Note that a linear mapping $\ga:\ell_\infty^d\to \ell_\infty^n$ is a disjoint preserving embedding if and only if its dual operator $\ga^*:\ell_1^n\to \ell_1^d$ satisfies that 
$$\{u_k\}_{k<d}\con \{\pm \ga^*(u_j)\}_{j<n}\con [-1,1]\cdot \{u_k\}_{k<d}.$$  
In other words, if $A$ denotes the matrix representing $\ga$ in the unit bases of $\R^d$ and $\R^n$, then the row vectors of $A$ must be of the form $a_k u_k$ for some $|a_k|\le 1$, and for each $k<d$ there must be a row vector with $a_k=\pm 1$.  Let $\mr{Quo}^\perp(\ell_1^n,\ell_1^d)$ be the collection of such surjections $\sig$.   Similarly, $\ga$ is a lattice embedding when 
$$\{u_k\}_{k<d}\con \{ \ga^*(u_j)\}_{j<n}\con [0,1]\cdot \{u_k\}_{k<d}.$$  
Let $\mr{Quo}^\diamond(\ell_1^n,\ell_1^d)$ be the collection of such surjections $\sig$.  Given a 1-1 mapping $f:d\to m$ and  $\theta:=(\theta_k)_{k<d}\in \{-1,1\}^d$, let $\ga_{f,\theta}\in \Emb^\perp(\ell_\infty^d,\ell_\infty^m)$ be linearly defined by $\eta_{f,\theta}(u_k):= \theta_k u_{f(k)}$.  Observe that for every $\sig\in \mr{Quo}^\perp(\ell_\infty^d,\ell_\infty^m)$ there are $f$ and $\theta$ such that $\sig \circ \eta_{f,\theta}=\id_{\ell_\infty^d}$, and if in addition $\sig$ is positive, $\theta=(1)$  and consequently $\ga_{f,(1)}\in \Emb^\diamond(\ell_\infty^d,\ell_\infty^m)$. We choose one of those embedding  and we denote its parameters by $(f_\sig,\theta_\sig)$.  

\prue[{\sc Proof of Theorem \ref{kjnuihuihuiuiur}}]
We prove the \arp of finite dimensional $M$-spaces with respect to disjoint preserving in its dual form; that is, for every $d,m,r\in \N$  and every $\vep>0$ there is some $n$ such that every $r$-coloring of $\mr{Quo}^\perp(\ell_1^n,\ell_1^d)$ has an $\vep$-monochromatic set of the form $ \mr{Quo}^\perp(\ell_1^m,\ell_1^d)\circ \sig$ for some $\sig\in \mr{Quo}^\perp(\ell_1^n,\ell_1^m)$; the corresponding \arp  for lattice embeddings is proved similarly, and we leave the details to the reader.

Fix $d,m,r\in \N$ and $\vep>0$. Let  $e:=\lceil 1/\vep\rceil$, and  let   $\De:= \conj{s (l/e) u_k}{ s\in \{-1,1\},\, 0\le l\le t, \, k<d}$ ordered by $\prec$ such that if $l<l'$, then $s (l/e) u_k\prec s' (l'/e) u_{k' }$    Let also $\La:= \De\times E$, where $E=\{({f,\theta})\, : \, f:d \to m \text{ is 1-1 and }\theta\in \{-1,1\}^d\}$. We linearly order $E$ arbitrarily and then $\La$ by the corresponding  lexicographic ordering.   Let $n$, ordered canonically, be the result of applying the dual Ramsey Theorem for the parameters $\De$, $\La$ and  number of colors $r$. We claim that $n$ works. For suppose that  $c:\mr{Quo}^\perp(\ell_1^n,\ell_1^d)\to r$.  Let $\Phi: \epi(n,\De)\to \mr{Quo}^\perp(\ell_1^n,\ell_1^d)$ be linearly defined for $f\in \epi(n,\De)$ and $j<n$ by $\Phi(f)(u_j):=s \cdot a \cdot u_k$ if $f(j)=(a,s,k)\in A\times \{-1,1\}\times d$. Let $\widehat{c}:=c\circ \Phi$. By the ``Ramsey property'' of the number $n$, there is some $g\in \epi(n,\La)$ such that $\epi(\La,\De)\circ g $ is $\widehat{c}$-monochromatic with constant value $\bar r  <r$.   Let $\sig\in \mr{Quo}^\perp(\ell_1^n,\ell_1^m)$ be linearly defined for $j<n$ by $\sig(u_j):= s\cdot a \cdot \ga_{f,\theta}(u_k)$, if $g(j)=((a,s,k), (f,\theta))\in \La$. We see that $\mr{Quo}(\ell_1^m,\ell_1^d)\circ \sig\con (c^{-1}(\bar r))_\vep$ by establishing the following.
\clam	
For every $\tau\in \mr{Quo}(\ell_1^m,\ell_1^d)$ there is some rigid surjection $h\in \epi(\La,\De)$ such that $\nrm{\tau\circ \sig - \Phi(h\circ g)}\le \vep$. 
\fclam
\prucl
Fix such $\tau$, and $\la:=(s (l/e) u_k,({f,\theta}))\in \La$. If  $\tau(l \cdot u_{f(k)})=0$, then we declare $h(\la)=0$. Suppose that   $\tau(l \cdot u_{f(k)})\neq 0$. There are two cases to consider:
\begin{enumerate}[a)]
\item $(f,\theta)= (f_\tau,\theta_\tau)$.  We let $h(\la)= s(l/e)u_k$.
\item $(f,\theta)\neq  (f_\tau,\theta_\tau)$.   Write $\tau ( \ga_{f,\theta}(s(l/e) u_k))= b u_i\neq 0$, for some $b\in [-1,1]$ and $i<d$.  Let $0\le l\le e$ and $c\in \{-1,1\}$ be such that $ | b-  c l /e |\le \vep$ and be such that $j/e<|b|$, and we set $h(\la):= c (l/e) u_i$

\end{enumerate}  
We show   that $h$ is a rigid surjection by proving  that $\min_{\La} h^{-1}( 0)= (0, \min E)$, and  $\min_{\La} h^{-1}( s (l/e) u_k)= (s(l/e)u_k, ({f_\tau,\theta_\tau}))$.  The fact that $\min_{\La} h^{-1}( 0)= (0,\min E)$ is trivial; as for the other equality, set $\la:= (s (l/e)u_k, ({f_{\tau},\theta_{\tau}}))$.  First of all, $l\neq 0$, and  $\tau(s (l/e) u_{f_\tau(k)})=a \tau ( \ga_{f_\tau,\theta_\tau}(u_k))=a u_k$, where $a= t \theta_k (l/e)$ with $\theta_k$ being the $k^\mr{th}$-coordinate of $\theta_\tau$. So, in particular, $\tau(l u_{f_\tau(k)})\neq 0$, and consequently, by definition, $h( s (l/e) u_k ,({f_{\tau},\theta_\tau}))= s (l/e) u_k$. On the other hand, if $h( \la')=s (l/e)u_k$ with  $\la\neq \la'=(s' (l'/e)u_{k'},({f,\theta}))$, then necessarily $s' (l'/e)u_{k'}\neq s(l/e) u_k$, and consequently we are in case b) of the definition of $h$, so $l/e <\nrm{\tau ( \ga_{f',\theta'}(s'(l'/e) u_{k')})}\le \nrm{s'(l'/e)u_{k'}}= l'/e$, because $\tau$ is a contraction and $\ga_{f',\theta'}$ is an embedding. Hence, by the ordering on $\De$, $s(l/e) u_k \prec s'(l'/e) u_{k'}$ and consequently, $\la<\la'$. 

Finally, let us see that $\nrm{\tau\circ \sig - \Phi(h\circ g)}\le \vep$, that is, $\nrm{\tau(\sig(u_j))- \Phi(h\circ g)(u_j)}\le \vep$ for every $j<n$, because we consider operators from $\ell_1^n$ onto $\ell_1^d$. So, fix $j<n$, and set $\la=(s(l/e)u_k, ({f,\theta}))=g(j)$.  By definition, $\sig(u_j)= \ga_{f,\theta}(s (l/e) u_k)$, while  $\Phi(h\circ g)(u_j) = h(\la)$.  Suppose first that $\tau(\sig(u_j))=0$; then by definition of $h$, $h(\la)=0$, hence $\tau(\sig(u_j))=\Phi(h\circ g)(u_j)=0$. Suppose now that  $\tau(\sig(u_j))\neq 0$. Now suppose that $(f,\theta)=(f_\tau,\theta_\tau)$. It follows that $\tau(\sig(u_j))=\tau\circ   \ga_{f,\theta}(s (l/e) u_k)=s (l/e) u_k$, while $\Phi(h\circ g)(u_j) = h(\la)=s (l/e) u_k=\tau(\sig(u_j))$. Finally, suppose that  $(f,\theta)\neq (f_\tau,\theta_\tau)$. By the choice of $h(\la)$, $\nrm{ \tau(\sig(u_j)) -h(\la)}\le \vep$.
\fprucl
\fprue


\begin{thebibliography}{99}

\bibitem[Als]{Als} D. Alspach, {\em Small into isomorphisms on $L_p$ spaces}, Illinois Journal of Mathematics, 27 (1983), 300-314.
\bibitem[AvCSCaGoMo]{ACCGM} A. Aviles, F. Cabello-Sanchez, J. Castillo, M. Gonzalez, Y. Moreno,
{\em On ultrapowers of Banach spaces of type ${\mathcal L}_\infty$}, Fund. Math. 222 (2013), 195-212.
\bibitem[Ba]{B} S. Banach, {\em Th\'eorie des op\'erations lin\'eaires}, Warsaw, 1932.

\bibitem[BaLALuMbo1]{BaLALuMbo1} D. Bartošová,  J. Lopez-Abad, M. Lupini and  B.  Mbombo,  
{\em The Ramsey property for Banach spaces and Choquet simplices, and applications}. 
C. R. Math. Acad. Sci. Paris 355 (2017), no. 12, 1242--1246.






\bibitem[BaLALuMbo2]{BaLALuMbo2} D. Bartošová, J. Lopez-Abad, M. Lupini, B. Mbombo, {\em The Ramsey property for Banach spaces and Choquet simplices}, to appear in Journal of European Mathematical Society. arXiv:1708.01317. 



\bibitem[BYBeHeUs]{BYBeHeUs}I. Ben Yaacov, A. Berenstein, C. W. Henson, and A. Usvyatsov, {\em Model theory for metric
structures}, Model theory with applications to algebra and analysis. Vol. 2, London Mathematical Society
Lecture Note Series, vol. 350, Cambridge University Press, 2008, pp. 315--427.


\bibitem[BY]{BY} I. Ben Yaacov,  {\em 
Fraïssé limits of metric structures.} 
J. Symb. Log. 80 (2015), no. 1, 100--115. 

\bibitem[BY1]{BeYa0} I. Ben Yaacov,  {\em 
Private communication.} (2018). 




\bibitem[Bi]{Bi}  P.  Billingsley, 
{\em Convergence of probability measures.  }
Second edition. Wiley Series in Probability and Statistics: Probability and Statistics. A Wiley-Interscience Publication. John Wiley \& Sons, Inc., New York, 1999.
\bibitem[CaSa]{CS} F. Cabello-Sanchez, {\em Transitivity of $M$-spaces and Wood's conjecture}, Math. Proceedings of the Cambridge Philosophical Society, 124 (1998), 513-520.



 

\bibitem[CaGaKu]{CaGaKu}F. Cabello Sanchez, J. Garbulińska-W\k{e}grzyn and W. Kubiś, Quasi-Banach 
spaces of almost universal disposition, Journal of Functional Analysis 
267 (2014) 744--771.

\bibitem[Cu]{Cu} R. Cuppens, 
\emph{Decomposition of multivariate probabilities.} 
Probability and Mathematical Statistics, Vol. 29. Academic Press, 1975.
\bibitem[FleJa]{FJ1} R. J. Fleming and J. E. Jamison, \emph{isometries on Banach spaces: function spaces.} Chapman \& Hall/CRC Monographs and Surveys in Pure and Applied Mathematics, 129. Chapman \& Hall/CRC, Boca Raton, FL, 2003.


\bibitem[Fra]{Fra} R. Fraïssé, {\em Sur l’extension aux relations de quelques propriétés des ordres},
Ann. Sci. Ecole Norm. Sup. 71 (1954), 363--388.


\bibitem[GiPe]{GiPe} T. Giordano and V. Pestov,  {\em Some extremely amenable groups related to operator algebras and ergodic theory.} 
J. Inst. Math. Jussieu 6 (2007), no. 2, 279--315. 


\bibitem[GrRo]{GrRo} R. L. Graham and B. L. Rothschild, {\em Ramsey's theorem for n-parameter sets}, Transactions of the
American Mathematical Society 159 (1971) 257--292.


\bibitem[GrJaKa]{GrJaKa}
P. Greim,  J. E. Jamison and A. Kamińska, 
{\em Almost transitivity of some function spaces.}
Math. Proc. Cambridge Philos. Soc. 116 (1994), no. 3, 475--488. 



\bibitem[GrMi]{GrMi} M. Gromov, and V. D. Milman.  A topological application of the isoperimetric inequality. Amer. J. Math. 105 (1983), no. 4, 843--854.
\bibitem[Gu]{Gu} V. I. Gurari\v{\i}, 
{\em Spaces of universal placement, isotropic spaces and a problem of Mazur on rotations of Banach spaces. (Russian) }
Sibirsk. Mat. \v{Z}. 7 (1966) 1002--1013.



\bibitem[Ha]{Ha} P. R. Halmos, {\em Measure Theory.} D. Van Nostrand Company, Inc., New York, N. Y., 1950.

\bibitem[Har]{Har} C. Hardin, 
{\em Isometries on subspaces of $L_p$}. 
Indiana Univ. Math. J. 30 (1981), no. 3, 449--465. 


\bibitem[He]{He} C. W. Henson, 
{\em Nonstandard hulls of Banach spaces.} 
Israel J. Math. 25 (1976), no. 1--2, 108--144.
\bibitem[Ho]{Ho} W. Hodges, {\em Model theory}. Encyclopedia of Mathematics and its Applications, 42. Cambridge University Press, Cambridge, 1993. 

\bibitem[Ja]{Ja} G. J. O. Jameson,  {\em Counting zeros of generalized polynomials: Descartes’ rule of
signs and Laguerre’s extensions}.
 Math. Gazette 90, no. 518 (2006), 223--234.


\bibitem[Kak]{Kak} S. Kakutani. {\em Concrete representation of abstract (L)-spaces and the mean ergodic theorem}. Ann. of Math. (2) 42, (1941). 523--537.


 \bibitem[Kad]{Kad} M. I. Kadets, 
\emph{Remark on the gap between subspaces. }  
Funkcional. Anal. i Priložen. 9 (1975), no. 2, 73–74. 

 \bibitem[KaOs]{KaOs}N. J. Kalton and M. I. Ostrovskii, 
\emph{Distances between Banach spaces.}
Forum Math. 11 (1999), no. 1, 17–48. 

 
 \bibitem[KePeTo]{KPT} A. S. Kechris, V. G.  Pestov, and S. Todorcevic, 
{\em Fraïssé limits, Ramsey theory, and topological dynamics of automorphism groups.} 
Geom. Funct. Anal. {\bf 15} (2005), no. 1, 106--189. 

\bibitem[KeSo]{KeSo} A. S. Kechris and  M.  Sokić,  {\em Dynamical properties of the automorphism groups of the random poset and random distributive lattice.} Fund. Math. 218 (2012), no. 1, 69--94. 

 \bibitem[KeSoTo]{KeSoTo} A. S. Kechris,  M. Sokić and S. Todorcevic,  {\em Ramsey properties of finite
measure algebras and topological dynamics of the group of measure
preserving automorphisms: some results and an open problem.} Foundations
of mathematics, 69–85, Contemp. Math., 690, Amer. Math. Soc., Providence,
RI, 2017.
 
 \bibitem[KoK\"o]{KoKo} A. 
 Koldobsky and H. K\"onig,  
{\em Aspects of the isometric theory of Banach spaces}. Handbook of the geometry of Banach spaces, Vol. I, 899--939, North-Holland, Amsterdam, 2001. 

\bibitem[KuSo]{kubissolecki} W. Kubiś and S. Solecki,  
{\em A proof of uniqueness of the Gurarii space.}
Israel J. Math. 195 (2013), no. 1, 449--456. 

\bibitem[Ku]{Ku}W. Kubiś, Metric-enriched categories and approximate Fraïssé limits 
[arXiv:1210.6506]

 
\bibitem[La]{La} E. Laguerre,  {\em Sur la th\'eorie des equations num\'eriques}.
Journal de Math\'ematiques pures et appliqu\'ees, 3e s\'erie, t. IX; (1883).

\bibitem[Lam]{Lam} J. Lamperti, {\em On the isometries of some function spaces}, Pacific J. Math. 8 (1958),
459--466.


\bibitem[Le]{Le} M. Ledoux. {\em The concentration of measure phenomenon.} Mathematical Surveys and Monographs, 89. American Mathematical Society, Providence, RI, 2001. 
\bibitem[LiTza]{LiTza} J. Lindenstrauss and L. Tzafriri,
{\em Classical Banach spaces. 1: Sequence spaces. 2. Function spaces. Repr. of the 1977 and 1979 ed.}  
Classics in Mathematics. Berlin: Springer-Verlag. xx, 432 p. (1996).
\bibitem[Lup]{Lup} M. Lupini, {\em Fraïssé limits in functional analysis.} Adv. Math. 338 (2018), 93--174.
\bibitem[Lu1]{lusky} W. Lusky, {\em The Gurarij spaces are unique}, Arch. Math. (Basel) 27 (1976) 627–635
\bibitem[Lu2]{luss} W. Lusky, {\em Some Consequences of Rudin's Paper "$L_p$-Isometries and Equimeasurability"}, Indiana University Mathematics Journal 27 (1978), 859--866.

\bibitem[LuSch]{LuSch} L. Lyusternik and S. Shnirel’man. Topological Methods in Variational
Problems (in Russian). Issledowatelski\v{\i} Institut Matematiki
i Mechaniki pri O. M. G. U., Moscow, 1930. 

\bibitem[MaPi]{MaPi} 
B. Maurey and G. Pisier, {\em
Caractérisation d'une classe d'espaces de Banach par des propriétés de séries aléatoires vectorielles.}
C. R. Acad. Sci. Paris Sér. A-B 277 (1973), A687--A690. 

\bibitem[MaRo]{Ma-Ro} J. Matou\v{s}ek and V.  R\"{o}dl.
On Ramsey sets in spheres. 
J. Combin. Theory Ser. A 70 (1995), no. 1, 30--44. 
\bibitem[Ma]{Ma} J. Matou\v{s}ek, 
Using the Borsuk-Ulam theorem. 
Lectures on topological methods in combinatorics and geometry. Written in cooperation with Anders Björner and Günter M. Ziegler. Universitext. Springer-Verlag, Berlin, 2003.


\bibitem[MeTsa]{MeTsa} J. Melleray and  T. Tsankov, {\em 
 Extremely amenable groups via continuous logic}  arXiv:1404.4590v1. (2014).

\bibitem[MiSch]{Mi-Sch} V. D. Milman and G. Schechtman, \emph{Asymptotic theory of finite-dimensional normed spaces.} Lecture Notes in Mathematics, 1200. Springer-Verlag, Berlin, 1986

\bibitem[OdRoSchl]{Od-Ro-Schl} E. Odell, H. P. Rosenthal and Th. Schlumprecht.  On weakly null FDDs in Banach spaces. Israel J. Math. 84 (1993), no. 3, 333--351.

\bibitem[OiTr]{OiTr} T. Oikhberg and P. Tradacete, {\em Almost band preservers}. Positivity 21 (2017), 4, 1393--1423.

\bibitem[PelRol]{PelRol}  A.  Pełczyński  and S. Rolewicz, {\em Best Norms with Respect to Isometry Groups
in Normed Linear Spaces}, Short communication on International Math. Congress
in Stockholm, 104 (1962).

\bibitem[Pe1]{Pe}V. Pestov, {Ramsey-Milman phenomenon, Urysohn metric spaces, and extremely amenable groups.} Israel J. Math. 127 (2002), 317--357.

\bibitem[Pe2]{Pe2} V.  Pestov, {\em Dynamics of infinite-dimensional groups. The Ramsey-Dvoretzky-Milman phenomenon.} Revised edition of Dynamics of infinite-dimensional groups and Ramsey-type phenomena [Inst. Mat. Pura. Apl. (IMPA), Rio de Janeiro, 2005; MR2164572]. University Lecture Series, 40. American Mathematical Society, Providence, RI, 2006.

\bibitem[Plo]{Plo} A. I. Plotkin, \emph{Continuation of $L^p$-isometries}, Journal of Soviet Mathematics 2 (1974), 143--165.
\bibitem[Plo1]{Plo1} A. I. Plotkin, 
\emph{An algebra that is generated by translation operators, and $L^p$-norms.} (Russian) Functional analysis, No. 6: Theory of operators in linear spaces (Russian), pp. 112--121. Ul'janovsk. Gos. Ped. Inst., Ul'yanovsk, 1976.     
\bibitem[Ran]{beata} B. Randrianantoanina, {\em On isometric stability of complemented subspaces of $L_p$}.  
Israel J. Math. 113 (1999), 45--60. 


\bibitem[Rol]{Rol} S. Rolewicz, {\em Metric linear spaces}, Reidel, Dordrecht, 1985.

\bibitem[Ros]{Ro} H. P. Rosenthal, \emph{On the subspaces of $L_p$ $(p > 2)$ spanned by sequences of independent random variables.} Israel J. Math., 8 273-303, 1970.

\bibitem[Ru]{Ru} W. Rudin, \emph{$L^p$-isometries and equimeasurability}, Indiana University Mathematics Journal 25 (1976), 215--228.
\bibitem[Sch]{Sch} G. Schechtman, \emph{Almost isometric $L_p$ subspaces of $L_p(0,1)$}, Journal of the London Mathematical Society (2) 20 (1979), 516--528.
\bibitem[We]{We} A. Weston, \emph{On the uniform classification of $L_p(\mu)$ spaces. }
 Miniconference on Probability and Analysis,
 231--237, Centre for Mathematics and its Applications, Mathematical Sciences Institute, The Australian National University, Canberra AUS, 1992.

\end{thebibliography}
\end{document}